\documentclass[reqno, 11pt]{amsart}
\usepackage{url}
\usepackage{amssymb,amsthm,amsfonts,amstext,amsmath}
\usepackage{pdfsync}  
\usepackage{mathptmx}       
\usepackage{newcent}       
\usepackage{helvet}         
\usepackage{courier}        
\usepackage{graphicx}
\usepackage{enumerate}
\usepackage{hyperref}
\usepackage{tikz-cd}
\usepackage{multicol}
\usepackage{verbatim}
\usepackage{a4wide}
\usepackage{xfrac} 
\numberwithin{equation}{section}
\usepackage{color}
\definecolor{ao}{rgb}{0.0, 0.5, 0.0}

\usepackage[normalem]{ulem} 

\usepackage[active]{srcltx}
\usepackage{pstricks}
\usepackage{float}
\usepackage{mathrsfs}
\usepackage[charter]{mathdesign}
\usepackage{tikz}
\usepackage{bbm}

\newtheorem{theorem}{Theorem}[section]
\newtheorem{lemma}[theorem]{Lemma}
\newtheorem{conjecture}{Conjecture}
\newtheorem{proposition}[theorem]{Proposition}

\newtheorem{remark}[theorem]{Remark}
\newtheorem{definition}{Definition}
\newtheorem*{theosn}{Theorem}
\newtheorem*{propsn}{Proposition}
\newcommand{\mc}[1]{{\mathcal #1}}

\newcommand{\bb}[1]{{\mathbb #1}}

\newcommand{\<}{\langle}
\renewcommand{\>}{\rangle}
\newcommand{\p}{\partial}
\newcommand{\pfrac}[2]{\genfrac{}{}{}{1}{#1}{#2}}

\newcommand{\at}[2]{\genfrac{}{}{0pt}{}{#1}{#2}}

\newcommand{\dl}{\<\!\<}
\newcommand{\dr}{\>\!\>}

\let\oldtocsection=\tocsection
\let\oldtocsubsection=\tocsubsection
\let\oldtocsubsubsection=\tocsubsubsection
\renewcommand{\tocsection}[2]{\hspace{0em}\oldtocsection{#1}{#2}}
\renewcommand{\tocsubsection}[2]{\hspace{1em}\oldtocsubsection{#1}{#2}}
\renewcommand{\tocsubsubsection}[2]{\hspace{2em}\oldtocsubsubsection{#1}{#2}}
\DeclareRobustCommand{\SkipTocEntry}[5]{}

\newcommand{\tclock}[5]{
	\begin{pgflowlevelscope}{\pgftransformscale{#4}}
		\begin{scope}[shift={(#1,#2)}]
			\shadedraw [inner color=#3!7!white, outer color=#3!90!black, 
			line width=0.2pt] (0,0) circle (0.5cm);
			\foreach \x in {6,12,...,360} {\draw[line width=0.2pt] (\x:0.40cm) -- (\x:0.45cm);}
			\foreach \y in {30,60,...,360} {\draw[line width=0.2pt] (\y:0.35cm) -- (\y:0.45cm);}
			{\pgfsetarrowsstart{to}
				\draw[line width=0.4pt] (0:0.29cm) -- (0.02,0);
				\draw[line width=0.4pt]  (90:0.32cm)--(0,0.02);}
			\filldraw[fill=black] (-0.055,0.55) rectangle (0.055,0.6);
			\filldraw[fill=black] (-0.015,0.51) rectangle (0.015,0.55);
			\draw [line width=0.2pt](0,0.61) circle (0.11cm);
			\draw [line width=0.2pt](0,0) circle (0.5cm);
			\draw [line width=0.2pt](0,0) circle (0.02cm);
			\draw [red,thick,domain=30:45] plot ({#5*0.6*cos(\x)}, {#5*0.6*sin(\x)});
			\draw [red,thick,domain=20:55] plot ({#5*0.65*cos(\x)}, {#5*0.65*sin(\x)});
			\draw [red,thick,domain=10:65] plot ({#5*0.7*cos(\x)}, {#5*0.7*sin(\x)});
			\draw [red,thick,domain=135:150] plot ({#5*0.6*cos(\x)}, {#5*0.6*sin(\x)});
			\draw [red,thick,domain=125:160] plot ({#5*0.65*cos(\x)}, {#5*0.65*sin(\x)});
			\draw [red,thick,domain=170:115] plot ({#5*0.7*cos(\x)}, {#5*0.7*sin(\x)});
		\end{scope}
	\end{pgflowlevelscope}
}

\newcommand\ve{\varepsilon}

\keywords{Porous medium model,  Porous medium equation, Robin boundary conditions.}

 \date{}

\begin{document}
	
	\title[Energy estimates and convergence of weak solutions of the  porous medium equation]{Energy estimates and convergence  of \\weak solutions of the porous medium equation}

	\author{R. de Paula}
	\address{Center for Mathematical Analysis,  Geometry and Dynamical Systems, Instituto Superior T\'ecnico, Universidade de Lisboa, Av. Rovisco Pais, 1049-001 Lisboa, Portugal.}
	\curraddr{}
	\email{renato.paula@tecnico.ulisboa.pt}
	\thanks{}
	
	\author{P. Gon\c calves}
	\address{Center for Mathematical Analysis,  Geometry and Dynamical Systems, Instituto Superior T\'ecnico, Universidade de Lisboa, Av. Rovisco Pais, 1049-001 Lisboa, Portugal.}
	\curraddr{}
	\email{pgoncalves@tecnico.ulisboa.pt}
	\thanks{}
	
	\author{A. Neumann}
	\address{UFRGS, Instituto de Matem\'atica e Estat\'istica, Campus do Vale, Av. Bento Gon\c calves, 9500. CEP 91509-900, Porto Alegre, Brasil}
	\curraddr{}
	\email{aneumann@mat.ufrgs.br}
	\thanks{}

	\subjclass[2010]{60K35, 26A24, 35K55}

	\maketitle

	\begin{abstract}
We study the convergence of the weak solution of the porous medium equation with a type of Robin boundary conditions, by tuning a parameter either to zero or to infinity. The convergence is in the strong sense, with respect to the $L^2$-norm, and the limiting function solves the same equation with Neumann (resp. Dirichlet) boundary conditions when the parameter is taken to zero  (resp. infinity). Our approach is to consider an underlying microscopic dynamics whose space-time evolution of the density is ruled by the solution of those equations and from this, we derive sufficiently strong energy estimates which are the keystone to the proof of our convergence result. 
	\end{abstract}
	
	\section{Introduction}

	In this article we are focused on the porous medium equation (PME) with  a type of Robin boundary conditions given by 
	\begin{equation}\label{eq:Robin}
\begin{cases}
&\partial_{t}\rho_{t}^\kappa(u)= \Delta\, ({\rho_t}^\kappa)^m(u), \quad (t,u) \in (0,T]\times(0,1),\\
&\partial_{u} ({\rho_t}^\kappa)^m(0)=\kappa(\rho_{t}^\kappa(0) -\alpha),\quad t \in (0,T], \\
&\partial_{u} ({\rho_t}^\kappa)^m(1)=\kappa(\beta -\rho_{t}^\kappa(1)),\quad t \in (0,T], \\
&{ \rho}_{0}^\kappa(u)= g(u), \quad u\in[0,1], 
\end{cases}
\end{equation}
	where $m\in \mathbb N$, $\kappa> 0$, $\alpha,\beta\in (0,1)$ and $T>0$. When $m=1$ we recover the heat equation which is a linear parabolic equation. When $m\neq 1$ the equation becomes non-linear and its study is more  intricate with respect to the case $m=1$.  Observe that the solution of the equation above depends on $\alpha,\beta, m$ and $\kappa$, but we only index the solutions on $\kappa$ because $\alpha,\beta$, and $m$  are fixed parameters along the article. Our goal is to analyze the limit, with respect to the norm of $L^2([0,T]\times [0,1])$, of the weak solution $\rho^\kappa$ of  \eqref{eq:Robin} as we tune $\kappa$ to $0$ or to $\infty$. Our notion of weak solution of \eqref{eq:Robin} can be seen in Definition \ref{Def. Robin}.  We observe that this weak solution is unique and this was proved in Section 7.2 of \cite{BPGN} for the case $m=2$, but in the general case, the proof is analogous. In fact, one can just repeat the proof of  \cite{BPGN} and  consider  a function $u^m$ instead of a function $\beta^{-1}(u)$, according to the notation of \cite{filo}, and, as in \cite{BPGN}, use the boundedness of the solutions to 
overcome the fact that 	$u^m$ is not Lipschitz, which is an important hypothesis for the proof given in \cite{filo}.  In this article, we will prove that the weak solution  of \eqref{eq:Robin}, $\rho^\kappa$, converges in $L^2$, when we take $\kappa\to 0$,  to the unique weak solution of the porous medium equation with Neumann boundary conditions given by
		\begin{equation}\label{eq:Neumann}
\begin{cases}
&\partial_{t}\rho_{t}(u)= \Delta\,  (\rho_{t})^m (u), \quad (t,u) \in (0,T]\times(0,1),\\
&\partial_{u} (\rho_{t})^m (0)=\partial_{u} (\rho_{t})^m (1)=0,\quad t \in (0,T], \\
&{ \rho}_{0}(u)= g(u), \quad u\in[0,1], 
\end{cases}
\end{equation}
or, when $\kappa\to \infty$, to the unique weak solution of the porous medium equation with Dirichlet boundary conditions given by
		\begin{equation}\label{eq:Dirichlet}
\begin{cases}
&\partial_{t}\rho_{t}(u)=\Delta\,  ({\rho_t})^m(u), \quad (t,u) \in (0,T]\times(0,1),\\
&{ \rho} _{t}(0)=\alpha, \quad { \rho}_{t}(1)=\beta,\quad t \in (0,T], \\
&{ \rho}_{0}(u)= g(u), \quad u\in[0,1],
\end{cases}
\end{equation}
The notions of weak solutions of the two equations above are given in Definitions  
 \ref{Def.Neumann} and \ref{Def. Dirichlet}, respectively.  The choice of this nomenclature for all the boundary conditions above is based on \cite{vazquez_book}, although we know that according to \cite{GA} it would be more appropriate to use free (resp. fixed) boundary conditions for the  Neumann  (resp. Dirichlet) boundary conditions. For the named Robin boundary conditions, one could think of an elastic or partially absorbed boundary and for that reason, one should refer to it as an elastic boundary condition. 

We observe that our result is very intuitive as one can see, naively, by looking at the expressions $\partial_{u} ({\rho_t}^\kappa)^m(0)=\kappa(\rho^\kappa_{t}(0) -\alpha)$ and $\partial_{u} ({\rho_t}^\kappa)^m(1)=\kappa(\beta -\rho^\kappa_{t}(1))$, when taking $\kappa\to0$ one can guess that Neumann boundary conditions are recovered, but when $\kappa\to \infty$ the only possibility for the identity to make sense is to require the right-hand side to be zero, i.e., $\rho^\kappa_{t}(0)=\alpha$ and $\rho^\kappa_t(1)=\beta$, for which one should get Dirichlet boundary conditions. 
  We highlight that our result handles bounded weak solutions,  therefore, we cannot use classical arguments as, for example, the ones used to control the $L^2$-norm of $\rho_t^\kappa$. Since our solutions are given in the weak sense, there is an integral formulation of the equation obtained by means of  integrating with respect to a test function.  Therefore, employing  any argument that uses
 as test function the solution itself, that is $\rho^\kappa$, will not work, since $\rho^\kappa$ is not time differentiable as required for the input test functions. In our case,  the best time regularity we get  is that for each test function $H $ the map $t\mapsto\<\rho^\kappa_t, H_t\>:=\int_0^1\rho^\kappa_t(u)H_t(u)\,du$ is $\sfrac{1}{2}-$H\"older continuous, uniformly on $\kappa$, see Proposition \ref{prop3} below.

Our strategy of proof of the convergence described above is similar to one of \cite{phase}, where it was considered the heat equation evolving on the torus and with another type of Robin boundary condition,
 { 
 		\begin{equation*}
 		\begin{cases} 
 		&\partial_{t}\rho_{t}(u)=\Delta\, {\rho} _{t}(u), \quad (t,u) \in (0,T]\times(0,1), \\
 		&\partial_u{ \rho} _{t}(0)=\partial_u{ \rho}_{t}(1)=\kappa(\rho_t(0)-\rho_t(1)), \quad t\in(0,T],\\
 		&{ \rho}_{0}(u)= g(u), \quad u\in\bb T\,.
 		\end{cases} 
 		\end{equation*}
}
We highlight that even in our linear case, that is, for $m=1$,  there are several  differences between  our case  and the one of \cite{phase}, because the equations not only live in different spaces but also the nature of the boundary conditions is different. Note that  the particle system studied in \cite{phase} models the flux through a barrier  and our system models the flux between two reservoirs. In the present work we handle both with  the linear  ($m=1$) and the non-linear case, and in \cite{phase} only the linear case was observed. If we take our system evolving on the torus and we add a slow bond as done in  \cite{phase} for the linear case, we would obtain the  PME with Robin boundary conditions similar to those of \cite{phase}. The corresponding equation in that case, for $\theta=1$, should be:
\begin{equation}\label{eq:Robin2}
\begin{cases}
&\partial_{t}\rho_{t}^\kappa(u)= \Delta\, ({\rho_t}^\kappa)^m(u), \quad (t,u) \in (0,T]\times\bb T\backslash\{0\} ,\\
&\partial_{u} ({\rho_t}^\kappa)^m(0)=\partial_{u} ({\rho_t}^\kappa)^m(1)=\kappa((\rho_{t}^\kappa(0))^m -(\rho_{t}^\kappa(1))^m ),\quad t \in (0,T], \\
&{ \rho}_{0}^\kappa(u)= g(u), \quad u\in\bb T\,.
\end{cases}
\end{equation} We did not pursue this approach here but we strongly  believe that the techniques employed in \cite{BPGN, phase}  together with those developed in this work, would allow us to analyse it. Nevertheless, we give some details on this approach  in Section  \ref{slow bond}.

  In order to handle our type of weak solutions, we need to obtain a  detailed information about their boundary behavior and this is done based on a non-trivial energy estimate, which is one of the main contributions of the present work. We highlight that we employ the energy estimate that we derive  in order to obtain convergence results for the weak solution of our equations, but we believe that it should have many other application when looking at other observables of the system, we leave this investigation to a future work.  
More precisely, the argument  we employ can be described as follows.  First, we consider a proper $L^2$ space that gives adequate weights at the boundary points $0$ and $1$, which  are explicitly  given by \eqref{p}. From this space, we define an energy functional which we prove to be  finite  on functions $\xi$ such that their  $m$-th power,  namely $\xi^m$,  belongs to  $L^2(0,T;\mathcal H^1)$, where $\mathcal H^1$ is the usual Sobolev space, and their  weak derivative exists and lives on the introduced  $L^2$ space, and they also satisfy the Robin boundary conditions, as given in \eqref{eq:Robin}, almost everywhere  in time. Given this result, we then prove that the solution $\rho^\kappa$ of \eqref{eq:Robin} has finite energy with respect to  that  energy functional, so that all the aforementioned results come for free for $\rho^\kappa$. This is the content of Theorem \ref{energy estimate consequence}. 
From this result we are able to obtain information about the limit points of $\{\rho^\kappa\,:\,\kappa>0\}$ which allows  showing that any limit point $\rho^\star$ is a weak solution of \eqref{eq:Neumann} (resp. \eqref{eq:Dirichlet}) when  $\kappa\to 0$ (resp. $\kappa\to \infty$).  Observe that according to our definition of weak solutions, see Definitions \ref{Def.Neumann} and \ref{Def. Dirichlet}, we have to prove for both situations that $(\rho^\star)^m$ is in the usual Sobolev space and satisfies the corresponding integral equation, this is done through  Propositions \ref{prop2}, \ref{prop5}, \ref{prop3}, and \ref{prop4}.  Nevertheless, in Definition \ref{Def. Dirichlet} there is an extra point  with respect to Definition \ref{Def.Neumann}, which is to show that the solution satisfies, for almost every $t\in(0,T]$, the Dirichlet boundary conditions and this is proved in  Proposition \ref{prop:3.5}.

 Now we explain how we prove that the solution of \eqref{eq:Robin} has finite energy. The idea is to consider an underlying Markovian dynamics, that we denote by $\{\eta_t;\,t\in [0,T]\}$ and  which has a hydrodynamic limit given  by the unique weak solution of \eqref{eq:Robin}. The dynamics of our Markov process $\{\eta_t;\,t\in [0,T]\}$  is composed of combining the dynamics of the porous medium model (PMM) with the dynamics of the symmetric simple exclusion process (SSEP), and a Glauber dynamics. The SSEP dynamics is chosen in such a way that it does not destroy the nature of the hydrodynamic equation \eqref{eq:Robin}, but it is fundamental in order to allow a  microscopic transport of mass. On the other hand, the Glauber dynamics adds boundary conditions as given in \eqref{eq:Robin}, \eqref{eq:Neumann}, and \eqref{eq:Dirichlet}. Now we explain each one of these dynamics as follow. First, we discretize the space where the weak solutions live, that is, the macroscopic space $[0,1]$, by a scaling factor $n$, in such a way that we associate to each interval of the form $[\frac{x}{n},\frac{x+1}{n})$ the microscopic point  $x\in\{1,\ldots, n-1\}$. Now, in the microscopic space $\{1,\ldots, n-1\}$, we allow at most one particle per site and jumps between the sites $x$ and $x+1$  with $x\in\{1,\ldots, n-2\}$, occur by looking at the realization of Poisson processes of rate $r_x(\eta)+n^{a-2}$, with $1<a<2$ and $r_x(\eta)$ is a rate which depends on the occupation sites close to $x$. The form of $r_x$ will be explained below, but it might happen for certain configurations that $r_x(\eta)=0$. In those cases, without the presence of the SSEP dynamics the particle would be blocked and the role of perturbing the system in this way is exaclty to avoid this situation. At the boundary bonds, namely $(0,1)$ and $(n-1,n)$, particles can be injected to  (resp. removed from) the site $1$  or $n-1$, respectively, by  looking at the realization of Poisson processes with rate $\frac{\kappa\alpha}{n^\theta}$ or $\frac{\kappa\beta}{n^\theta}$  (resp. $\frac{\kappa(1-\alpha)}{n^\theta}$ or $\frac{\kappa(1-\beta)}{n^\theta}$). All the Poisson processes are independent. We observe that the role of the parameter $\theta$ is to regulate the intensity of the boundary dynamics, and depending on its range, we will obtain, at the  macroscopic level, different  boundary conditions. The PMM dynamics  is simple to describe: a particle at site $x$ jumps to $x+1$, as long as, there are at least $m-1$ particles in one of the sets of points \begin{equation}\label{eq:sets}
	\{x-(m-1),\ldots, x-1\},\{x-(m-2),\ldots, x-1,x+2\},\ldots, \{x+2,x+3,\ldots, x+m\},\end{equation} 	see the figure below.

\begin{figure*}[h!]
	\begin{center}
		\begin{tikzpicture}[thick, scale=0.75]
			\draw [line width=0.8] (-9,10.5) -- (9,10.5) ; 
			\foreach \x in  {-8,-7,-6,-5,-4,-3,-2,-1,0,1,2,3,4,5,6,7,8} 
			\draw[shift={(\x,10.5)},color=black, opacity=1] (0pt,0pt) -- (0pt,-4pt) node[below] {};
			\draw[] (-2.8,10.5) node[] {};
			
			\draw[] (-6,10.4) node[below] {\tiny{$x-(m-1)$}};
			\draw[] (-1,10.4) node[below] {\tiny{$x-1$}};
			\draw[] (0,10.34) node[below] {\tiny{$x$}};
			\draw[] (1,10.4) node[below] {\tiny{$x+1$}};
			
			
		\shade[shading=ball, ball color=black!50!] (-6,10.76) circle (.245);
		\shade[shading=ball, ball color=black!50!] (-5,10.76) circle (.245);
		\shade[shading=ball, ball color=black!50!] (-4,10.76) circle (.245);
		\shade[shading=ball, ball color=black!50!] (-3,10.76) circle (.245);
		\shade[shading=ball, ball color=black!50!] (-2,10.76) circle (.245);
		\shade[shading=ball, ball color=black!50!] (-1,10.76) circle (.245);
		\shade[shading=ball, ball color=black!50!] (0,10.76) circle (.245);
			
		\end{tikzpicture} 
	\end{center}
\end{figure*}
\vspace{-0.8cm}
\begin{figure*}[h!]
	\begin{center}
		\begin{tikzpicture}[thick, scale=0.75]
			\draw [line width=0.8] (-9,10.5) -- (9,10.5) ; 
			\foreach \x in  {-8,-7,-6,-5,-4,-3,-2,-1,0,1,2,3,4,5,6,7,8} 
			\draw[shift={(\x,10.5)},color=black, opacity=1] (0pt,0pt) -- (0pt,-4pt) node[below] {};
			\draw[] (-2.8,10.5) node[] {};
			
			\draw[] (-5,10.4) node[below] {\tiny{$x-(m-2)$}};
			\draw[] (0,10.34) node[below] {\tiny{$x$}};
			\draw[] (1,10.4) node[below] {\tiny{$x+1$}};
			\draw[] (2,10.4) node[below] {\tiny{$x+2$}};
			
			
		\shade[shading=ball, ball color=black!50!] (-5,10.76) circle (.245);
		\shade[shading=ball, ball color=black!50!] (-4,10.76) circle (.245);
		\shade[shading=ball, ball color=black!50!] (-3,10.76) circle (.245);
		\shade[shading=ball, ball color=black!50!] (-2,10.76) circle (.245);
		\shade[shading=ball, ball color=black!50!] (-1,10.76) circle (.245);
		\shade[shading=ball, ball color=black!50!] (0,10.76) circle (.245);
		\shade[shading=ball, ball color=black!50!] (2,10.76) circle (.245);
		\end{tikzpicture} 
	\end{center}
\end{figure*}
\vspace*{-1.2cm}
\begin{figure*}[h!]
	\begin{center}
		\begin{tikzpicture}[thick, scale=0.75]
			\draw[] (0,10.34) node[below] {$\vdots$};
		\end{tikzpicture} 
	\end{center}
\end{figure*}
\vspace{-0.8cm}
\begin{figure*}[h!]
	\begin{center}
		\begin{tikzpicture}[thick, scale=0.75]
			\draw [line width=0.8] (-9,10.5) -- (9,10.5) ; 
			\foreach \x in  {-8,-7,-6,-5,-4,-3,-2,-1,0,1,2,3,4,5,6,7,8} 
			\draw[shift={(\x,10.5)},color=black, opacity=1] (0pt,0pt) -- (0pt,-4pt) node[below] {};
			\draw[] (-2.8,10.5) node[] {};
			
			\draw[] (-1,10.4) node[below] {\tiny{$x-1$}};
			\draw[] (0,10.34) node[below] {\tiny{$x$}};
			\draw[] (1,10.4) node[below] {\tiny{$x+1$}};
			\draw[] (6,10.4) node[below] {\tiny{$x+m-1$}};
			
			
		\shade[shading=ball, ball color=black!50!] (-1,10.76) circle (.245);
		\shade[shading=ball, ball color=black!50!] (0,10.76) circle (.245);
		\shade[shading=ball, ball color=black!50!] (2,10.76) circle (.245);
		\shade[shading=ball, ball color=black!50!] (3,10.76) circle (.245);
		\shade[shading=ball, ball color=black!50!] (4,10.76) circle (.245);
		\shade[shading=ball, ball color=black!50!] (5,10.76) circle (.245);
		\shade[shading=ball, ball color=black!50!] (6,10.76) circle (.245);
		\end{tikzpicture} 
	\end{center}
\end{figure*}
\vspace{-0.8cm}
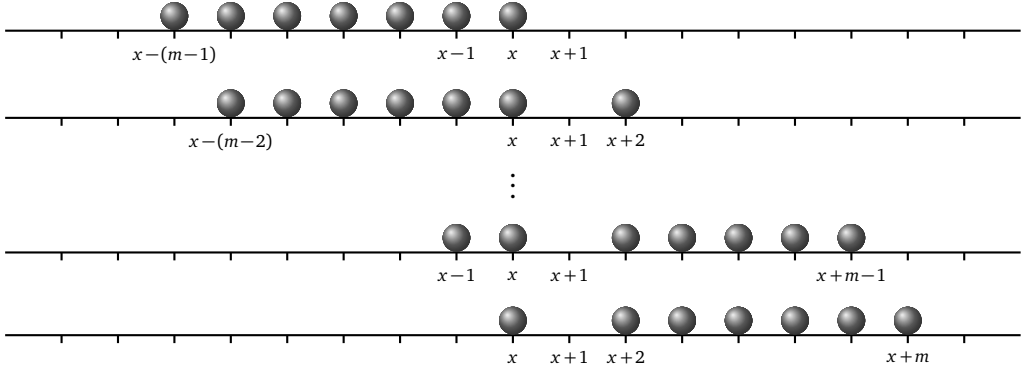
\begin{figure}[h!]
	\begin{center}
		\begin{tikzpicture}[thick, scale=0.75]
			\draw [line width=0.8] (-9,10.5) -- (9,10.5) ; 
			\foreach \x in  {-8,-7,-6,-5,-4,-3,-2,-1,0,1,2,3,4,5,6,7,8} 
			\draw[shift={(\x,10.5)},color=black, opacity=1] (0pt,0pt) -- (0pt,-4pt) node[below] {};
			\draw[] (-2.8,10.5) node[] {};
			
			\draw[] (0,10.34) node[below] {\tiny{$x$}};
			\draw[] (1,10.4) node[below] {\tiny{$x+1$}};
			\draw[] (2,10.4) node[below] {\tiny{$x+2$}};
			\draw[] (7,10.4) node[below] {\tiny{$x+m$}};
			
			
			\shade[shading=ball, ball color=black!50!] (0,10.76) circle (.245);
			\shade[shading=ball, ball color=black!50!] (2,10.76) circle (.245);
			\shade[shading=ball, ball color=black!50!] (3,10.76) circle (.245);
			\shade[shading=ball, ball color=black!50!] (4,10.76) circle (.245);
			\shade[shading=ball, ball color=black!50!] (5,10.76) circle (.245);
			\shade[shading=ball, ball color=black!50!] (6,10.76) circle (.245);
			\shade[shading=ball, ball color=black!50!] (7,10.76) circle (.245);
		\end{tikzpicture} 
	\end{center}
	\caption{Sets of points for which the particle at $x$ jumps  to $x+1$ in the PMM dynamics.}
	\label{figure-sets}
\end{figure}

\noindent The jump mechanism from $x+1$ to $x$ is exactly the same as described above. We observe that the exact form of the jump rate is given \eqref{PMM rate}.

 Since we considered a space rescaling $u\to u/n$, if we want to see a non-trivial evolution of the density, we need to accelerate the process by a factor $n^2$. In this case, we will have to take the so-called diffusive time scale $tn^2$, so that we analyze the process $\{\eta_{tn^2}; t\in[0,T]\}$. Since we only allow one particle per site, the state space of our Markov process is the set $\{0,1\}^{\{1,\ldots, n-1\}}$ and when $\eta_{tn^2}(x)=1$ or $0$ the site $x$ is occupied or empty, respectively.

To convince the reader that our choice of the bulk dynamics is related to the porous medium equation, we present an heuristic argument as follows. First, we observe that the dynamics introduced above, when $\alpha=\beta=\rho$, has a unique invariant probability, the Bernoulli product measure denoted by $\nu^{n}_\rho$.  In fact, in this case,  this measure is reversible. Under this probability measure, the random variables $\eta(x)$ are all independent and Bernoulli distributed with parameter $\rho$. Moreover, the infinitesimal generator $L_n^m$ of this process, which is defined in  \eqref{eq:gen_full}, is such that when we look at the {occupation variable at site $x$,  $\eta(x)$,} it acts as $L^m_n\eta(x)=j_{x-1,x}^m(\eta)-j^m_{x,x+1}(\eta)$, where $j_{x,x+1}^m(\eta)$ denotes the instantaneous current associated to the bond $(x,x+1)$, and in this case, it writes as in \eqref{eq:current}. Putting this together, we see that \begin{equation}\label{heuristics}
L_n^m \eta(x)=\tau_{x-1}h^m(\eta)+\tau_{x+1}h^m(\eta)-2\tau_xh^m(\eta),\end{equation}
 which is close to a discrete Laplacian of the function $\tau_xh^m$. The definition of $\tau_xh^m$ is  given in \eqref{shift}. This gives us an intuition that the equation we are looking for is a parabolic equation of the form  $\partial_t\rho_t(u)=\Delta \langle h^m\rangle_{\rho_t(u)}$, where $\langle\cdot\rangle_{\rho_t(u)}$ denotes the expectation with respect to the Bernoulli product measure  $\nu_{\rho_t(\cdot)}$. Now for the special choice of our dynamics, a simple computation shows that $\langle h^m\rangle_{\rho_t(u)}=(\rho_t(u))^m$, and from this, we recover the porous medium equation.
 
 We note that the PMM dynamics just described has the peculiarity that once we do not have  $m-1$ particles in one of the sets of points given  in \eqref{eq:sets}, then  the particle at $x$  is blocked. To overcome this difficulty, as in \cite{BPGN,patricia}, we perturb the PMM dynamics with a SSEP dynamics in such a way that we allow particles to move to  the neighbouring sites at a rate that vanishes as $n\to \infty$, so that macroscopically, at the level of the partial differential equations,  we do not see the effect of this perturbation.  We point that this hand-waving argument just described is rigorously proved in \cite{BPGN}, and we refer the reader to that article for more details on that result. There, it was proved that the density of particles of this model is ruled by the unique weak solution of \eqref{eq:Robin} for $\theta=1$, \eqref{eq:Neumann} for $\theta>1$ and \eqref{eq:Dirichlet} for $\theta\in[0,1)$.
In particular, in \cite{BPGN} it was proved, from the microscopic dynamics, that the density of particles  satisfies all the requirements in the notions of weak solutions given in Definitions {\ref{Def. Robin}, \ref{Def.Neumann}, and  \ref{Def. Dirichlet}.
 By using this underlying dynamics we are able to deduce an energy estimate for the solution $\rho^\kappa$
  of \eqref{eq:Robin}, which allows extracting a lot of information about the behavior of this solution close to the boundary of $[0,1]$. 
This is the content of Proposition \ref{prop:energy_estimate}.

 We believe that our results can be extended to  more general equations of the form 
 	\begin{equation*}\label{eq:intro_general}
	\begin{cases}
	&\partial_{t}\rho_{t}^\kappa(u)= \Delta\, \Phi({\rho} _{t}^\kappa(u)), \quad (t,u) \in (0,T]\times(0,1), \\
	&\partial_{u}\Phi(\rho_{t}^\kappa(0))=\kappa(\Phi(\rho_{t}^\kappa(0)) -\alpha), \quad t\in(0,T],\\
	&\partial_{u}\Phi( \rho_{t}^\kappa(1))=\kappa(\beta -\Phi(\rho_{t}^\kappa(1)))\,, \quad t\in(0,T],
	\end{cases}
	\end{equation*}
where $\Phi$  is  a $C^1([0,1])$ function and the $L^2$-norm of  $\rho_t^\kappa$  is uniformly  bounded on $\kappa$. Here we considered the case $\Phi(\rho^\kappa)=(\rho^\kappa)^m$ and $\rho^\kappa\in[0,1]$, since it was the natural equation arising from our microscopic system. When dealing with, for example, zero-range dynamics or exclusion dynamics with more than one particle per site, this is the natural candidate equation that one should obtain in the hydrodynamic limit for those models, see \cite{fmn,KL}. We intend to pursue the extension of our results to those cases in the near future, but one of the obstacles to be overcome is obtaining a uniform bound on $\kappa$ for the $L^2$-norm of $\rho_t^\kappa$.
 
Here follows an outline of this article. In Section \ref{sec:statement_results}, we introduce notation, and we state our main results, namely the energy estimate and its consequences, and our convergence result for $\rho^\kappa$. In Section \ref{sec:main_result}, we prove our convergence result, that is, {the weak solution of \eqref{eq:Robin}, $\rho^\kappa$,} converges as $\kappa\to 0$ or $\kappa\to \infty$ to the weak solution of \eqref{eq:Neumann} or \eqref{eq:Dirichlet}, respectively. In Section \ref{sec:prop:energy_consequence}, we derive several consequences of the energy estimate. In Section \ref{IPS} we prove, by means of the Markovian dynamics, that  $\rho^\kappa$ satisfies the energy estimate. Finally, in the Appendix, we present some auxiliary results that are needed throughout the article.

	\section{Statements of results}\label{sec:statement_results}

	In this section we describe the different partial differential equations  and the respective notion of weak solutions that we will consider along this article. For that, fix an interval $ I \subset \mathbb R$, $T>0$ and $n,p \in \mathbb{N}\cup\{0\}$. We denote by  $C^{n,p}([0, T] \times  I)$, the set of all real-valued functions defined on $[0,T] \times I $ that are $n$ times  differentiable on the first variable and $p$ times differentiable on the second variable (with continuous derivatives); by $C^{n}([0,1])$ (resp. $C^{n}_{c}((0,1))$) the set of all $n$ times continuously differentiable real-valued functions defined on $[0,1]$ (resp. and with compact support in $(0,1)$); by  $C_c^{n,p}([0,T]\times(0,1))$, the set of all real-valued functions $G\in C^{n,p}([0,T]\times(0,1))$ with compact support in $[0,T]\times(0,1)$; by  $C^{n,p}_0 ([0,T] \times [0,1])$, the set of all real-valued functions $G \in C^{n,p}([0, T] \times[0, 1])$ such that $G_{s}(0)=G_{s}(1)=0$, for all $s\in[0,T]$. 
	When $n=\infty$ or $p=\infty$ it means that the function is infinitely differentiable in the corresponding  variable. We denote  by
	$\langle\cdot, \cdot\rangle$, the inner product in $L^2([0,1])$ with corresponding norm $\| \cdot \|_{2}$.
	
		\begin{definition}
		Let $G,H \in L^{2}([0,T]\times[0,1])$. We denote the inner product in $L^{2}([0,T]\times[0,1])$ by
		\begin{equation}\label{double bracket}
		\dl G,H \dr := \int_{0}^{T}\langle G_{s},H_s\rangle \,\,ds.
		\end{equation}
	\end{definition}
Now we introduce the  Sobolev space $\mathcal H^1.$

	\begin{definition}[Sobolev space]
		\label{Def. Sobolev space}
		Let $\mathcal{H}^{1}$ be the set of all locally summable functions $\varphi: (0,1) \to \mathbb{R}$ such that there exists a function $\partial_{u} \varphi \in L^{2}([0,1])$ satisfying
		$$\langle  \varphi ,\partial_u g\rangle =-\langle  \partial_{u} \varphi , g\rangle,$$
		for all $g \in C_{c}^{\infty}\left((0,1)\right)$. For $\varphi \in \mathcal{H}^1$, we define the norm 
		\begin{equation*}
		\|\varphi \|^{2}_{\mathcal{H}^1} := \|\varphi\|^{2}_{2} + \|\partial_{u} \varphi \|^{2}_{2}.
		\end{equation*}  
			\end{definition}
Recall that a function $\varphi\in \mathcal H^1$ can be extended to $[0,1]$ by setting  $\varphi(0):=\varphi(0^+)$ and $\varphi(1):=\varphi(1^-)$. 
		\begin{definition}
		Let $L^2(0,T; \mathcal{H}^1)$ be the set of all measurable functions $\zeta: [0,T]\to \mathcal{H}^1$ such that
		\begin{equation}\label{sobolev norm 2}
		\|\zeta\|^{2}_{L^2(0,T;\mathcal{H}^1)} := \int_{0}^{T}\|\zeta_t\|^{2}_{\mathcal{H}^1}\, dt < \infty.
		\end{equation}
	\end{definition}

		\subsection{The porous medium equation}\label{sec: hydro eq}
		
		Along the text we fix the parameters 
				\begin{equation}\label{ab}
		\alpha,\beta\in(0,1)\;\mbox{  and  }\;m\in\mathbb N.	\end{equation}    
The restrictions on $\alpha,\beta\in(0,1)$ are  technical\footnote{It is used in entropy bound for the discrete model, see \eqref{ent}.}  and  throughout the text we will point out when the results do not need these restrictions.

	\begin{definition}[PME with a type of Robin boundary conditions]
	\label{Def. Robin}
	Let $T>0$, $\kappa >0$ and $g:[0,1]\rightarrow [0,1]$ a measurable function. We say that  $\rho^\kappa:[0,T]\times[0,1] \to [0,1]$ is a weak solution of the porous medium equation with Robin boundary conditions, see \eqref{eq:Robin}, 
	if the following conditions hold: 
	\begin{enumerate}
		\item $(\rho^\kappa)^m \in L^{2}(0,T; \mathcal{H}^{1})$; 
		\item $\rho^\kappa$ satisfies the integral equation:
		\begin{equation}\label{eq:Robin integral}
		\begin{split}
		\langle \rho_{t}^\kappa,  G_{t}\rangle  -\langle g,  G_{0}\rangle &- \int_0^t\langle \rho_{s}^\kappa,( \partial_s G_{s}+(\rho_s^\kappa)^{m-1} \Delta G_s ) \rangle   \, ds
		\\&+\int^{t}_{0}  \big\{  ({\rho_s}^\kappa)^m(1) \partial_u G_{s}(1)- ({\rho_s}^\kappa)^m(0) \partial_u G_{s}(0) \big\} \, ds\\
		&-  \kappa\int^{t}_{0} \big\{ G_{s}(0)(\alpha -\rho_{s}^\kappa(0)) + G_{s}(1)(\beta -\rho_{s}^\kappa(1)) \big\} \, ds=0,
		\end{split}   
		\end{equation}
		for all $t\in [0,T]$ and all functions $G \in C^{1,2} ([0,T]\times[0,1])$. 
	\end{enumerate}
\end{definition}

	\begin{definition}[PME with Neumann boundary conditions]
		\label{Def.Neumann}
		Let $T>0$ and $g:[0,1]\rightarrow [0,1]$ a measurable function. We say that  $\rho:[0,T]\times[0,1] \to [0,1]$ is a weak solution of the porous medium equation with Neumann boundary conditions, see \eqref{eq:Neumann}, if  $\rho^m \in L^{2}(0,T; \mathcal{H}^{1})$ and $\rho$ satisfies the integral equation \eqref{eq:Robin integral} with $\kappa=0$.
	\end{definition}
	
	\begin{definition}[PME with Dirichlet boundary conditions]
		\label{Def. Dirichlet}
		Let $T>0$ and $g:[0,1]\rightarrow [0,1]$ a measurable function. We say that $\rho:[0,T]\times[0,1] \to [0,1]$ is a weak solution of the porous medium equation with  Dirichlet boundary conditions, see \eqref{eq:Dirichlet},
		if the following conditions hold:
		\begin{enumerate}
			\item $\rho^{m} \in L^{2}(0,T; \mathcal{H}^{1})$;
			\item $\rho$ satisfies the integral equation:
			\begin{equation}\label{eq:Dirichlet integral}
			\begin{split}
			\langle \rho_{t} , G_{t}\rangle  -\langle g ,  G_{0}\rangle - \int_0^t\langle \rho_{s}, (\partial_s G_{s} + (\rho_s)^{m-1} \Delta G_{s} ) \rangle \, ds
			+ \int_0^t\big\{\beta^m\partial_uG_s(1) -\alpha^m\partial_uG_s(0)   \big\}\,ds =0,
			\end{split}   
			\end{equation}
			for all $t\in [0,T]$ and all functions $G \in C_0^{1,2} ([0,T]\times[0,1])$;
			\item for almost every $t\in(0,T]$, $\rho_{t}(0)=\alpha$ and $\rho_{t}(1)=\beta$.
		\end{enumerate}
	\end{definition}

	\begin{remark} \label{eq:neumann}
		Observe that for $m=1$ the equations above become the heat equation with different  boundary conditions. 
	\end{remark} 
	
		We observe that the weak solutions of \eqref{eq:Robin}, \eqref{eq:Neumann} and  \eqref{eq:Dirichlet} in the sense given above are unique. For a proof see, for example,  Section 7 of \cite{BPGN}. We observe that in \cite{BPGN} item (3) of Definition \ref{Def. Dirichlet} was asked to hold for all time $t\in(0,T]$, but, in fact, uniqueness holds as long as we ask it to be true for almost every $t\in(0,T]$. For a deeper discussion of the porous medium equation, we refer the reader to the seminal book \cite{vazquez_book}.

	\subsection{Main results}

We now state our fundamental results. In Section \ref{s221}, we start by introducing a proper weighted $L^2$ space 
and an energy functional. If the energy of a function $\xi$ is finite, then the functional captures a lot of information about $\xi^m$, see Proposition \ref{prop:2.6}. Then, in Section \ref{s222}, we will state that the solution $\rho^\kappa$ of \eqref{eq:Robin} has finite energy, and  this is the content of Theorem  \ref{energy estimate consequence}. Finally, in Section \ref{s223}, we will state our convergence result, see Theorem \ref{main theorem}.

	\subsubsection{Energy estimate}\label{s221}

Recall that the parameters $\alpha, \beta$ and $m$ are fixed as in \eqref{ab}.
	Let $\kappa>0$ and $a,b\geq 0$.
	We define a measure $W^{\alpha,\beta}_{\kappa,a,b}$ on $[0,1]$ by
	\begin{equation*}
	W_{\kappa,a,b}^{\alpha,\beta}(du) := du + \frac{1}{k}P^{\alpha}_{m}(a)\delta_{0}(du) + \frac{1}{k}P^{\beta}_{m}(b)\delta_{1}(du), 
	\end{equation*}
	where $\delta_{z}(du)$, with $z\in \{0,1\}$, is the Dirac measure and 
	\begin{equation}\label{p}
		P^{\gamma}_{m}(\rho) = \sum_{i=0}^{m-1}\gamma^{m-1-i}\rho^{i},\mbox{ for } \gamma\in\{\alpha, \beta\}\mbox{ and } \rho\geq 0\,.
	\end{equation} 
Thus,  
\begin{equation}\label{ppositive}
P_{m}^{\gamma}(\rho)\geq \gamma^{m-1}>0,\end{equation}
because $\rho\geq 0$ and $\gamma >0$, since $\gamma\in\{\alpha, \beta\}$ and  the restrictions of \eqref{ab}. If $\gamma=0$, then $P_{m}^{0}(\rho)=\rho^{m-1}$ so that the inequality \eqref{ppositive} would not hold.	Observe that, since
	 \begin{equation}\label{newton_binom}
	\gamma^m-\rho^m= (\gamma- \rho)\sum_{i=0}^{m-1}\gamma^{m-1-i}\rho^{i}=	(\gamma- \rho)P^{\gamma}_{m}(\rho) \,,
	\end{equation} 
then
	\begin{equation}\label{pp}
\gamma- \rho	= \frac{\gamma^m-\rho^m}{P^{\gamma}_{m}(\rho) },
	\end{equation}
for	$\gamma\in\{\alpha, \beta\}$ and for all $\rho\geq 0$.  The measure $W^{\alpha,\beta}_{\kappa,a,b}$ is the sum of the Lebesgue measure and Dirac measures concentrated on $0$  and $1$ with weights $\frac{1}{k}P^{\alpha}_{m}(a)$  and  $\frac{1}{k}P^{\beta}_{m}(b)$, respectively. 
	For 
	$g\in L^2([0,1])$ { such that $g(0)$ and $g(1)$ are both well-defined},
	we denote
	\begin{equation}\label{ppp} \| g \|^{2}_{W^{\alpha,\beta}_{\kappa,a,b}}:=\int_0^1g^2(u)\, W_{\kappa,a,b}^{\alpha,\beta}(du) \,.\end{equation}

		\begin{definition}\label{space_weighted}
	Let $\mathcal B$ be the space of {measurable} functions $\xi:[0,T]\times [0,1]\to {\mathbb [0,\infty)}$ such that the applications $s\mapsto\xi_s(0)$ and $s\mapsto\xi_s(1)$ are measurable and  bounded.
	\end{definition}

	\begin{definition}\label{L2 space weight}
		Let $\xi\in \mathcal B$. 
		For any $\kappa>0$, we denote by $L_{\kappa, \xi}^{2}([0,T]\times[0,1])$ the Hilbert space composed of all measurable functions $H:[0,T]\times[0,1]\to \mathbb{R}$ such that
	\begin{equation}\label{weighted norm}
		\begin{split}
		\dl H,H \dr_{\kappa,\xi}^{\alpha,\beta} &:= 	
	\int_{0}^{T} \| H_{s} \|^{2}_{W_{\kappa,\xi_s(0),\xi_s(1)}^{\alpha,\beta}} \, ds\\
	&=\dl H,H \dr + \int_{0}^{T}\left\{\dfrac{P_{m}^{\alpha}(\xi_s(0))}{\kappa}H_s^{2}(0) + \dfrac{P_{m}^{\beta}(\xi_s(1))}{\kappa}H_s^{2}(1)\right\}\,ds<\infty, \\
		\end{split}
		\end{equation}
		where $\dl \cdot, \cdot \dr$ is defined in \eqref{double bracket} and for $\gamma \in \{\alpha,\beta\}$, $P^{\gamma}_{m}$ was defined  in \eqref{p}.
	\end{definition}
Moreover, $L_{\kappa,\xi}^{2}([0,T]\times[0,1]) \subseteq L^{2}([0,T]\times[0,1])$. 	
	\begin{definition}
		For a  function $\xi$ such that $\xi^m\in L^2([0,T]\times [0,1])$,  we define the functional $\mathcal{T}_{\xi,m}^{\alpha,\beta}$   on $C^{0,1}([0,T]\times [0,1])$ by
		\begin{equation}\label{functional}
		\begin{split}
		\mathcal{T}_{\xi,m}^{\alpha,\beta}(H):= \dl\, \xi^m, \,\partial_u H \dr + \int_{0}^{T} \big\{\alpha^m H_s(0) - \beta^m H_s(1)\big\}\, ds.
		\end{split}
		\end{equation}
	\end{definition}

	Let us define the energy functional: 
	
	\begin{definition}[Energy functional] \label{enerrrrgy}For each $\kappa>0$ and  $c>0$ fixed, we define the functional $\mathcal{E}_{m,\kappa,c}^{\alpha, \beta} $ which acts on functions  $\xi\in\mathcal B$ such that $\xi^m\in L^2([0,T]\times [0,1])$ as
		\begin{equation}\label{energy_functional}
		\mathcal{E}_{m,\kappa,c}^{\alpha, \beta} (\xi) :=	\sup _{H\in C^{0,1}([0,T]\times [0,1])}\left\{ \mathcal{T}_{\xi,m}^{\alpha,\beta}(H) - c\<\!\<H, H\>\!\>_{\kappa,\xi}^{\alpha,\beta}\right\} .
		\end{equation}
	\end{definition}

	Fortunately, by means of estimating  the energy functional $\mathcal{E}_{m,\kappa,c}^{\alpha, \beta} (\cdot) $ we are able to obtain  a lot of information about $\xi^{m}$, which is given by the next proposition. 
	\begin{proposition}\label{prop:2.6}
	Let $\xi\in\mathcal B$ such that $\xi^m\in L^2([0,T]\times [0,1])$ and $\mathcal{E}_{m,\kappa,c}^{\alpha,\beta}(\xi) \leq M_0<\infty,$ for some $\kappa>0$, $c>0$, and $M_0>0$. Then, there exists $\partial_u \xi^m\in L_{\kappa,\xi}^{2}([0,T]\times[0,1])$ such that for all $H \in C^{0,1}([0,T]\times[0,1])$
	\begin{equation}\label{III}
	\mathcal{T}_{\xi,m}^{\alpha,\beta}(H) = -\dl \partial_u \xi^m, H \dr_{\kappa,\xi}^{\alpha,\beta} ,
	\end{equation}
	and $\xi^m \in L^{2}(0,T; \mathcal{H}^1).$ Moreover,  
		\begin{equation}\label{energy exp!!!!}
		\begin{split}
		\mathcal{E}_{m,\kappa,c}^{\alpha, \beta}(\xi)=
		\,\frac{1}{4c} \<\!\< \partial_u\xi^m, \partial_u\xi^m \>\!\>_{\kappa,\xi}^{\alpha,\beta} ,
		\end{split}
		\end{equation} and  
			\begin{equation}\label{boundaryconditions}
		\begin{split}
		\partial_u(\xi_s)^{m}(0)\,\,P_m^\alpha(\xi_s(0) )=\kappa\,\big((\xi_s)^m(0) - \alpha^m\big)\quad \textrm{and}\quad 
		\partial_u(\xi_s)^{m}(1)\, \,P_m^\alpha(\xi_s(0) )=\kappa\,\big(\beta^m -(\xi_s)^m(1)\big),
		\end{split}
		\end{equation} for almost every {$s\in(0,T]$}. 
	\end{proposition}
We observe that last result holds for $\alpha,\beta\in[0,1]$.
\begin{remark}
	In particular, since we have   assumed $\alpha, \beta >0$ (see \eqref{ab}), from \eqref{pp} and \eqref{ppositive},  the boundary conditions in \eqref{boundaryconditions} become
	\begin{equation}\label{boundaryconditions1}
	\begin{split}
	\partial_u(\xi_s)^{m}(0)=\kappa(\xi_s(0) - \alpha)\quad \textrm{and}\quad 
	\partial_u(\xi_s)^{m}(1)=\kappa(\beta -\xi_s(1)),
	\end{split}
	\end{equation} for almost every {$s\in(0,T]$}.
\end{remark}

\subsubsection{Properties of the weak solution of \eqref{eq:Robin}.} \label{s222}
In the next theorem, we state that  the unique weak solution $\rho^\kappa$ of \eqref{eq:Robin} has finite energy, and from the last proposition we obtain information about $(\rho^\kappa)^{m}$. The proof of the next theorem is quite long and it is presented in  Section \ref{IPS}. The idea is to consider an underlying interacting particle system of exclusion type, whose hydrodynamic limit is ruled by the weak solution of \eqref{eq:Robin} and  from its properties, we can prove the energy bound \eqref{energy estimate good}.
	\begin{theorem}[Energy estimate]\label{energy estimate consequence}
		For any $\kappa>0$, there exists a constant $c>0$ such that the unique weak solution $\rho^{\kappa}: [0,T]\times[0,1]\to[0,1]$ of \eqref{eq:Robin} satisfies the energy estimate:
		\begin{equation}\label{energy estimate good}
		\mathcal E_{m,\kappa,c}^{\alpha, \beta}(\rho^\kappa)\leq M_0,
		\end{equation}
		where $M_0$ is a constant that does not depend on $\kappa$. As a consequence, for all $\kappa>0$, the weak solution $\rho^{\kappa}$ satisfies the boundary conditions:
		\begin{equation} \label{eq:Robin_bc}
		\begin{split}
		\partial_u(\rho_s^{\kappa})^{m}(0)= \kappa(\rho_s^{\kappa}(0)-\alpha) \quad\textrm{and}\quad
		\partial_u(\rho_s^{\kappa})^{m}(1) = \kappa(\beta - \rho_s^{\kappa}(1)),
		\end{split}
		\end{equation}
for almost every $s\in(0,T]$ and the set $\left\{ (\rho^\kappa)^{m}:\, \kappa > 0 \right\}$ is bounded in $L^{2}(0,T; \mathcal{H}^1)$.
	\end{theorem}

\begin{remark}
If we had not assumed $\alpha, \beta >0$, the boundary conditions \eqref{eq:Robin_bc} would become
	\begin{equation}\label{eq:Robin_bc1}
\begin{split}
\partial_u(\rho_s^{\kappa})^{m}(0)\,\,P_m^\alpha(\rho_s^{\kappa}(0) )=\kappa\,\big((\rho_s^{\kappa})^m(0) - \alpha^m\big)\quad \textrm{and}\quad 
\partial_u(\rho_s^{\kappa})^{m}(1)\, \,P_m^\alpha(\rho_s^{\kappa}(0) )=\kappa\,\big(\beta^m -(\rho_s^{\kappa})^m(1)\big),
\end{split}
\end{equation} for almost every {$s\in(0,T]$}.
\end{remark}

The result just stated  is a generalization of the one presented in \cite{BPGN} in the following sense, in \cite{BPGN}, the estimate \eqref{energy estimate good} was derived:  for a different  functional, for other purposes than the ones we use here, for functions $H\in C_{c}^{0,1}([0,T]\times(0,1))$ and only for $m=2$.

	\subsubsection{The convergence of $\rho^\kappa$}\label{s223}
	Finally, we are able to state our convergence theorem. Once we have the energy estimates as stated in Theorem \ref{energy estimate consequence} for the solution $\rho^\kappa$ of \eqref{eq:Robin}, we are able to deduce the limit of the weak solution $\rho^\kappa$ by tuning the parameter $\kappa$ either to zero or to infinity and recover the weak solution of the porous medium equation with Neumann or Dirichlet boundary conditions, respectively.

	\begin{theorem}\label{main theorem}
		Let $g:[0,1]\to [0,1]$ be a measurable function. For each $\kappa > 0$, let $\rho^\kappa:[0,T]\times[0,1]\to [0,1]$ be the unique weak solution of \eqref{eq:Robin} with initial condition $g$.
		Then,
		\begin{equation*}
		\displaystyle \lim_{\kappa\to 0} \rho^\kappa \; = \;  \rho^0 \quad \textrm{ and } \quad \displaystyle \lim_{\kappa\to \infty} \rho^\kappa \; = \;  \rho^{\infty}
		\end{equation*}
		in $L^2([0,T]\times [0,1])$, where
		$\rho^{0}$ is the unique weak solution of \eqref{eq:Neumann}, and $\rho^{\infty}$ is the unique weak solution of \eqref{eq:Dirichlet}, both with initial condition $g$.
	\end{theorem}

	\section{Proof of Theorem \ref{main theorem}}\label{sec:main_result}
	The aim of this section is to prove Theorem \ref{main theorem}. First, we present the proof of the theorem, and then we present all the auxiliary results.
	
	\begin{proof}[Proof of Theorem \ref{main theorem}]
	 We need to show that any limit point, in $L^{2}([0,T]\times[0,1])$, of the set $\{\rho^{\kappa}: \kappa >0\}$ obtained when $\kappa \to 0$ (resp. $\kappa \to \infty$)   satisfies items $(1)$ and $(2)$ of Definition \ref{Def.Neumann} (resp. items $(1)$, $(2)$ and $(3)$ of  Definition \ref{Def. Dirichlet})).
	 
	Since {$\rho^\kappa_t\in[0,1]$}, for all $t\in[0,T]$, the set  $\left\{\rho^{\kappa}: \, \kappa >0 \right\}$ is bounded in  $L^{2}([0,T]\times[0,1])$ and, as a consequence, any sequence of $\left\{\rho^{\kappa}: \, \kappa >0 \right\}$ has a convergent subsequence in $L^{2}([0,T]\times[0,1])$. 
Thus, let us consider  $\{\kappa_{j}\}_{j\in\mathbb{N}}$ such that $\kappa_{j}\to 0$ (resp. $\kappa_{j}\to \infty$).  Then, there exist
			$\{\kappa_{j_\ell}\}_{\ell\in\mathbb{N}}\subset \{\kappa_{j}\}_{j\in\mathbb{N}}$ and  $\rho^{\star}\in L^{2}([0,T]\times[0,1])$ such that $\rho^{\kappa_{j_\ell}}$ converges to $\rho^{\star}$ in $L^{2}([0,T]\times[0,1])$, when $\ell\to\infty$.  
			
			Fixed $m\in\mathbb N$.
			We start by observing  that  $\rho^{\star}$ satisfies  item $(1)$ of Definition \ref{Def.Neumann} (resp. Definition \ref{Def. Dirichlet}), i.e., $(\rho^\star)^m$ belongs to $L^{2}(0,T; \mathcal{H}^1)$, as a consequence of 
Proposition \ref{prop2}  (see below). 
Moreover, in the case $\kappa_{j}\to \infty$, Proposition \ref{prop:3.5} (see below)  implies that $\rho^\star$ satisfies  item $(3)$ of  Definition \ref{Def. Dirichlet}, which allows replacing $\rho^\star_t(0)$ by $\alpha$ (resp. $\rho^\star_t(1)$ by $\beta$) for almost every $t\in(0,T]$.

	In order to finish the proof, we only need to show that $\rho^\star$ satisfies the integral equation, i.e.,  when $\kappa_j\to 0$ (resp. $\kappa_j\to \infty$), the limit  $\rho^\star$ satisfies  item  $(2)$ of Definition \ref{Def.Neumann}  (resp.  item  $(2)$ of Definition \ref{Def. Dirichlet}). 
	Note that after an integration by parts (see Lemma \ref{IIP}), we can rewrite the integral equation \eqref{eq:Robin integral} for $\kappa=\kappa_{j_\ell}$ as  
	\begin{equation}\label{integral1}
	\begin{split}
	\langle \rho_{t}^{\kappa_{j_\ell}}, \; H_{t}\rangle -\langle g, \; H_{0}\rangle &+ \int_{0}^{t}\< \partial_{u}\left(\rho^{\kappa_{j_\ell}}_{s}\right)^{m}, \;\partial_{u}H_s \>\, ds - \int_{0}^{t}\< \rho_{s}^{\kappa_{j_\ell}},\; \partial_{s}H_s \>\, ds \\
	&-  \int^{t}_{0} \kappa_{j_\ell}\big\{(\alpha -\rho_{s}^{\kappa_{j_\ell}}(0)) \,H_{s}(0) + (\beta -\rho_{s}^{\kappa_{j_\ell}}(1))\,H_{s}(1) \big\} \, ds=0,
	\end{split}   
	\end{equation}
	for all $t\in[0,T]$ and any test function $H\in C^{1,2}([0,T]\times[0,1])$. 
	We highlight that for $\kappa_{j_\ell} \to \infty$, as $\ell \to \infty$, we need to use test functions $H\in C_{0}^{1,2}([0,T]\times[0,1])$, because, in the limit,  we want to obtain  the integral equation \eqref{eq:Dirichlet integral}, which has test functions  in $C_{0}^{1,2}([0,T]\times[0,1])$. In this case, the boundary terms of   \eqref{integral1} 
	disappear and  \eqref{integral1}  becomes
	\begin{equation}\label{eqq3.2}
	\begin{split}
	\langle \rho_{t}^{\kappa_{j_\ell}},  H_{t}\rangle -\langle g,  H_{0}\rangle &+ \int_{0}^{t}\langle \partial_{u}\left(\rho^{\kappa_{j_\ell}}_{s}\right)^{m},\; \partial_{u}H_s \rangle\, ds - \int_{0}^{t}\langle \rho_{s}^{\kappa_{j_\ell}}, \;\partial_{s}H_s \rangle\, ds =0,
	\end{split}   
	\end{equation}
	for all $t\in[0,T]$ and all $H\in C_{0}^{1,2}([0,T]\times[0,1])$. Recall that in this case all the information at the boundary is given by $\rho^\star_t(0)=\alpha$ and $\rho^\star_t(1)=\beta$, for almost every $t\in(0,T]$.
	
	In the case when $\kappa_{j_\ell} \to 0$, as $\ell \to \infty$, the boundary terms in \eqref{integral1} vanish, when $\ell\to \infty$. In fact,
	since $ \alpha,\beta$, and $\rho^\kappa$ are bounded from above by one, the boundary term of \eqref{integral1} is bounded from above by
	$$\left|\int^{t}_{0} \kappa_{j_\ell}\big\{ (\alpha -\rho_{s}^{\kappa_{j_\ell}}(0))H_{s}(0) + (\beta -\rho_{s}^{\kappa_{j_\ell}}(1)) H_{s}(1) \big\} \, ds \right| \leq \kappa_{j_\ell} 2T\|H\|_{\infty},$$
	which vanishes as $\ell\to \infty$, {for all $t\in[0,T]$. Above
\begin{equation*}\label{linfinito}
\|H\|_{\infty}=\sup_{(t,u)\in [0,T]\times[0,1]}|H_t(u)|\,.
\end{equation*} }
	
We will analyze the remaining terms of \eqref{integral1} (or the terms of \eqref{eqq3.2}) separately.
	This is done at the end of this section:  Proposition \ref{prop5}  handles with the term with the weak derivative, and Proposition \ref{prop4}  deals with the remaining  terms. From those results, we can  see that  \eqref{integral1} converges, as $\ell\to \infty$, to 	
	\begin{equation*}
	\begin{split}
	\langle \rho^\star_{t},  H_{t}\rangle -\langle g,  H_{0}\rangle &+ \int_{0}^{t}\< \partial_{u}\left(\rho^{\star}_{s}\right)^{m}, \partial_{u}H_s \>\, ds - \int_{0}^{t}\< \rho_{s}^{\star}, \partial_{s}H_s \>\, ds=0,
	\end{split}   
	\end{equation*}
	for all $t\in[0,T]$ and all $H\in C^{1,2}([0,T]\times[0,1])$.  Therefore, performing an integration by parts (see Lemma \ref{IIP}) in last expression, we obtain
	\begin{equation}\label{integral2}
	\begin{split}
	\langle \rho^{\star}_{t},  H_{t}\rangle -\langle g,  H_{0}\rangle -& \int_0^t\langle \rho_{s}^\star,( \partial_s H_{s}+(\rho_s^\star)^{m-1} \Delta H_s ) \rangle \, ds\\
	+&\int^{t}_{0}  \big\{ (\rho_{s}^\star(1))^m \partial_u H_{s}(1)-(\rho_{s}^\star(0))^m  \partial_u H_{s}(0) \big\} \, ds = 0,		\end{split}   
	\end{equation}
	for all $t\in[0,T]$ and all $H\in C^{1,2}([0,T]\times[0,1])$, which proves the result in the Neumann case. To finish the proof in the Dirichlet case, note that  since Propositions \ref{prop5} and \ref{prop4}  do not impose any restriction on $\{\kappa_j\}_{j\in\mathbb N}$, we can do exactly the same argument to treat  \eqref{eqq3.2} just by taking   $H\in C^{1,2}_0([0,T]\times[0,1])$. From this, we get an integral equation similar to \eqref{integral2} with test functions  $H\in C^{1,2}_0([0,T]\times[0,1])$, and from Proposition \ref{prop:3.5}, it becomes the integral equation of the Dirichlet case, see \eqref{eq:Dirichlet integral}. Since any limit point $\rho^\star$ of  the set $\{\rho^\kappa:\kappa>0\}$  is a weak solution of the PME either with Neumann boundary conditions (when we take the limit $\kappa\to 0$) or Dirichlet boundary conditions  (when we take the limit $\kappa\to \infty$) and  these weak solutions are unique, the convergence follows.  
\end{proof}

From now on, until the end of this section, we present the auxiliary results used above on the proof of Theorem \ref{main theorem}.

	\begin{proposition}\label{prop2}
Fix $m\in\mathbb N$. If $\rho^{\star}$ is a limit point in $L^{2}([0,T]\times[0,1])$ of  $\{\rho^{\kappa}: \, \kappa >0 \}$, then $(\rho^{\star})^{m}\in L^{2}(0,T; \mathcal{H}^1)$.
	\end{proposition}
	\begin{proof}
		By the hypothesis, there exists  $\{\kappa_{j}\}_{j\in\mathbb{N}}$ such that 
		$\rho^{\kappa_{j}}$ converges to $\rho^{\star}$ in $L^{2}([0,T]\times[0,1])$, when $j\to\infty$.
				From \eqref{newton_binom},  Cauchy-Schwarz's inequality  and since $0\leq \rho^{\kappa_{j}}, \rho^{\star}\leq 1$, we have 
		\begin{equation*}
		\|\left(\rho^{\kappa_j}\right)^{m} - \left(\rho^{\star}\right)^{m}\|_{L^{2}([0,T]\times[0,1])} \leq m\|\rho^{\kappa_j} - \rho^{\star} \|_{L^{2}([0,T]\times[0,1])},\end{equation*}
		which implies that $\left(\rho^{\kappa_j}\right)^{m} \to \left(\rho^{\star}\right)^{m}$ in $L^{2}([0,T]\times[0,1])$, when $j \to \infty$.

		To show that $(\rho^{\star})^{m}\in L^{2}(0,T; \mathcal{H}^1)$, we will use Lemma \ref{charact} and we will only need to consider test functions $H\in C_c^{0,\infty}([0,T]\times (0,1))$. Observe that from Theorem \ref{energy estimate consequence} and Proposition \ref{prop:2.6}, considering $H\in C_c^{0,\infty}([0,T]\times (0,1))$, we have  
		\begin{equation*}
		\dl (\rho^{\kappa_j})^m, \partial_{u}H \dr =-  \dl \partial_u(\rho^{\kappa_j})^m,H \dr ,
	\end{equation*} 
for all  $j\in\mathbb N$. Thus, from Cauchy-Schwarz's inequality,
		\begin{equation}\label{eq.3.31}
	\lim_{j\to\infty} \dl \partial_u(\rho^{\kappa_j})^m,H \dr=-\lim_{j\to\infty}	\dl (\rho^{\kappa_j})^m, \partial_{u}H \dr =- \dl (\rho^{\star})^m, \partial_{u}H \dr ,
	\end{equation}
	for all
$H\in C_c^{0,\infty}([0,T]\times (0,1))$.
	 
			On the other hand, 	from Theorem \ref{energy estimate consequence},  the set $\left\{ (\rho^\kappa)^{m}:\, \kappa > 0 \right\}$ is bounded in $L^{2}(0,T; \mathcal{H}^1)$, therefore, 
		there exists  $\{\kappa_{j_\ell}\}_{\ell\in\mathbb{N}}\subset\{\kappa_{j}\}_{j\in\mathbb{N}}$ and  $\Psi\in L^{2}([0,T]\times[0,1])$  such that  $\p_u(\rho^{\kappa_{j_\ell}})^m$ converges to $\Psi$ in $L^{2}([0,T]\times[0,1])$, when $\ell\to \infty$. 
		Thus,
		\begin{equation*}
\lim_{\ell\to\infty} \dl \partial_u(\rho^{\kappa_{j_\ell}})^m,H \dr= \dl \Psi,H \dr ,
\end{equation*}
			  	for all
			  	$H\in C_c^{0,\infty}([0,T]\times (0,1))$. Denoting $\Psi$ by $\partial_u(\rho^{\star})^m $, we have
			  	\begin{equation}\label{eq.3.3}
			  \lim_{j\to\infty} \dl \partial_u(\rho^{\kappa_j})^m,H \dr= \dl \partial_u(\rho^{\star})^m,\,H \dr ,
			  \end{equation}
			  for all
			  	$H\in C_c^{0,\infty}([0,T]\times (0,1))$.
			  Therefore, by \eqref{eq.3.31}  and \eqref{eq.3.3}, we have
			$$ \dl (\rho^{\star})^m, \partial_{u}H \dr =-  \dl \partial_u(\rho^{\star})^m,\,H \dr ,$$ for all
				$H\in C_c^{0,\infty}([0,T]\times (0,1))$. 
			Finally, from Lemma \ref{charact} we conclude that $(\rho^\star)^m\in L^2(0,T;\mathcal H^1)$.
	\end{proof}

	\begin{proposition}\label{prop5}
		Suppose that $\rho^{\kappa_{j}}$ converges to $\rho^{\star}$ in $L^{2}([0,T]\times[0,1])$, as $j \to \infty$. Then, for all $t\in[0,T]$ and for all $H \in C^{1,2}([0,T]\times[0,1])$
		$$\lim_{j\to \infty} \int_{0}^{t}\< \partial_{u} \left(\rho_{s}^{\kappa_{j}}\right)^{m},\; \partial_{u} H_s \> \, ds = \int_{0}^{t} \< \partial_{u}\left(\rho^{\star}_{s}\right)^{m},\; \partial_{u} H_s \>\, ds\,.$$
	\end{proposition}
	\begin{proof}
		
			Note that, it is enough to prove 
			$$\lim_{j\to \infty} \int_{0}^{t}\< \partial_{u} \left(\rho_{s}^{\kappa_{j}}\right)^{m}, \; G_s \> \, ds = \int_{0}^{t} \< \partial_{u}\left(\rho^{\star}_{s}\right)^{m}, \;G_s \>\, ds\,,$$
			for all $G\in C^{1,1}([0,T]\times[0,1])$ and all $t\in [0,T]$.

	By \eqref{eq.3.3}, we have
			\begin{equation}\label{eq.3.4}\lim_{j\to \infty} \int_{0}^{T}\< \partial_{u} \left(\rho_{s}^{\kappa_{j}}\right)^{m},  \;\bar G_s \> \, ds = \int_{0}^{T} \< \partial_{u}\left(\rho^{\star}_{s}\right)^{m}, \;\bar G_s \>\, ds\,,	\end{equation}
			for all $\bar G\in C_c^{0,\infty}([0,T]\times (0,1))$.
			
			Consider a continuous function $\varphi_\delta:[0,T]\to [0,1]$ which is equal to 1 in the interval  $[0,t]$ and zero in $[t+\delta,T]$. For all $G\in C_c^{1,\infty}([0,T]\times (0,1))$, the function $\bar G_s(u)= \varphi_\delta(s)\,G_s(u)$ belongs to $C_c^{0,\infty}([0,T]\times (0,1))$. By \eqref{eq.3.4}, we have
	\begin{equation*}
	\begin{split}
	&\lim_{j\to \infty} \int_{0}^{t} \< \partial_{u} \left(\rho_{s}^{\kappa_{j}}\right)^{m},  G_s \> \, \,ds+\lim_{j\to \infty} \int_{t}^{t+\delta} \< \partial_{u} \left(\rho_{s}^{\kappa_{j}}\right)^{m},  G_s \> \, \varphi_\delta(s)\, ds  \\&=\lim_{j\to \infty} \int_{0}^{T} \< \partial_{u} \left(\rho_{s}^{\kappa_{j}}\right)^{m},  G_s \> \, \varphi_\delta(s)\,ds\, = \int_{0}^{T}  \< \partial_{u}\left(\rho^{\star}_{s}\right)^{m}, G_s \>\, \varphi_\delta(s)\,ds\\
&=	\int_{0}^{t}  \< \partial_{u}\left(\rho^{\star}_{s}\right)^{m}, G_s \>\,ds+
	\int_{t}^{t+\delta}  \< \partial_{u}\left(\rho^{\star}_{s}\right)^{m}, G_s \>\, \varphi_\delta(s)\,ds\,,
	\end{split}
	\end{equation*}
	for all $G\in C_c^{1,\infty}([0,T]\times (0,1))$ and all $\delta>0$. Using Cauchy-Schwarz's inequality  and the boundedness of $\rho^\kappa$, the integrals above on the interval $[t,t+\delta]$ vanish when $\delta\to 0$. Thus,
		\begin{equation*}
	\begin{split}
	&\lim_{j\to \infty} \int_{0}^{t} \< \partial_{u} \left(\rho_{s}^{\kappa_{j}}\right)^{m},  G_s \> 
	=	\int_{0}^{t}  \< \partial_{u}\left(\rho^{\star}_{s}\right)^{m}, G_s \>\,ds,
	\end{split}
	\end{equation*}
	for all $G\in C_c^{1,\infty}([0,T]\times (0,1))$ and all $t\in[0,T]$.
By a density argument, we can extend the result above to functions $G\in C^{1,1}([0,T]\times[0,1])$. For more details, we refer the reader to Proposition $5.5$ of \cite{phase}.
	\end{proof}

	\begin{proposition}\label{prop:3.5}
		Let 	$\rho^\star=\lim_{j\to\infty}\rho^{\kappa_j}$ in $L^{2}([0,T]\times[0,1])$, where 
		$\{\rho^{\kappa_j}:\, j\in\mathbb N\}$ is   a subsequence of $\{\rho^{\kappa}: \, \kappa >0\}$ and $\kappa_j\to \infty$, when $j\to\infty$. Then, $\rho^\star$  satisfies  item $(3)$  of Definition \ref{Def. Dirichlet}, i.e.,  $\rho_{t}^\star(0)=\alpha$ and $\rho_{t}^\star(1)=\beta$, for almost every $t\in(0,T]$.
	\end{proposition}

	\begin{proof}
		We only present the proof of $\rho_{s}^\star(0)=\alpha$,  for almost every  $s\in(0,T]$, since the other boundary condition can be proved   similarly.
		In order to prove this, it is enough to show that 
		\begin{equation}\label{35.222}
		\int_0^T \Big((\rho^{\star}_s)^m(0)-\alpha^m\Big)^2\;ds=0\,.\end{equation}
		The idea behind this proof is to compare separately $(\rho^{\star}_s)^m(0)$ and $\alpha^m$ with $(\rho^{\kappa_j}_s)^m(0)$. In the next lines, for ease of notation, we fix a parameter $\kappa$, and at the end of the argument, we consider $\kappa_j$ and we will take the limit in $j\to\infty$.

We  start by studying  $\int_0^T \Big((\rho^\kappa_s)^m(0)-(\rho^\star_s)^m(0)\Big)^2\;ds$. Summing and subtracting $\frac{1}{\varepsilon}\int_0^\varepsilon(\rho^\kappa_s)^m(u)\,du$ and $\frac{1}{\varepsilon}\int_0^\varepsilon(\rho^\star_s)^m(u)\,du$, and using the inequality $(a+b+c)^2\leq 3(a^2+b^2+c^2) $, we get
		\begin{equation*}
		\begin{split}
		\int_0^T \Big((\rho^\kappa_s)^m(0)-(\rho^\star_s)^m(0)\Big)^2\;ds\;&\leq\; 3\int_0^T\Bigg((\rho^\kappa_s)^m(0)-\frac{1}{\varepsilon}\int_0^\varepsilon(\rho^\kappa_s)^m(u)\,du\Bigg)^2\,ds\\
		&+ 3\int_0^T\Bigg(\frac{1}{\varepsilon}\int_0^\varepsilon(\rho^\kappa_s)^m(u)\,du-\frac{1}{\varepsilon}\int_0^\varepsilon(\rho^\star_s)^m(u)\,du\Bigg)^2\,ds\\
		&+ 3\int_0^T\Bigg((\rho^\star_s)^m(0)-\frac{1}{\varepsilon}\int_0^\varepsilon(\rho^\star_s)^m(u)\,du\Bigg)^2\,ds\,,
		\end{split}
		\end{equation*}
		for all $\kappa>0$ and $\varepsilon>0$. Note that, from Definition \ref{L2 space weight}, \eqref{energy exp!!!!}, and Theorem \ref{energy estimate consequence} (for more details see \eqref{6.1}),  we have 
		\begin{equation}\label{35.3}
		\dl \p_u(\rho^\kappa)^m,\,\p_u(\rho^\kappa)^m \dr
		\leq 	 4c\,M_0,
		\end{equation} for all $\kappa>0$. 
		Since 
\begin{equation*}
		\begin{split} 
		\left|\,(\rho^\kappa_s)^m(0)-\frac{1}{\varepsilon}\int_0^\varepsilon(\rho^\kappa_s)^m(u)\,du\,\right| &=\left|\,\frac{1}{\varepsilon}\int_0^\varepsilon\int_0^u\partial_{v}(\rho_{s}^{k})^{m}(v)\,dv\,du\,\right| \\
		&= {\left| \frac{1}{\varepsilon}\int_{0}^{\varepsilon}\partial_{v}(\rho_{s}^{k})^{m}(v){(\varepsilon-v)}\,dv \right|} \\ 
		&\leq \sqrt{\frac{\varepsilon}{3}}\;\big\Vert\p_u(\rho^\kappa_s)^m\big\Vert_2\,,
		\end{split}
		\end{equation*}
		for all $\kappa>0$ and $\varepsilon>0$, we obtain
		\begin{equation}\label{35.4}
		\int_0^T\Bigg((\rho^\kappa_s)^m(0)-\frac{1}{\varepsilon}\int_0^\varepsilon(\rho^\kappa_s)^m(u)\,du\Bigg)^2\,ds\leq \frac{\varepsilon}{3}\, \dl \p_u(\rho^\kappa)^m,\,\p_u(\rho^\kappa)^m \dr\leq \frac{4cM_0}{3}\,\varepsilon\,,
		\end{equation}
		for all $\kappa>0$ and $\varepsilon>0$.
		
		From  Theorem \ref{energy estimate consequence},  the set $\left\{ (\rho^\kappa)^{m}:\, \kappa > 0 \right\}$ is bounded in $L^{2}(0,T; \mathcal{H}^1)$. Therefore, 
		there exist  $\{\kappa_{j_\ell}\}_{\ell\in\mathbb{N}}\subset\{\kappa_{j}\}_{j\in\mathbb{N}}$ and  $\Psi\in L^{2}([0,T]\times[0,1])$  such that  $\p_u(\rho^{\kappa_{j_\ell}})^m$ converges to $\Psi$ in $L^{2}([0,T]\times[0,1])$, when $\ell\to \infty$. As in the proof of Proposition \ref{prop2}, we denote $\Psi$ by $\p_u(\rho^{\star})^m$. Thus, taking the limit in \eqref{35.3} we obtain 
		\begin{equation*}
		\dl \p_u(\rho^\star)^m,\,\p_u(\rho^\star)^m \dr
		\leq 	 4c\,M_0\,.
		\end{equation*}By the same arguments as in \eqref{35.4}, we also get
		\begin{equation}\label{35.5}
		\int_0^T\Bigg((\rho^\star_s)^m(0)-\frac{1}{\varepsilon}\int_0^\varepsilon(\rho^\star_s)^m(u)\,du\Bigg)^2\,ds\leq \frac{\varepsilon}{3}\, \dl \p_u(\rho^\star)^m,\,\p_u(\rho^\star)^m \dr\leq \frac{4cM_0}{3}\,\varepsilon\,,
		\end{equation}
		for all $\varepsilon>0$.

		From \eqref{newton_binom}, the  fact that $\rho^\kappa$ and $\rho^\star$ are both bounded from above by one together with applying the Cauchy-Schwarz's inequality twice, first in the time integral and then in the space integral, we get
		\begin{equation}\label{35.6}
		\int_0^T\Bigg(\frac{1}{\varepsilon}\int_0^\varepsilon(\rho^\kappa_s)^m(u)\,du-\frac{1}{\varepsilon}\int_0^\varepsilon(\rho^\star_s)^m(u)\,du\Bigg)^2\,ds\leq 	m^2\sqrt{\frac{T}{\varepsilon^{3}}}\;\dl \,\rho^\kappa-\rho^\star,\;\rho^\kappa-\rho^\star\,\dr^{\sfrac{1}{2}}\,,
		\end{equation}
		for all $\kappa>0$ and $\varepsilon>0$.
		Thus,  \eqref{35.4}, \eqref{35.5} and \eqref{35.6} imply 
		\begin{equation*}
		\int_0^T \Big((\rho^\kappa_s)^m(0)-(\rho^\star_s)^m(0)\Big)^2\;ds\leq\,\,8cM_0\,\,\varepsilon\,+3{m^2}\sqrt{\frac{T}{\varepsilon^{3}}}\;\dl \,\rho^\kappa-\rho^\star,\;\rho^\kappa-\rho^\star\,\dr^{\sfrac{1}{2}}\,,\end{equation*}
		for all $\kappa>0$ and $\varepsilon>0$. By the hypothesis and taking the limit first on $\kappa_j\to \infty$ and after in $\varepsilon\to 0$, we obtain
		\begin{equation}\label{35.2}
		\lim_{\kappa_j\to \infty} \int_0^T \Big((\rho^{\kappa_j}_s)^m(0)-(\rho^\star_s)^m(0)\Big)^2\;ds=0\,.\end{equation}
		
Observe that in the limit above  it is not necessary that $\kappa_j\to \infty$.

		Now, we observe that,
		from \eqref{35.2}, in order to show \eqref{35.222}, we just need to study the integral
		$$ \int_0^T \Big((\rho^\kappa_s) ^m(0)-\alpha^m\Big)^2\;ds.$$
		Using  the equality \eqref{eq:Robin_bc1} and the fact that $\rho^\kappa$ and $\alpha$ are both bounded from above by one, last integral is bounded from above by
		\begin{equation*}
		\frac{m}{\kappa}\int_0^T \frac{P_m^\alpha(\rho^\kappa_s(0))}{\kappa}\Big(\partial_u(\rho^\kappa_s)^m(0)\Big)^2\;ds\,,
		\end{equation*} for all $\kappa>0$.
		From Proposition \ref{prop:2.6}, Definition \ref{L2 space weight} and Theorem \ref{energy estimate consequence},  last display is bounded from above by $\frac{m}{\kappa}\,\,4cM_0$. 
		Since $\kappa_j\to \infty$, when $j\to \infty$, the inequalities above imply that
		\begin{equation*}
		\lim_{\kappa_j\to \infty} \int_0^T \Big((\rho^{\kappa_j}_s)^m(0)-\alpha^m\Big)^2\;ds=0\,.\end{equation*}
		Therefore, from the limit above and the limit \eqref{35.2}, we get \eqref{35.222}.

	\end{proof}

	\begin{proposition}\label{prop3}
		For $\kappa>0$, let  $\rho^{\kappa}:[0,T]\times[0,1]\rightarrow [0,1]$ be the unique weak solution of \eqref{eq:Robin}. Then, there exists a constant $C >0$ that does not depend on $\kappa$ such that
		\begin{equation*}
		\vert\,\< \rho^\kappa_t, \, H_t\,\>\,  -\,\< \rho^\kappa_s, \, H_s\,\>\,\vert \leq C |\,t-s\,|^{1/2}\,,\qquad\forall s,t\in [0,T]\,,
		\end{equation*}
where $H \in C^{1,2} ([0,T]\times[0,1])$, for $0<\kappa \leq 1$,  and  $H \in C^{1,2}_0 ([0,T]\times[0,1])$, for $\kappa>1$.  Above, the constant $C$ depends on $T,\;H,\;m,\;\alpha,\;\beta,\;c,\;\mbox{and} \;M_0$. 
	\end{proposition}
	
	\begin{proof}
		The proof will be divided into two cases, $0<\kappa \leq 1$ and $\kappa>1$.  We observe that this  splitting  of the value of $\kappa$ is purely a choice, any other constant would play the same role and would not affect the argument. We begin by analyzing the case $0<\kappa \leq 1$. Let $H\in C^{1,2}([0,T]\times[0,1])$. Performing an integration by parts (see Lemma \ref{IIP})  in \eqref{eq:Robin integral}, we have
		\begin{equation}\label{integral eq}
		\begin{split}
		\< \rho^{\kappa}_{t}, H_t \> - \< g, H_0 \> &+ \int_{0}^{t}\< \partial_{u} (\rho^{\kappa}_{s})^{m}, \partial_{u} H_s \>\, ds - \int_{0}^{t}\< \rho^{\kappa}_s, \partial_s H_s \>\, ds \\
		&- \int_{0}^{t}\big\{\kappa(\alpha - \rho_{s}^{\kappa}(0))H_{s}(0)+\kappa(\beta- \rho_{s}^{\kappa}(1))H_{s}(1)\big\}\, ds = 0.
		\end{split}
		\end{equation}
		Now, we consider  \eqref{integral eq}  with $t$ and then $s$ and subtract one from the other, to obtain:
	\begin{equation}\label{eq:important}
	\begin{split}
		\left|\< \rho^{\kappa}_{t}, H_t \> - 	\< \rho^{\kappa}_{s}, H_s \> \right|& \leq \Big|\int_s^t\<\, \partial_u(\rho^\kappa_r )^m, \,\partial_u  H_r\>\,dr\Big| +\Big|\int_s^t\< \rho^\kappa_r, \, \partial_r H_r\,\>\, dr\Big|\\&+\Big|\int_{s}^{t}\big\{\kappa(\alpha - \rho_{r}^{\kappa}(0))H_{r}(0)+\kappa(\beta- \rho_{r}^{\kappa}(1))H_{r}(1)\big\}\, dr\Big|.
		\end{split}
		\end{equation}
In the following lines, we will estimate  each one of the  terms on the right-hand side of last display. 
		Let us begin by estimating the first term. Note that 
		\begin{equation*}
\Big|\int_s^t\<\, \partial_u(\rho^\kappa_r )^m, \,\partial_u  H_r\>\,dr\Big| \leq \Vert \partial_u H\Vert_\infty \int_s^t\int_0^1|\partial_u(\rho^\kappa_r )^m|\,du\,dr.
		\end{equation*}
	Applying the Cauchy-Schwarz's inequality twice, first in the time integral and then in the space integral, we bound the last expression from above by
			\begin{equation*}
			\begin{split}
	\Vert \partial_u H\Vert_\infty \sqrt{|t-s|}\,{\<\!\<  \partial_u(\rho^\kappa )^m,\,  \partial_u(\rho^\kappa )^m \>\!\>^{1/2}.}
			\end{split}
			\end{equation*}
		Since $\rho^{\kappa}$ is the weak solution of \eqref{eq:Robin}, from Theorem \ref{energy estimate consequence}  and Proposition \ref{prop:2.6} , the last expression becomes bounded from above by 
		\begin{equation}\label{R1.1}
		 \Vert \partial_u H\Vert_\infty \sqrt{4cM_0}\sqrt{|t-s|}\,.
		\end{equation}
		
		To estimate the second term on the right-hand side of \eqref{eq:important}, we use the fact that $0\leq \rho^\kappa \leq 1$ to get
		\begin{equation}\label{R2.1}
		\begin{split}
	\Big|\int_s^t\< \rho^\kappa_r, \, \partial_r H_r\,\>\, dr\Big| & \leq \| \partial_r H \|_{\infty} |t-s|\leq \| \partial_r H \|_{\infty} \sqrt{2T}\sqrt{|t-s|}\,.
		\end{split}
		\end{equation}

	Here we used the fact that $0<\kappa\leq 1$. Since  $|\alpha - \rho_{r}^{\kappa}(0)| \leq 2$ and $|\beta - \rho_{r}^{\kappa}(1)| \leq 2$, we can bound the last term on the right-hand side of \eqref{eq:important} by
		\begin{equation}\label{R2.1AA}
	 4\|H\|_{\infty}\sqrt{2T}\sqrt{|t-s|}\,.
		\end{equation}

			We now turn to the proof for the case $\kappa >1$. Observe that we did not impose  any restriction on $\kappa$ nor on $H$  when we obtained the  estimates \eqref{R1.1} and \eqref{R2.1}. Nevertheless, we used the fact that $\kappa\leq 1$ on the estimate \eqref{R2.1AA}. Now, since $\kappa >1$, it is even easier because we use functions $H\in C^{1,2}_0 ([0,T]\times[0,1])$ so that they vanish at the boundary. Then, the last term on the right-hand side of \eqref{eq:important} is null, proving the proposition.
	\end{proof}

	\begin{proposition}\label{prop4}
	Let 	$\rho^\star=\lim_{j\to \infty}\rho^{\kappa_j}$ in $L^{2}([0,T]\times[0,1])$, where 
		$\{\rho^{\kappa_j}:\, j\in\mathbb N\}$ is   a subsequence of $\{\rho^{\kappa}: \, \kappa >0\}$. Then, 
		there exists a subsequence $\{\kappa_{j_{l}}\}_{l \in \mathbb{N}} \subseteq \{\kappa_{j}\}_{j\in \mathbb{N}}$ such that
		$$\lim_{l \to \infty} \< \rho_{t}^{\kappa_{j_{l}}}, H_{t} \> = \< \rho^\star_t, H_t \>, $$
	for all	$ t \in [0,T]$ and $ H\in C^{0,2}([0,T]\times[0,1])$.
	
	\end{proposition}
	\begin{proof} By the hypothesis and Cauchy-Schwarz's inequality,
		\begin{equation}\label{LLL}
		\begin{split}
	\lim_{j\to \infty}\< \rho_{t}^{\kappa_{j}},\, H_{t}\>=	\< \rho_{t}^\star\!,\, \,H_{t} \>\,,
		\end{split}
		\end{equation}
		for almost every $t\in[0,T]$ and for all $ H\in C^{0,2}([0,T]\times[0,1])$.
		Unfortunately, this  is not enough since we need the previous convergence to hold for every $t\in[0,T]$.
		In order to do that we will do several steps that we describe as follows. 	
		 First, we will use Arzel\`a-Ascoli's theorem to find a subsequence $\{\kappa_{j_{l}}\}_{l\in\mathbb N}$ such  that the function $t \mapsto \lim_{l \to \infty} \< \rho_{t}^{\kappa_{j_{l}}}, H_{t} \> $ is continuous. To do that, we will need some technical steps, such as a diagonal argument and  extend  a bounded linear functional,  to guarantee that the subsequence $\{\kappa_{j_{l}}\}_{l\in\mathbb N}$ does not depend on the choice of the function $H$.  Then, we will have to prove that the function  $\lim_{l \to \infty} \< \rho_{t}^{\kappa_{j_{l}}}, H_{t} \> $  is equal to  $\< \rho^\star_t, H_t \>$, for all $t\in[0,T]$. For this step, we will use Riesz's representation theorem, and we will need to modify the function $t\mapsto \rho_{t}^\star\in L^2([0,1])$ for $t$ on a null measure set, which is armless since it produces an element equal to $\rho^\star$ in $L^{2}([0,T]\times[0,1])$).

	As we mentioned above, in the first step, we will construct a subsequence $\{\kappa_{j_{l}}\}_{l\in\mathbb N}$, which does not depend on the choice of $H$ and such that $t \mapsto \lim_{l \to \infty} \< \rho_{t}^{\kappa_{j_{l}}}, H_{t} \> $ is a continuous function.  If one wants to skip these technical steps, then can go directly to \eqref{L}. If not, fix $ H\in C^{0,2}([0,T]\times[0,1])$  and 
		for all $j\in \mathbb{N}$, we define the function 
		$$t\in[0,T]\mapsto\mathcal{F}_{j}(t,H_t) := \< \rho^{\kappa_j}_{t}, H_t \>.$$ 
	From Proposition \ref{prop3}, the sequence of functions $\{ \mathcal{F}_{j}(\cdot,H_\cdot)\}_{j\in \mathbb N} $ is uniformly  H\" older continuous, i.e.,
		\begin{equation*}
		\begin{split}
		|\mathcal{F}_{j}(t,H_t)-\mathcal{F}_{j}(s,H_s)| = |\< \rho_{t}^{\kappa_{j}}, H_t \> - \< \rho_{s}^{\kappa_{j}}, H_s \>| 
		\leq C \sqrt{|t-s|},
		\end{split}
		\end{equation*}
		and the constant $C$ depends on $H$, but it does not depend on $j$.
	This  implies that   $\{\mathcal{F}_{j}(\cdot,H_{\cdot})\}_{j\in \mathbb{N}}$  is equicontinuous.
		 Since $|\mathcal{F}_{j}(t,H_t)| \leq \|H\|_{\infty}$, for any $t\in[0,T] $ and $j\in\mathbb N$, the sequence $\{\mathcal{F}_{j}(\cdot,H_{\cdot})\}_{j\in \mathbb{N}}$ is also uniformly bounded. Hence, by  Arzel\`a-Ascoli's theorem, there exists a  continuous function $\mathcal{F}(\cdot,H_{\cdot}): [0,T] \to \mathbb{R}$ and a subsequence $\{\kappa^{H}_{j_{l}}\}_{l\in \mathbb{N}}$ of $\{\kappa_{j}\}_{j\in \mathbb{N}}$, depending on $H$, such that $\mathcal{F}_{j_{l}}(\cdot,H_{\cdot})$ converges uniformly to the continuous function $\mathcal{F}(\cdot,H_{\cdot})$, as $l \to \infty$. 
	Now, let us remove the dependence of the subsequence $\{\kappa^{H}_{j_{l}}\}_{l\in \mathbb{N}}$ on $H$ by using a diagonal argument.  Observe that  $C^{0,2}([0,T]\times[0,1])$ is a separable space, so there exists a dense countable subset $D \subseteq C^{0,2}([0,T]\times[0,1])$. Let $D=\{H^{i} : \, i\in \mathbb{N}\}$. For each $i\in \mathbb{N}$, we denote by $\{\kappa^{i}_{j_{l}}\}_{l\in\mathbb{N}}$ a subsequence of $\{\kappa_{j}\}_{j\in\mathbb{N}}$ depending on $H^{i}\in D$. Repeating the previous argument, since $\{\mathcal{F}_{j}(\cdot,H_{\cdot}^{1})\}_{j\in \mathbb{N}}$ is equicontinuous and uniformly bounded, there exist a continuous function $\mathcal{F}(\cdot,H_{\cdot}^{1})$ and a subsequence $\{\kappa^{1}_{j_{l}}\}_{l \in \mathbb{N}}$ such that $\mathcal{F}_{j_{l}}(\cdot,H_{\cdot}^{1})$ converges uniformly to $\mathcal{F}(\cdot,H_{\cdot}^{1})$, as $l \to \infty$. Now, we pick a subsequence $\left\{ \kappa_{j_{l}}^{2} \right\}_{l\in \mathbb{N}} \subseteq \left\{ \kappa_{j_{l}}^{1} \right\}_{l\in \mathbb{N}}$ such that the sequence of functions $\mathcal{F}_{j_{l}}(\cdot,H_{\cdot}^{2})$ converges uniformly to a continuous function $\mathcal{F}(\cdot,H_{\cdot}^{2})$, as $l \to \infty$. Proceeding in an inductive fashion, we obtain a subsequence $\left\{ \kappa_{j_{l}}^{q} \right\}_{l\in \mathbb{N}} \subseteq \left\{ \kappa_{j_{l}}^{q-1} \right\}_{l\in \mathbb{N}}$ such that the sequence of functions $\mathcal{F}_{j_{l}}(\cdot,H_{\cdot}^{q})$ converges uniformly to a continuous function $\mathcal{F}(\cdot,H_{\cdot}^{q})$, as $l \to \infty$. Fixing $q\in \mathbb{N}$, this construction guarantees that  
		\begin{equation*}\label{conv1}
		\mathcal{F}_{j_{l}}(t,H^{i}_t) =\< \rho_{t}^{\kappa_{j_{l}}^q}, H_{t}^{i} \> \xrightarrow[l \to \infty]{} \mathcal{F}(t,H^{i}_t),
		\end{equation*}
		for all $i\leq q$. Now, set $\kappa_{j_{l}}:=\kappa_{j_{l}}^l$ and note that by a diagonal argument the following convergence holds

		\begin{equation}\label{conv1}
		\mathcal{F}_{j_{l}}(t,H^{i}_t) =\< \rho_{t}^{\kappa_{j_{l}}}, H_{t}^{i} \> \xrightarrow[l \to \infty]{} \mathcal{F}(t,H^{i}_t), \;\; \forall i\in \mathbb{N}.
		\end{equation} 
			Observe  that the convergence above is uniform in $t\in[0,T]$, and the function $t\mapsto \mathcal{F}(t,H^{i}_t)$ is continuous, for all $ i\in \mathbb{N}$, since  it is a uniform limit of continuous functions.

	Fixing  $t\in [0,T]$,  the functional $H \mapsto \mathcal{F}_{j_{l}}(t,H_t)$ is linear in $D$.
	From the fact that $\rho^\kappa_t$ is bounded from above by one, we have $|\mathcal{F}_{j_l}(t,H_t)| \leq \|H\|_{\infty}$, for any $l\in\mathbb N$.
	Then, by  \eqref{conv1}, it is  simple to check that  $H \mapsto \mathcal{F}(t,H_t)$ is also a bounded linear functional in $D$. 
	As usual, the extension of   $H \mapsto \mathcal{F}(t,H_t)$ to $ C^{0,2}([0,T]\times[0,1])$ is 
	$\mathcal{F}(t,H_t)=\lim_{\varepsilon\to 0} \mathcal{F}(t,H_t^\varepsilon)$, where
	$\{H^\varepsilon:\,\varepsilon>0\}$ is a sequence on $D$ such that $H^\varepsilon$ converges to $H$ with respect $\Vert\cdot\Vert_\infty$, when $\varepsilon\to 0$. 
By $\mathcal{F}(t,H_t^\varepsilon)=\lim_{l\to \infty}\< \rho_{t}^{\kappa_{j_{l}}}, H_{t}^{\varepsilon} \> $, uniformly on $t\in [0,T]$,  and 
	$$\Big|\mathcal{F}(t,H_t^\varepsilon) -\lim_{l\to \infty}\< \rho_{t}^{\kappa_{j_{l}}}, H_{t} \> \Big|=\lim_{l\to \infty}\big|\< \rho_{t}^{\kappa_{j_{l}}}, \,(H_{t}^{\varepsilon} -H_{t} )\> \big|\leq ||H^{\varepsilon}-H||_{\infty},$$ for all $\varepsilon>0$ and $t\in [0,T]$, we have
	\begin{equation}\label{L}
	\begin{split}
\mathcal{F}(t,H_t)
=\lim_{\varepsilon\to 0}\mathcal{F}(t,H_t^\varepsilon)
=\lim_{l\to \infty}\< \rho_{t}^{\kappa_{j_{l}}},\, H_{t}\>,\\
	\end{split}
	\end{equation}
	for all  $t\in [0,T]$ and  $ H\in C^{0,2}([0,T]\times[0,1])$.  Observe that the subsequence $\{\kappa _{j_l}\}_{l\in\mathbb N}$ does not depend on the choice of $H$ and also that 
 the function $t\mapsto\mathcal{F}(t,H_t)$ is continuous.	
	
	Recall  that $\rho^\star=\lim_{j\to \infty}\rho^{\kappa_j}$ in $L^{2}([0,T]\times[0,1])$, then we have that 
		\begin{equation}\label{LL}
	\begin{split}
\mathcal{F}(t,H_t)=\lim_{l\to \infty}\< \rho_{t}^{\kappa_{j_{l}}},\, H_{t}\>=	\< \rho^\star_{t}\!,\, \,H_{t} \>\,,
	\end{split}
	\end{equation}
for almost every $t\in[0,T]$. Although the equality above is for almost every $t\in[0,T]$ as in  \eqref{LLL}, the limit $\lim_{l\to \infty}\< \rho_{t}^{\kappa_{j_{l}}},\, H_{t}\>$ now exists for all  $t\in[0,T]$, which was not the case in  \eqref{LLL}.
In order to show that this equality holds for every $t$, we  consider, for  $t\in[0,T]$ fixed, a linear functional $\tilde{\mathcal{F}}(t,\cdot)$ acting on functions $g\in L^2([0,1])$ as
$$\tilde{\mathcal{F}}(t,g):=\lim_{l\to \infty}\< \rho_{t}^{\kappa_{j_{l}}},\, g\>.\\
$$
By Cauchy-Schwarz's inequality and the fact that $0\leq \rho_{t}^{\kappa_{j_{l}}}\leq 1$, for all $l\in \mathbb N$, we have that $\tilde{\mathcal{F}}(t,\cdot)$ is a bounded linear functional in $L^2([0,1])$. Then, by Riesz's representation theorem  there exists $ \tilde\rho_t\in L^2([0,1])$ such that $\tilde{\mathcal{F}}(t,g)=\< \tilde \rho_{t},\, g\>$, for all $g\in L^2([0,1])$.
In particular, if  $ H\in C^{0,2}([0,T]\times[0,1])$, $H_t\in C^{2} ([0,1])\subset L^2([0,1])$, then 
\begin{equation}\label{4L}
\<\tilde  \rho_{t},\, H_t\>=\tilde{\mathcal{F}}(t,H_t)=\lim_{l\to \infty}\< \rho_{t}^{\kappa_{j_{l}}},\, H_t\>=\mathcal{F}(t,H_t),
\end{equation} for every $t\in [0,T]$.
The last equality is due to \eqref{L}. From \eqref{LL}, $\< \rho_{t}^\star,\, H_t\>=\< \tilde \rho_{t},\, H_t\>$, for almost every $t$ and all functions $ H\in C^{0,2}([0,T]\times[0,1])$, then $\rho_{t}^\star(u)=\tilde\rho_{t}(u)$, for almost every $(t,u)\in[0,T]\times[0,1]$, and this implies that $\rho^\star=\tilde\rho$ in $L^2([0,T]\times[0,1])$.
Therefore, in \eqref{4L} we can replace $\tilde\rho$ by $\rho^\star$ and we obtain
$$\lim_{l \to \infty} \< \rho_{t}^{\kappa_{j_{l}}}, H_{t} \> = \< \rho^\star_t,\, H_t \>, $$
for all	$ t \in [0,T]$ and $ H\in C^{0,2}([0,T]\times[0,1])$.

	\end{proof}

\section{Proof of Proposition \ref{prop:2.6}}\label{sec:prop:energy_consequence}

	From Lemmas \ref{dense} and  \ref{L2.2}, $\mathcal{T}_{\xi,m}^{\alpha,\beta}(\cdot)$  can be extended to $L_{\kappa,\xi}^{2}([0,T]\times[0,1]) $ and from  Riesz's representation Theorem, there exists  $\varphi$ in $L_{\kappa,\xi}^{2}([0,T]\times[0,1])$ such that
	$
	\mathcal{T}_{\xi,m}^{\alpha,\beta}(H) = \dl \varphi, H \dr^{\alpha,\beta}_{\kappa,\xi},
	$
	for all $H \in C^{0,1}([0,T]\times[0,1])$.
	Denoting $\varphi = - \partial_u \xi^m$, we proved  \eqref{III}.
	
	Now, observe that  for  $H\in C_c^{0,\infty}([0,T]\times (0,1))$,    equation \eqref{III} becomes $$ \dl \xi^m, \partial_{u}H \dr =-  \dl \partial_u\xi^m,H \dr ,$$ for any $m\in\mathbb N$, and from Lemma \ref{charact}
	we conclude that $\xi^m\in L^2(0,T;\mathcal H^1)$.

	Now we prove  \eqref{energy exp!!!!}. 
	Recall  that 
	\begin{equation*}
	\begin{split}
	\mathcal{E}_{m,\kappa,c}^{\alpha,\beta}(\xi)=\sup _{H\in C^{0,1}([0,T]\times [0,1])}\left\{ \mathcal{T}_{\xi,m}^{\alpha,\beta}(H) - c\<\!\<H, H\>\!\>_{\kappa,\xi}^{\alpha,\beta}\right\}.
	\end{split}
	\end{equation*}
	From \eqref{III}, it follows that
	\begin{equation}\label{eq:identity_new}
	\mathcal{T}_{\xi,m}^{\alpha,\beta}(H) - c\<\!\<H, H\>\!\>_{\kappa,\xi}^{\alpha,\beta} =  - \dl \partial_u\xi^m, H \dr_{\kappa,\xi}^{\alpha,\beta}-c \dl H,H \dr_{\kappa,\xi}^{\alpha,\beta}\,,
	\end{equation}	
	for all $H\in C^{0,1}([0,T]\times [0,1])$. 
	By Young's inequality, for $A>0$
	\begin{equation*}
	-\dl \partial_u\xi^m, H \dr_{\kappa,\xi}^{\alpha,\beta}\leq \frac{1}{2A} \dl \partial_u\xi^m, \partial_u\xi^m \dr_{\kappa,\xi}^{\alpha,\beta}+ \frac{A}{2} \dl H,H \dr_{\kappa,\xi}^{\alpha,\beta}\,.
	\end{equation*}
	Taking $A=2c$ in last expression, we get
	\begin{equation*}
	\mathcal{T}_{\xi,m}^{\alpha,\beta}(H) - c\<\!\<H, H\>\!\>_{\kappa,\xi}^{\alpha,\beta} \leq \frac{1}{4c}\dl \partial_u\xi^m, \partial_u\xi^m \dr_{\kappa,\xi}^{\alpha,\beta}\,,
	\end{equation*}	
	for all $H\in C^{0,1}([0,T]\times [0,1])$. Then, 
	\begin{equation*}
	\begin{split}
	\mathcal{E}_{m,\kappa,c}^{\alpha,\beta}(\xi)\leq 
	\,\frac{1}{4c} \<\!\< \partial_u\xi^m, \partial_u\xi^m \>\!\>_{\kappa,\xi}^{\alpha,\beta}\,.
	\end{split}
	\end{equation*}
	
	For the reversed inequality, from Lemma \ref{dense}, let  $\{H^\varepsilon\}_{\varepsilon >0}$ be a sequence of functions in $C^{\,0,1}([0,T]\times [0,1])$ converging to $-\frac{1}{2c}\partial_u\xi^m$ in $L_{\kappa,\xi}^2([0,T]\times [0,1])$, as $\varepsilon \to 0$. Thus, from \eqref{eq:identity_new} we have that
	\begin{equation*}
	\begin{split}
	\mathcal{E}_{m,\kappa,c}^{\alpha,\beta}(\xi)&= \sup_{H\in C^{0,1}([0,T]\times [0,1])} \left\{ -\dl \left(\partial_u\xi^m +cH \right), H \dr_{\kappa,\xi}^{\alpha,\beta}\right\} \\
	& \geq \lim_{\varepsilon \to 0}\left\{-\dl \left(\partial_{u} \xi^{m}+cH^{\varepsilon}\right), H^{\varepsilon} \dr_{\kappa,\xi}^{\alpha,\beta}\right\}.
	\end{split}
	\end{equation*}
	Note that  from hypothesis on $H^\varepsilon$, we have that the following convergence
	$$\partial_u\xi^m+cH^{\varepsilon}\xrightarrow[\varepsilon \to 0]{} \pfrac{1}{2}\partial_u\xi^m,$$
	holds in $L_{\kappa,\xi}^2([0,T]\times [0,1])$.
	Then,
	$$\lim_{\varepsilon \to 0}\left\{-\dl \left(\partial_{u} \xi^{m}+cH^{\varepsilon}\right), H^{\varepsilon} \dr_{\kappa,\xi}^{\alpha,\beta}\right\}= \frac{1}{4c}\dl \partial_u\xi^m, \partial_u\xi^m \dr_{\kappa,\xi}^{\alpha,\beta},$$proving that 	\begin{equation*}
	\begin{split}
	\mathcal{E}_{m,\kappa,c}^{\alpha,\beta}(\xi)\geq
	\,\frac{1}{4c} \<\!\< \partial_u\xi^m, \partial_u\xi^m \>\!\>_{\kappa,\xi}^{\alpha,\beta}\,.
	\end{split}
	\end{equation*}
	Thus, we proved the equality \eqref{energy exp!!!!}.
	
	Finally, we will show the boundary conditions for almost every $s\in(0,T]$.
	Using the definitions of $\mathcal{T}_{\xi,m}^{\alpha,\beta}(\cdot) $ and $\dl \cdot, \cdot \dr_{\kappa,\xi}^{\alpha,\beta}$, \eqref{III} becomes
	\begin{equation}\label{Riesz6}
	\begin{split}
	\int_{0}^{T}\big\{\langle \xi^{m}_s, \partial_u H_{s}\rangle + \langle \partial_u \xi^{m}_s,H_{s}\rangle \big\}\, ds 
	&=\int_{0}^{T} H_{s}(1)\left\{\beta^{m}- \dfrac{P_{m}^{\beta}(\xi^{m}_s(1))}{\kappa}\partial_u \xi^{m}_s(1)\right\}\, ds \\
	&-\int_{0}^{T}H_{s}(0)\left\{\alpha^{m}+\dfrac{P_{m}^{\alpha}(\xi^{m}_s(0))}{\kappa}\partial_u \xi^{m}_s(0)\right\}\, ds,\\ 
	\end{split}
	\end{equation}
	for all $H \in C^{0,1}([0,T]\times[0,1])$. {Observe that the previous identity can be extended to functions $H\in C^{0,0}([0,T]\times[0,1])$  with a space weak derivative.}
	For all $\varepsilon>0$ and for all continuous function $g:[0,T]\to [0,\infty)$, we define $H^{\varepsilon}:[0,T]\times[0,1]\to {[0,\infty)}$ by
	\begin{equation*}
	H^{\ve}(t,u)\;=\;H_{t}^{\ve}(u)\;=\;\left\{\begin{array}{ccc}
	g(t)\,	(-\tfrac{1}{\varepsilon}u+1), &  \quad \text{if}\,\,\,\,\,\,u\in[0, \varepsilon],\\
	0, &  \quad \text{if} \,\,\,\,\,\,u\in(\varepsilon,1].\\
	\end{array}
	\right.
	\end{equation*}
	Observe that $\lim_{\varepsilon\downarrow0}H^{\varepsilon}_t(u)=0$, for all $u\in(0,1)$,
	and 
	$\partial_{u}H^{\varepsilon}_t(u) =  -\frac{1}{\varepsilon}\mathbb{1}_{[0,\varepsilon)}(u)\,g(t)$, for all $u\in[0,1]$ and $t\in[0,T]$. 
	Thus, the term inside the integral on the left-hand side of \eqref{Riesz6} (replacing $H$ by $H^\varepsilon$) can be estimated as
	\begin{equation*}
	\begin{split}
	\langle \xi_t^m, \partial_{u}H^{\varepsilon}_t \rangle + \langle \partial_u\xi_{t}^m, H^{\varepsilon}_t \rangle
	&\leq \Big(-\frac{1}{\varepsilon}\int_{0}^{\ve} \xi_{t}^m(u)\,du + \int_{0}^{\ve}|\partial_u\xi_{t}^m(u)|\, du\Big)\, g(t) \,,
	\end{split}
	\end{equation*}
	for almost every $t\in[0,T]$.
	By the Cauchy-Schwarz's  inequality, the expression inside the parenthesis on the right-hand side of the previous inequality is bounded from above by
	\begin{equation*}
	\begin{split}
	-\frac{1}{\ve}\int_{0}^{\ve}\xi_{t}^m(u)\,du  + \|\partial_u\xi_t^m\|_{2}\sqrt{\ve}=-\frac{1}{\ve}\int_{0}^{\ve}\Big\{\xi_{t}^m(u)-\xi_t^m(0)\Big\}\,du -\xi_t^m(0) + \|\partial_u\xi_t^m\|_{2}\sqrt{\ve}.
	\end{split}
	\end{equation*}
	Since $\xi^m\in L^2(0,T;\mathcal H^1)$, from Cauchy-Schwarz's  inequality, we have
	\begin{equation*}
	\begin{split}&\frac{1}{\ve}\int_{0}^{\ve}\Big\{\xi_{t}^m(u)-\xi_t^m(0)\Big\}\,du
	=\frac{1}{\ve}	\int_{0}^{\ve}\p_v\xi_{t}^m(v) (\ve-v)\,dv\leq 	||\p_u\xi_{t}^m||_2	
	\sqrt{\frac{\varepsilon}{3}}.
	\end{split}
	\end{equation*}

	Therefore,
	\begin{equation*}
	\lim_{\ve \downarrow 0} \langle \xi_t^m, \partial_{u}H^{\varepsilon}_t \rangle + \langle \partial_u\xi_{t}^m, H^{\varepsilon}_t \rangle = -\xi_{t}^m(0)\,g(t),
	\end{equation*}
	for almost every$t\in [0,T]$.  In order to apply the Dominated Convergence Theorem, we observe that 
	from Cauchy-Schwarz's inequality 
{\begin{equation*}
	\begin{split}
\Big|\< \xi^{m}_s, \partial_u H_{s}^{\ve} \> + \< \partial_u\xi^{m}_s, H_{s}^{\ve} \>\Big|  &\leq \|\xi_{s}^{m}\|_{2} \|\partial_{u}H_{s}^{\varepsilon}\|_{2} + \|\partial_u\xi_{s}^{m}\|_{2} \|H_{s}^{\varepsilon}\|_{2} ,
	\end{split}
	\end{equation*} and by the definition of $H^\varepsilon$ the right-hand side of last expression is bounded from above by
	\begin{equation*}
	\begin{split}
	||g||_\infty\,	\Big\{\|\xi_{s}^{m}\|_{2}  + \|\partial_u\xi_{s}^{m}\|_{2} \Big\} <+\infty.
	\end{split}
	\end{equation*}
	Now from Dominated Convergence Theorem, replacing in \eqref{Riesz6} $H$ by $H^\varepsilon$, and taking the limit as $\varepsilon \to 0$, we get
	\begin{equation*}
	\begin{split}
	\int_0^T\Big\{\xi_s^{m}(0)-\alpha^{m} - \frac{P_{m}^{\alpha}(\xi_s(0))}{\kappa}\partial_u\xi_s^{m}(0)\Big\}\, g(s)\, ds=0,
	\end{split}
	\end{equation*}
	{for all continuous functions $g:[0,T]\to[0,\infty)$.}
	By a density argument,   we can extend the equality above for all $g\in L^1([0,T])$. Choosing a suitable $g$, we obtain 
	\begin{equation*}
	\begin{split}
	\xi_s^{m}(0)-\alpha^{m} - \frac{P_{m}^{\alpha}(\xi_s(0))}{\kappa}\partial_u\xi_s^{m}(0)=0,
	\end{split}
	\end{equation*}	
	for almost every $s\in (0,T]$. Therefore, we get the boundary condition on the left-hand side of \eqref{boundaryconditions}.

	In order to prove the  other boundary condition of \eqref{boundaryconditions}	is enough to repeat the previous argument with 
	\begin{equation*}
	G^{\ve}(t,u)\;=\;G_{t}^{\ve}(u)\;=\;\left\{\begin{array}{ccc}
	g(t)(\tfrac{1}{\varepsilon}(u-1)+1), &  \quad \text{if}\,\,\,\,\,\,u\in[1- \varepsilon,1],\\
	0, &  \quad \text{if} \,\,\,\,\,\,u\in[0,1-\varepsilon).\\
	\end{array}
	\right.
	\end{equation*}
	We leave the details to the reader.

\section{Proof of Theorem \ref{energy estimate consequence} }\label{IPS}

The goal of this section is to introduce an interacting particle system whose hydrodynamic equation is the porous medium equation with boundary conditions of Robin, Neumann, and Dirichlet type  as in Definitions \ref{Def. Robin}, \ref{Def.Neumann}, and \ref{Def. Dirichlet}, respectively. From this particle system, we can deduce the energy estimate stated in Theorem \ref{energy estimate consequence}. Before deducing the energy estimate, we introduce the model, and we recall the hydrodynamic limit proved in  \cite{BPGN}.
\subsection{The Porous Medium Model}

Let $n\in\mathbb N$ be a scaling parameter. Let $\Sigma_n = \{1,\ldots,n-1\}$ be a finite lattice of size $n-1$ that we call  bulk. We now introduce a Markov process, that we denote by $\{\eta_t; t\geq0\}$ and with state space $\Omega_{n} := \{0,1\}^{\Sigma_{n}}$. We denote by $\eta$ a configuration of particles on $\Omega_{n}$, such that for $x\in \Sigma_{n}$, $\eta(x)=0$ means that the site $x$ is empty while $\eta(x)=1$ means that the site $x$ is occupied. The dynamics of the model has two parts: a conservative dynamics acting on the bulk of the system and a non-conservative dynamics acting at the boundary points $1$ and $n-1$.   The bulk dynamics can be described by associating a Poisson clock at each bond of the form $(x,x+1)$, with $x\in\{1,\ldots,n-2\}$ and with a parameter depending on the constraints of the process. To define the dynamics at the boundary we first need to artificially add the sites $0$ and $n$ to the bulk. Hereafter,  at the left boundary (resp. right boundary) we add Poisson clocks at the bonds $(0,1)$ (resp. $(n-1,n)$) and $(1,0)$ (resp. $(n,n-1)$) with a parameter that depends on the injection or removal  rate at site $1$ (resp. $n-1$).  All these rates will be defined below.  

Let $\theta \geq 0$, $\kappa>0$, $a\in(1,2)$, $\alpha, \beta \in (0,1)$, and $m\in\mathbb N$. For $x,y,z\in \Sigma_{n}$ we denote the exchange and flip configurations by
\begin{equation*}
\eta^{x,y}(z) = 
\begin{cases}
\eta(z), \; z \ne x,y,\\
\eta(y), \; z=x,\\
\eta(x), \; z=y,
\end{cases}
\quad \text{and} \quad \eta^x(z)= 
\begin{cases}
\eta(z), \; z \ne x,\\
1-\eta(x), \; z=x.
\end{cases}
\end{equation*}
With these notations, we define the infinitesimal generator of the process acting on functions $f: \Omega_{n} \to \mathbb{R}$ as
\begin{equation}\label{eq:gen_full}
(L^{m}_{n}f)(\eta) = (L^{m}_{P}f)(\eta) + n^{a-2}(L_{S}f)(\eta) + (L_{B}f)(\eta),\end{equation}
where 
\begin{equation}\label{generator PMM}
(L^{m}_{P}f)(\eta) = \sum_{x=1}^{n-2}c^{m}_{x,x+1}(\eta)\{a_{x,x+1}(\eta)+a_{x+1,x}(\eta)\}[f(\eta^{x,x+1})-f(\eta)],
\end{equation}
is the generator of the porous medium model,
\begin{equation}\label{genrator SSEP}
(L_{S}f)(\eta) = \sum_{x=1}^{n-2}\{a_{x,x+1}(\eta)+a_{x+1,x}(\eta)\}[f(\eta^{x,x+1})-f(\eta)],
\end{equation}
is the generator of the symmetric simple exclusion process (SSEP), and 
\begin{equation}\label{generator boundary}
(L_B f)(\eta) = \tfrac{\kappa}{n^{\theta}}I^{\alpha}_{1}(\eta)[f(\eta^1)-f(\eta)] + \tfrac{\kappa}{n^{\theta}}I^{\beta}_{n-1}(\eta)[f(\eta^{n-1})-f(\eta)],
\end{equation}
is the generator of the Glauber dynamics acting at sites $1$ and $n-1$.

\noindent
Let $\eta$ be a configuration of particles on $\Omega_{n}$. For $x,y\in\{1,\ldots,n-2\}$, we define the exchange rates at the bulk by
	\begin{equation}\label{PMM rate}
	c^{m}_{x,x+1}(\eta) = \sum_{k=1}^{m}\;\;\mathop{\prod_{j=-(m-k)}^{k}}_{j\neq 0,1}\eta(x+j), \;\; 
	\end{equation}
	\begin{equation}\label{SSEP rate}
	a_{x,y}(\eta) = \eta(x)(1-\eta(y)),
	\end{equation}
and at the boundary, for $z\in\{1,n-1\}$ and $\gamma\in\{\alpha,\beta\}$
\begin{equation}\label{Boundary rates}
\begin{split}
I_{z}^{\gamma}(\eta) &=  \gamma(1-\eta(z)) + (1-\gamma)\eta(z).
\end{split}
\end{equation}

\begin{remark}
Above we used the convention
	\begin{equation}\label{convention}
	\begin{split}
	\eta(x)&=\alpha, \;\;\; \text{for} \;\;\; x \leq 0, \\
	\eta(x)&=\beta, \;\;\; \text{for} \;\;\; x \geq n.
	\end{split}
	\end{equation}
\end{remark}

Note that when $m=1$ the rate \eqref{PMM rate} is equal to 1 and \eqref{generator PMM} is exactly the generator of the SSEP. The reader can see in Figure \ref{fig.1} the dynamics of the model for the case $m=2$.

\begin{figure}[htb]
	\begin{center}
		\begin{tikzpicture}[thick, scale=0.75][h!]
		\draw [line width=1] (-7,10.5) -- (7,10.5) ; 
		\foreach \x in  {-7,-6,-5,-4,-3,-2,-1,0,1,2,3,4,5,6,7} 
		\draw[shift={(\x,10.5)},color=black, opacity=1] (0pt,0pt) -- (0pt,-4pt) node[below] {};
		\draw[] (-2.8,10.5) node[] {};
		
		\draw[] (-6,10.3) node[below] {\tiny{$1$}};
		\draw[] (-5,10.3) node[below] {\tiny{$2$}};
		\draw[] (2,10.3) node[below] {\tiny{$x$}};
		\draw[] (5,10.3)  node[below] {\tiny{$n-2$}};
		\draw[] (6,10.3) node[below] {\tiny{$n-1$}};
		
		
		
		\node[ball color=blue!60!, shape=circle, minimum size=0.1cm] at (-7.,10.73) {};
		
		\node[ball color=blue!60!, shape=circle, minimum size=0.1cm] (A) at (-7.,11.33) {};
		
		\node[shape=circle,minimum size=0.5cm] (A1) at (-7.,12) {};
		
		\node[shape=circle,minimum size=0.5cm] (B) at (-6.,11.18) {};
		
		\node[shape=circle,minimum size=0.5cm] (B1) at (-6.,12) {};
		
		\node[ball color=black!60!, shape=circle, minimum size=0.1cm] (G) at (-4,10.73) {};
		\node[shape=circle,minimum size=0.5cm] (H) at (-3,10.55) {};
		\node[ball color=black!60!, shape=circle, minimum size=0.1cm] (I) at (-2,10.73) {};
		\node[ball color=black!60!, shape=circle, minimum size=0.1cm] (J) at (-1,10.73) {};
		\node[shape=circle,minimum size=0.5cm] (K) at (0,10.73) {};
		\node[ball color=black!60!, shape=circle, minimum size=0.1cm] at (2,10.73) {};
		\node[shape=circle,minimum size=0.5cm] (L) at (1,10.73) {};
		\node[shape=circle,minimum size=0.5cm] (M) at (3,10.73) {};
		\node[ball color=black!60!, shape=circle, minimum size=0.1cm] (N) at (2,10.73) {};
		
		\node[shape=circle,minimum size=0.5cm] (F) at (5.,10.73) {};
		\node[shape=circle,minimum size=0.5cm] (C) at (6.,10.73) {};
		\node[shape=circle,minimum size=0.5cm] (C1) at (6,11.68) {};
		\node[shape=circle,minimum size=0.5cm] (E) at (6.,11.18) {};
		\node[ball color=black!60!, shape=circle, minimum size=0.015cm] (E) at (6,10.73) {};
		\node[ball color=blue!60!, shape=circle, minimum size=0.015cm] () at (7.,10.73) {};
		\node[shape=circle,minimum size=0.5cm] (D) at (7.,10.73) {};
		\node[shape=circle,minimum size=0.5cm] (D1) at (7.,11.68) {};
		
		\path [->] (A) edge[bend left =60] node[above] {\tiny{$\frac{\kappa\alpha}{n^\theta}$}} (B);    
		\path [->] (C) edge[bend left=60] node[above] {\tiny{$\frac{\kappa(1-\beta)}{n^\theta}$}} (D); 
		\path [->] (G) edge[bend left=60] node[above] {\tiny{$1\textcolor{black}{+\frac{n^{a}}{n^{2}}}$}} (H);
		\path [->] (I) edge[bend right=60] node[above] {\tiny{$2\textcolor{black}{+\frac{n^{a}}{n^{2}}}$}} (H);
		\path [->] (J) edge[bend left=60] node[above] {\tiny{$1\textcolor{black}{+\frac{n^{a}}{n^{2}}}$}} (K);
		\path [->] (N) edge[bend left=60] node[above] {\tiny{$\textcolor{black}{\frac{n^{a}}{n^{2}}}$}} (M);
		\path [->] (N) edge[bend right=60] node[above] {\tiny{$\textcolor{black}{\frac{n^{a}}{n^{2}}}$}} (L);
		\path [->] (C) edge[bend right=60] node[above] {\tiny{$\beta+\frac{n^{a}}{n^2}$}} (F);
		
		\tclock{-13}{19.3}{gray}{0.5}{0}
		\tclock{-7}{19.3}{gray}{0.5}{0}
		\tclock{-5}{19.3}{gray}{0.5}{0}
		\tclock{-1}{19.3}{gray}{0.5}{0}
		\tclock{3}{19.3}{gray}{0.5}{0}
		\tclock{5}{19.3}{gray}{0.5}{0}
		\tclock{13}{19.3}{gray}{0.5}{0}
		\tclock{11}{19.3}{gray}{0.5}{0}

		\end{tikzpicture}
		\bigskip
		\caption{The porous medium model with slow reservoirs (with $m=2$)}\label{fig.1}
	\end{center}
\end{figure}
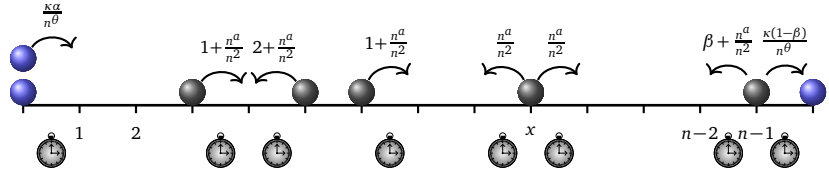

\noindent Let $\eta \in \Omega_{n}$ and $\alpha,\beta \in (0,1)$. For every $x\in\{1,\ldots,n-2\}$, the instantaneous current associated to the bond $(x,x+1)$ is given by
\begin{equation}\label{eq:current}
j^{m}_{x,x+1}(\eta) = \tau_{x}h^{m}(\eta) - \tau_{x+1}h^{m}(\eta),
\end{equation}
where
\begin{equation}\label{shift}
\tau_{x}h^{m}(\eta) = \sum_{k=1}^{m}\;\; \prod_{j=-(m-k)}^{k-1}\eta(x+j) - \sum_{k=1}^{m-1}\;\; \mathop{\prod_{j=-(m-k)}^{k}}_{j \neq 0}\eta(x+j) +n^{a-2}\eta(x). 
\end{equation}
Observe that from the convention in \eqref{convention}, for $x=1$ (resp. $n-1$), we have
\begin{equation}\label{shift-bd}
\begin{split}
\tau_{1}h^{m}(\eta) &= \sum_{k=0}^{m-1}\alpha^{k}\; \prod_{j=1}^{m-k}\eta(j) - \sum_{k=1}^{m-1}\alpha^{k}\; \prod_{j=2}^{m+1-k} \eta(j)+n^{a-2}\eta(1),\\
\tau_{n-1}h^{m}(\eta) &= \sum_{k=0}^{m-1}\beta^{k}\; \prod_{j=1}^{m-k}\eta(n-j) - \sum_{k=1}^{m-1}\beta^{k}\; \prod_{j=2}^{m+1-k} \eta(n-j)+n^{a-2}\eta(n-1).
\end{split}
\end{equation}
\begin{remark}\label{remark:exclusion terms}
The identities above share a term of the form $n^{a-2}\eta(x)$. These terms come from the SSEP dynamics, and since they  both vanish, as $n\to \infty$, from here on, we ignore them  and we  only look  at the contribution of the  remaining terms.
\end{remark}

Let $\{\eta_{tn^2};t\geq 0\}$ be the Markov process speeded up in the diffusive time scale $tn^2$ and with  infinitesimal generator $n^2L^{m}_{n}$.

\subsection{Hydrodynamic limit}
Let us begin this subsection by introducing the empirical measure associated to the process $\{\eta_{tn^2};t\geq 0\}$. For $\eta \in \Omega_{n}$, this measure gives weight $1/n$ to each particle
\begin{equation*}
\pi^{n}(\eta, du) := \frac{1}{n}\sum_{x\in \Sigma_n}\eta(x)\delta_{x/n}(du),
\end{equation*}
where $\delta_u$ is a Dirac mass on $u\in[0,1]$. In order to analyze the temporal evolution of the empirical measure, we define $\pi^{n}_{t}(\eta,du) := \pi^{n}(\eta_{tn^2},du)$. For a test function $G:[0,1]\to \mathbb{R}$, we denote the integral of $G$ with respect to  $\pi_{t}^{n}$, by $\<\pi^{n}_{t},G\>$, which is equal to 
\begin{equation}\label{eq:emp_mea}
\textcolor{black}{\< \pi^{n}_{t},G\>} = \frac{1}{n}\sum_{x\in \Sigma_n}G\left(\tfrac{x}{n}\right)\eta_{tn^2}(x).
\end{equation}
\textcolor{black}{Note that the notation $\< \cdot, \cdot \>$ above is not related to the inner product in $L^{2}([0,1])$.}
Fix $T>0$ and $\theta \geq 0$. Let $\mu_n$ be a probability measure in $\Omega_{n}$. We denote by $\mathcal{D}([0,T], \Omega_{n})$ the Skorokhod space, that is, the space of c\`adl\`ag trajectories. We denote by $\mathbb{P}_{\mu_n}$ the probability measure on the space $\mathcal{D}([0,T], \Omega_{n})$ induced by the accelerated Markov process$\{\eta_{tn^2}; t\geq 0\}$and the initial measure $\mu_n$. The corresponding expectation is denoted by $\mathbb{E}_{\mu_n}$. Let $\mathcal{M}_{+}$ be the space of positive measures on $[0,1]$ with total mass bounded by $1$ and equipped with the weak topology. We denote by $\{\mathbb{Q}_n\}_{n\in\mathbb{N}}$ the sequence of probability measures on $\mathcal{D}([0,T], \mathcal{M}_{+})$, {defined as $\mathbb{Q}_n=\mathbb{P}_{\mu_n}(\pi_{t}^{n})^{-1}$}. The corresponding expectation is denoted by $\mathbb{E}_{n}$. 

In order to state the hydrodynamic limit for$\{\eta_{tn^2}; t\geq 0\}$presented in \cite{BPGN}, we need to impose some conditions on the initial distribution of the process. Let $g: [0,1]\rightarrow[0,1]$ be a measurable function and $\{\mu_n\}_{n\in \mathbb{N}}$ be a sequence of probability measures on $\Omega_{n}$ such that for any continuous function $G:[0,1]\rightarrow \mathbb{R}$ and any $\delta > 0$ 
	\begin{equation}\label{eq:associated}
	\lim _{n\to \infty } \mu _n\Bigg( \eta \in \Omega_n : \Bigg| \< \pi^{n},G\> - \int_{0}^{1} G(u)g(u)\,du \, \Bigg|    > \delta \Bigg)= 0
	\end{equation}
and in this case we say that $\mu_n$ is associated with $g$. The aim of the hydrodynamic limit is to show that the empirical measure $\pi^{n}_{\cdot}$ converges in probability, with respect to $\mathbb P_{\mu_n}$, when $n\to \infty$, to a deterministic trajectory of measures $\pi_{\cdot}$, such that for each $t$, $\pi_t(du)$  is absolutely continuous with respect to the Lebesgue measure, that is, $\pi_{t}(du) = \rho_t(u)\,du$ and $\rho_t$ is the weak solution of the corresponding partial differential equation with certain boundary conditions and with initial condition $g$.   The precise statement is given in the next theorem.
\begin{theosn}[\cite{BPGN}]\label{hydrodynamic limit}
	Let $g:[0,1]\rightarrow[0,1]$ be a measurable function and $\lbrace\mu _n\rbrace_{n \in \mathbb{N}}$  a sequence of probability measures on $\Omega_n$ associated with $g$.  Then, for any $t \in [0,T]$ and any $\delta>0$,
	\begin{equation*}\label{limHidreform}
	\lim_{n \to \infty}\mathbb{P}_{\mu_n}\Bigg( \eta_{\cdot} \in \mathcal{D}([0,T], \Omega_n): \Bigg| \< \pi^{n}_{t},G\> - \int_{0}^{1} G(u)\rho_t(u)\,du  \,  \Bigg| > \delta \Bigg)=0,
	\end{equation*}
	where
	\begin{itemize}
		\item[$\bullet$] $\rho_{t}(\cdot)$ is a weak solution of \eqref{eq:Dirichlet}, for \textcolor{black}{$0 \leq \theta<1$};
		\item[$\bullet$] $\rho_{t}(\cdot)$ is a weak solution of (\ref{eq:Robin}), for $\theta =1$, and in the present paper it is denoted by $\rho^\kappa_t(\cdot)$;
		\item[$\bullet$] $\rho_{t}(\cdot)$ is a \textcolor{black}{weak solution of (\ref{eq:Neumann}),  for $\theta >1$}.
	\end{itemize}
\end{theosn}

In fact the theorem above is a corollary of the next result also proved in \cite{BPGN}:

\begin{propsn}[\cite{BPGN}]
The sequence of probability measures $\{\bb Q_{n}\}_{n\in \mathbb N}$  converges weakly to $\bb Q$, when $n\to \infty$, where $\bb Q$ is a Delta of Dirac measure on top of the trajectory of measures that are absolutely continuous with respect to the Lebesgue measure, i.e., $\pi_{t}(du) = \rho_t(u)du$ and  the density $\rho_t$ is the unique weak solution of the corresponding hydrodynamic equation. 
\end{propsn}

\subsection{Proof of Theorem \ref{energy estimate consequence}}
\label{sec: energy}

In the beginning of this subsection, we present the proof of Theorem \ref{energy estimate consequence} and then we state our  main result on energy estimates, which we prove by using interacting particle systems' tools, see Proposition \ref{prop:energy_estimate}. We dedicated a whole subsection to prove the proposition, where we recall some results from \cite{BPGN}.

\begin{proof}[Proof of Theorem \ref{energy estimate consequence}]
The proof of Theorem \ref{energy estimate consequence} is a simple consequence of the next proposition, which states that  the expectation with respect to $\mathbb Q$ of $\mathcal E_{m,\kappa,c}^{\alpha, \beta}(\rho^\kappa)$ is bounded.  To remove the expectation in \eqref{estimate1}, we use the  last proposition,
 and then, the  first result in Theorem \ref{energy estimate consequence} follows. Now, for \eqref{eq:Robin_bc}  we start  by observing that $(\rho^\kappa)^m\in L^2([0,T]\times [0,1])$. Moreover, since $\mathcal E_{m,\kappa,c}^{\alpha, \beta}(\rho^\kappa)	<\infty$,  from \eqref{boundaryconditions} the identities in \eqref{eq:Robin_bc} follow. Finally, to prove the boundedness of  $\left\{ (\rho^\kappa)^{m}:\, \kappa > 0 \right\}$ in $L^{2}(0,T; \mathcal{H}^1)$, we argue as follows. By Definition \ref{L2 space weight},  \eqref{energy exp!!!!} and \eqref{energy estimate good}, it holds   that
	\begin{equation}\label{6.1}
	\begin{split}
	\dl \p_u(\rho^\kappa)^m,\,\p_u(\rho^\kappa)^m \dr
	&\leq 	\dl \p_u(\rho^\kappa)^m,\,\p_u(\rho^\kappa)^m \dr_{\kappa,\rho^\kappa}^{\alpha,\beta}\\
	&=	\,4c\,\mathcal E_{m,\kappa,c}^{\alpha, \beta}(\rho^\kappa)	\leq 4c\,M_0,\\
	\end{split}
	\end{equation}
	for all $\kappa>0$.
	Therefore, from the definition of the norm in $L^{2}(0,T;\mathcal{H}^1)$ given in \eqref{sobolev norm 2}, and the fact that $0 \leq (\rho^\kappa)^m \leq 1$,  we have
	\begin{equation*}
	\begin{split}
	\| (\rho^\kappa)^m \|_{L^{2}(0,T;\mathcal{H}^1)}^{2} 
	&= \dl (\rho^\kappa)^m, (\rho^\kappa)^m \dr + \dl \partial_u (\rho^\kappa)^m, \partial_u (\rho^\kappa)^m \,\dr 
	\leq T+ 4cM_{0},
	\end{split}
	\end{equation*}
	concluding the proof.
To finish, we are only left to prove the next proposition. 
	
\end{proof}

\begin{proposition}\label{prop:energy_estimate}
	For any $\kappa >0$ and any $m\in \mathbb N$, there exist constants $M_0$ and $c$, that do not depend on $\kappa$, such that 
	\begin{equation}\label{estimate1}
	\mathbb{E}_{\bb Q}\left[ \mathcal E_{m,\kappa,c}^{\alpha, \beta}(\rho^\kappa)\right] \,\leq \, M_0 < \infty,
	\end{equation}
	where $\rho^\kappa$ is the unique weak  solution  of \eqref{eq:Robin},  $\bb Q$ is a limit point of $\bb Q_{n}$ and $\mathcal E_{m,\kappa,c}^{\alpha, \beta} $ is defined in \eqref{energy_functional}.
\end{proposition}

\begin{remark}\label{energy estimate compact}
 If we restrict the supremum in the definiton of $\mathcal E_{m,\kappa,c}^{\alpha, \beta}$, see \eqref{energy_functional}, to functions $H\in C_{c}^{0,1}([0,T]\times(0,1))$, the statement of Proposition \ref{prop:energy_estimate} reduces to 
	$$	\mathbb{E}_{\bb Q}\Big[ \sup _{H\in C_c^{0,1}([0,T]\times(0,1))} \Big\{ \dl (\rho^{\kappa})^m, \partial_u H \dr - c\dl H,H \dr \Big\} \Big]\,\leq \, M_0 < \infty,$$
which is exactly the energy estimate (for $m=2$) stated in Proposition $6.1$ of \cite{BPGN}.
\end{remark}

\subsubsection{Proof of Proposition \ref{prop:energy_estimate}}
In order to prove the proposition we observe that in Proposition $6.1$ of \cite{BPGN} the supremum was restricted to functions $H\in C_{c}^{0,1}([0,T]\times(0,1))$. In our case now, we have to consider functions $H\in C^{0,1}([0,T]\times[0,1])$ in order to have information about the boundary behavior of $\rho^\kappa$. 
To make the article self-contained, we go quickly over the proof and sketch the arguments for the case related to the bulk, and then we explain carefully how to proceed with the boundary. 

	For simplicity of the presentation, we will present the proof for even $m$ since the proof for $m$ odd is analogous and, when necessary, we explain the changes for the case $m$ odd. We begin by noticing that the space $C^{0,1}\left([0,T]\times[0,1]\right)$ is separable with respect to the norm $\|H\|_{\infty} + \|\partial_u H\|_{\infty}$. Thus, it is enough to restrict the supremum inside the expectation in the statement of the proposition to functions $H$ on a countable dense subset $\lbrace  H^{q}\rbrace_{q \in \mathbb{N}}$ of $C^{0,1}([0,T]\times[0,1])$. In addition, since $$\max_{l \leq q}\left\{ \mathcal{T}^{\alpha,\beta}_{\rho^\kappa,m}(H^l)-c\dl H^l,H^l \dr^{\alpha,\beta}_{\kappa,\rho^\kappa}  \right\} \;\; \uparrow \sup_{H\in C^{0,1}([0,T]\times[0,1])} \left\{ \mathcal{T}^{\alpha,\beta}_{\rho^\kappa,m}(H)-c \dl H,H \dr^{\alpha,\beta}_{\kappa,\rho^\kappa} \right\},$$
as $l \to \infty$, then by Monotone Convergence Theorem 
	$$\mathbb{E}_{\mathbb{Q}}\left[ \max_{l \leq q}\left\{ \mathcal{T}^{\alpha,\beta}_{\rho^\kappa,m}(H^l)-c\dl H^l,H^l \dr^{\alpha,\beta}_{\kappa,\rho^\kappa}  \right\}\right] \rightarrow \mathbb{E}_{\mathbb{Q}}\left[\sup_{H\in C^{0,1}([0,T]\times[0,1])} \left\{ \mathcal{T}^{\alpha,\beta}_{\rho^\kappa,m}(H)-c \dl H,H \dr^{\alpha,\beta}_{\kappa,\rho^\kappa} \right\}\right],$$
	when $q \to \infty$. Therefore, we are left to show that 
	
	\begin{equation}\label{estimate2}
	\mathbb{E}_{\mathbb{Q}}\left[ \max_{l \leq q}\left\{ \mathcal{T}^{\alpha,\beta}_{\rho^\kappa,m}(H^l)-c\dl H^l,H^l \dr^{\alpha,\beta}_{\kappa,\rho^\kappa} \right\} \right] \leq M_{0},
	\end{equation} 
	for any $q$ and for some $M_{0}$ independent from $q$ and $\kappa$. From \eqref{weighted norm} and \eqref{functional}, we have   
		\begin{equation}\label{term0}
		\begin{split}
		\mathcal{T}^{\alpha,\beta}_{\rho^\kappa,m}(H^l)-c\dl H^l,H^l &\dr^{\alpha,\beta}_{\kappa,\rho^\kappa}=\int_0^T \Bigg\{ \int_0^{1} \left(\rho_{s}^{\kappa}\right)^{m}(u) \;\p_u H^l_s(u) \,\,du +\alpha^mH^{l}_s(0)-\beta^mH^{l}_s(1) \\
		-\,&c\int_{0}^{1}(H^l_s(u))^2du\, -\, \dfrac{c}{\kappa}P^{\alpha}_{m}\left(\rho_{s}^{\kappa}(0)\right)(H^{l}_s(0))^2 \,-\, \dfrac{c}{\kappa}P^{\beta}_{m}\left(\rho^{\kappa}_{s}(1) \right)(H_s^{l}(1))^{2}\Bigg\}\,ds.
		\end{split}
		\end{equation}

Now we consider two approximations of the identity, for fixed $u\in[0,1]$, which are given on $v\in[0,1]$ by
	$\overleftarrow{\iota}^{u}_{\varepsilon}(v)= \dfrac{1}{\varepsilon}1_{(u-\varepsilon,u]}(v) $ and $\overrightarrow{\iota}^{u}_{\varepsilon}(v)= \dfrac{1}{\varepsilon}1_{[u,u+\varepsilon)}(v).$
We use the notation
	\begin{equation}\label{approx identity}
	\langle \rho_{s}^\kappa,  \overleftarrow{\iota}^{u}_{\varepsilon}\rangle = \dfrac{1}{\varepsilon}\int_{u-\varepsilon}^{u}\rho_{s}^\kappa(v)\,dv \;\;\; \text{and} \;\;\; \langle \rho_{s}^\kappa,  \overrightarrow{\iota}^{u}_{\varepsilon}\rangle = \dfrac{1}{\varepsilon}\int_{u}^{u+\varepsilon}\rho_{s}^\kappa(v)\,dv. 
	\end{equation}

	Since  $0\leq \rho_s^\kappa(\cdot) \leq 1$ and Lebesgue  differentiation Theorem,  
	\begin{equation}\label{lebesgue diff}
	  \lim_{\varepsilon \to 0}|\rho_{s}^\kappa(u)- \langle \rho_{s}^\kappa,  \overleftarrow{\iota}^{u-i\varepsilon}_{\varepsilon}\rangle | =0 \;\;\; \text{and} \;\;\; \lim_{\varepsilon \to 0}|\rho_{s}^\kappa(u)- \langle \rho_{s}^\kappa,  \overrightarrow{\iota}^{u+i\varepsilon}_{\varepsilon}\rangle |=0,
	\end{equation} for  almost every $u\in[0,1]$,
for any $i\in\{0,1,\ldots,m-1\}$. 
We can conclude, for $m$ even, that \begin{equation*}\label{term1}
	\left|\int_0^T \int_0^{1}(\rho_{s}^{\kappa})^{m}(u)\;\p_u H^l_s(u) \,\,du\,ds - \int_0^T \int_{\ve\frac {m}{ 2}}^{1-{\ve\frac {m}{ 2}}}\prod_{i=0}^{\tfrac{m}{2}-1}\< \rho_{s}^{\kappa},\overleftarrow{\iota_{\varepsilon}}^{u-i\varepsilon} \>\prod_{i=0}^{\tfrac{m}{2}-1}\< \rho_{s}^{\kappa},\overrightarrow{\iota}^{u+i\varepsilon}_{\varepsilon} \>\p_u H^l_s(u) \,\,du\,ds\right| 
	\end{equation*}
vanishes when $\ve \to 0$.  For $m$ odd, we would replace the previous display by 
\begin{equation*}\label{term1_new}
	\left|\int_0^T \int_0^{1}(\rho_{s}^{\kappa})^{m}(u)\;\p_u H^l_s(u) \,\,du\,ds - \int_0^T \int_{\ve\frac {m+1}{ 2}}^{1-{\ve\frac {m-1}{ 2}}}\prod_{i=0}^{\tfrac{m+1}{2}-1}\< \rho_{s}^{\kappa},\overleftarrow{\iota_{\varepsilon}}^{u-i\varepsilon} \>\prod_{i=0}^{\tfrac{m-1}{2}-1}\< \rho_{s}^{\kappa},\overrightarrow{\iota}^{u+i\varepsilon}_{\varepsilon} \>\p_u H^l_s(u) \,\,du\,ds\right|.
	\end{equation*} 
	Observe that in the last two displays we changed the function $(\rho_s^\kappa)^m$ by a choice of products of  the form \eqref{approx identity}, and despite not being the obvious change, it will be useful when we move to the microscopic system due to the result of Theorem \ref{replacement bulk}. 
For the terms with $P^\alpha_m(\cdot)$ and $P^\beta_m(\cdot)$ in \eqref{term0}, we need the result stated in \eqref{lebesgue diff} for the boundary points $u=0$ and $u=1$,  and this does not come for free from  \eqref{lebesgue diff}. In order to overcome this, we need to derive a stronger statement as in  Lemma \ref{L_bound}. We claim that, by successively applying Lemma \ref{L_bound}, the next display vanishes when $\ve \to 0$:
\begin{equation*}\label{term2}
\left|\int_{0}^{T} \dfrac{c}{\kappa}P_{m}^{\alpha}\left( \rho_{s}^{k}(0)\right)(H^{l}_s(0))^2\,ds - \int_{0}^{T}\dfrac{c}{\kappa}\sum_{i=0}^{m-1}\alpha^{m-1-i}\prod_{j=0}^{i-1}\langle \rho^\kappa _{s},  \overrightarrow{\iota}^{j\varepsilon}_{\varepsilon}\rangle (H^{l}_s(0))^2\, ds\right|
\end{equation*} 
Observe that by the definition of $P^\alpha_m(\cdot)$ given in \eqref{p}, the product above is understood as being equal to one for $i=0$. The other term with $P^\beta_m(\cdot)$ in \eqref{term0} is similar.

To prove the claim it is enough to note that last display can be bounded from above by a constant times
\begin{equation*}
\int_{0}^{T} \dfrac{c}{\kappa}\sum_{i=0}^{m-1}\alpha^{m-1-i}\left|(\rho^\kappa_s(0))^i-\prod_{j=0}^{i-1}\langle \rho^\kappa _{s},  \overrightarrow{\iota}^{j\varepsilon}_{\varepsilon}\rangle\right|\, ds.
\end{equation*} 
By summing and subtracting proper terms of the form $\rho^\kappa_s(0))^p\prod_{j=0}^{i-1-p}\langle \rho^\kappa _{s},  \overrightarrow{\iota}^{j\varepsilon}_{\varepsilon}\rangle$, for $p\in\{1,\cdots, i-1\}$, which are all bounded from above by a constant, the last display is bounded from above by  a constant times
\begin{equation*}
\int_{0}^{T} \sum_{j=0}^{m-1}\left|\rho^\kappa_s(0)-\langle \rho^\kappa _{s},  \overrightarrow{\iota}^{j\varepsilon}_{\varepsilon}\rangle\right|\, ds,
\end{equation*} 
and this vanishes by Lemma \ref{L_bound}, as $\ve\to 0$.

To treat the boundary term at the right-hand side of  \eqref{term0}, we can do exactly the same argument as we did to control the left boundary term. For simplicity of the presentation, we just present the arguments for the left boundary, but for the right, it is completely analogous. From here on, we neglect all the contributions from the right boundary. 

From previous results, we are left to prove that there exist constants $c>0$ and $M_0>0$ such that
\begin{equation}\label{a_provar_1}
\begin{split}
\varlimsup_{\ve\to 0}\mathbb{E}_{\bb Q}\Bigg[ \max _{ l\leq q} \Bigg\{\int_0^T \Bigg( \int_{\ve\frac {m}{ 2}}^{1-{\ve\frac {m}{ 2}}}\prod_{i=0}^{\tfrac{m}{2}-1}&\< \rho_{s}^{\kappa},\overleftarrow{\iota_{\varepsilon}}^{u-i\varepsilon} \>\prod_{i=0}^{\tfrac{m}{2}-1}\< \rho_{s}^{\kappa},\overrightarrow{\iota}^{u+i\varepsilon}_{\varepsilon} \>\p_u H^l_s(u) \,\,du +\alpha^mH^{l}_s(0) \\
-c&\int_{0}^{1}(H^l_s(u))^2du - \dfrac{c}{\kappa}\sum_{i=0}^{m-1}\alpha^{m-1-i}\prod_{j=0}^{i-1}\langle \rho^\kappa _{s},  \overrightarrow{\iota}^{j\varepsilon}_{\varepsilon}\rangle(H^{l}_s(0))^2\Bigg) ds\Bigg\}\Bigg]  \, \leq \, M_{0}.
\end{split}
\end{equation}
Now, we define the application $\Phi: \mathcal{D}\left([0,T],\mathcal{M}_{+}\right) \to \mathbb{R}$ by
\begin{equation*}\label{a_provar_2}
\begin{split}
\Phi(\pi_{\cdot}) = \max _{ l\leq q} \Bigg\{ \int_0^T \Bigg(\int_{\ve\frac {m}{ 2}}^{1-{\ve\frac {m}{ 2}}}&\
\prod_{i=0}^{\tfrac{m}{2}-1}\< \pi_{s},\overleftarrow{\iota_{\varepsilon}}^{u-i\varepsilon} \>\prod_{i=0}^{\tfrac{m}{2}-1}\< \pi_{s},\overrightarrow{\iota}^{u+i\varepsilon}_{\varepsilon} \>
\;\p_u H^l_s(u) \,\,du +\alpha^mH^{l}_s(0)\\
-c&\int_{0}^{1}(H^l_s(u))^{2}du - \dfrac{c}{\kappa}\sum_{i=0}^{m-1}\alpha^{m-1-i}\prod_{j=0}^{i-1}\langle \pi_{s},  \overrightarrow{\iota}^{j\varepsilon}_{\varepsilon}\rangle(H_s^{l}(0))^2 \Bigg)\,ds\Bigg\}.
\end{split}
\end{equation*}

The function $\Phi$  is  lower semi-continuous and bounded with respect to the Skorokhod topology of $ \mc D([0,T], \mc M_+)$. Therefore, recalling that $\mathbb{E}_{n}$ is the expectation with respect to the measure $\mathbb Q_n$, we can bound the expectation in \eqref{a_provar_1} from above by
\begin{equation*}\label{a_provar_3}
\begin{split}
\varliminf_{n\rightarrow +\infty}\mathbb{E}_{n}\Bigg[ \max _{l \leq q} &\Bigg\{ \int_0^T \Bigg(\int_{\ve\frac {m}{ 2}}^{1-{\ve\frac {m}{ 2}}}\prod_{i=0}^{\tfrac{m}{2}-1}\< \pi^n_s,\overleftarrow{\iota_{\varepsilon}}^{u-i\varepsilon} \>\prod_{i=0}^{\tfrac{m}{2}-1}\< \pi^n_s,\overrightarrow{\iota}^{u+i\varepsilon}_{\varepsilon} \>\;\p_u H^l_s(u) \,\,du +\alpha^mH^{l}_s(0) \\
-c&\int_{0}^{1}(H^l_s(u))^2du - \dfrac{c}{\kappa}\sum_{i=0}^{m-1}\alpha^{m-1-i}\prod_{j=0}^{i-1}\langle \pi^n _{s},  \overrightarrow{\iota}^{j\varepsilon}_{\varepsilon}\rangle(H_s^{l}(0))^2 \Bigg)\,ds\Bigg\}\Bigg].
\end{split}
\end{equation*}

Now we want to compare this expression to its  analogue at the microscopic level. To that end, fix $n\in \mathbb{N}$, $x\in \Sigma_{n}$, and $\varepsilon>0$. Let $$\Sigma_{n,m}^{\varepsilon} = \{1+\tfrac{m}{2}\varepsilon n, \ldots, n-1-\tfrac{m}{2}\varepsilon n\}.$$ For $m$ even we consider $\Sigma_{n,m}^{\varepsilon}$,  while for $m $ odd, we  would  consider  $\Sigma_{n,m}^{\varepsilon}$  given by $$\Sigma_{n,m}^{\varepsilon} = \{1+\tfrac{m+1}{2}\varepsilon n, \ldots, n-1-\tfrac{m-1}{2}\varepsilon n\},$$
see Figure \ref{figure-set}. Above $\varepsilon n$ denotes $\lfloor \varepsilon n \rfloor$.
\begin{figure*}[h!]
	\begin{center}
		\begin{tikzpicture}[thick, scale=0.75][h!]
		\draw [line width=0.6] (-9,10.5) -- (9,10.5) ;
		\draw [line width=1.4] (-6,10.5) -- (6,10.5) ;
		\foreach \x in  {-9,-6,-5,-4,-3,-2,-1,0,1,2,3,4,5,6,9} 
		\draw[shift={(\x,10.5)},color=black, opacity=1.3] (0pt,0pt) -- (0pt,-4pt) node[below] {};
		\draw[] (-2.8,10.5) node[] {};
		
		\draw[] (-9,10.34) node[below] {\tiny{$1$}};
		\draw[] (-7.5,10.34) node[below] {\tiny{$\cdots$}};
		\draw[] (-6,10.34) node[below] {\tiny{$1+\tfrac{m}{2}\varepsilon n$}};
		\draw[] (5.8,10.34) node[below] {\tiny{$n-1-\tfrac{m}{2}\varepsilon n$}};
		\draw[] (7.5,10.34) node[below] {\tiny{$\cdots$}};
		\draw[] (9,10.34) node[below] {\tiny{$n-1$}};

		
		\shade[shading=ball, ball color=black!50!] (-2,10.76) circle (.245);
		\shade[shading=ball, ball color=black!50!] (-1,10.76) circle (.245);
		\shade[shading=ball, ball color=black!50!] (1,10.76) circle (.245);
		\shade[shading=ball, ball color=black!50!] (3,10.76) circle (.245);
		\end{tikzpicture} 
	\end{center}
\end{figure*}
\vspace*{-1cm}
\begin{figure}[h!]
	\begin{center}
		\begin{tikzpicture}[thick, scale=0.75][h!]
		\draw [line width=0.6] (-9,10.5) -- (9,10.5) ;
		\draw [line width=1.4] (-5,10.5) -- (7,10.5) ;
		\foreach \x in  {-9,-5,-4,-3,-2,-1,0,1,2,3,4,5,6,7,9} 
		\draw[shift={(\x,10.5)},color=black, opacity=1.3] (0pt,0pt) -- (0pt,-4pt) node[below] {};
		\draw[] (-2.8,10.5) node[] {};
		
		\draw[] (-9,10.34) node[below] {\tiny{$1$}};
		\draw[] (-7,10.34) node[below] {\tiny{$\cdots$}};
		\draw[] (-5,10.34) node[below] {\tiny{$1+\tfrac{m+1}{2}\varepsilon n$}};
		\draw[] (6.8,10.34) node[below] {\tiny{$n-1-\tfrac{m-1}{2}\varepsilon n$}};
		\draw[] (8.2,10.34) node[below] {\tiny{$\cdots$}};
		\draw[] (9,10.34) node[below] {\tiny{$n-1$}};

		
		\shade[shading=ball, ball color=black!50!] (-2,10.76) circle (.245);
		\shade[shading=ball, ball color=black!50!] (-1,10.76) circle (.245);
		\shade[shading=ball, ball color=black!50!] (1,10.76) circle (.245);
		\shade[shading=ball, ball color=black!50!] (3,10.76) circle (.245);
		\end{tikzpicture} 
	\end{center}
	\caption{The set $\Sigma^{\varepsilon}_{n,m}$ for $m$ even and for $m$ odd, respectively.}
	\label{figure-set}
\end{figure}
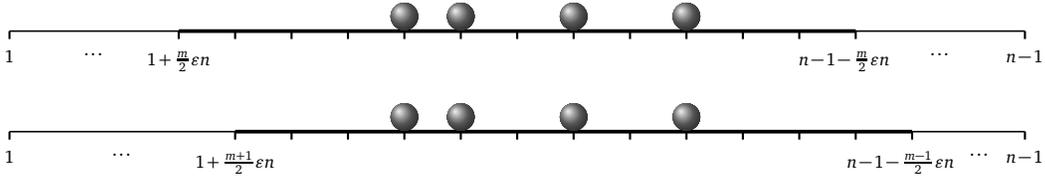

By changing the space integral by its Riemann sum, we can rewrite last display as 
\begin{equation}\label{a_provar_new}
\begin{split}
\varliminf_{n\rightarrow +\infty}\mathbb{E}_{n}\Bigg[ \max _{l \leq q} &\Bigg\{ \int_0^T \Bigg(\frac 1n \sum_{x\in\Sigma_{n,m}^\varepsilon}\prod_{i=0}^{\tfrac{m}{2}-1}\< \pi^n_s,\overleftarrow{\iota_{\varepsilon}}^{x/n-i\varepsilon} \>\prod_{i=0}^{\tfrac{m}{2}-1}\< \pi^n_s,\overrightarrow{\iota}^{x/n+i\varepsilon}_{\varepsilon} \>\;\p_u H^l_s(\tfrac xn)  +\alpha^mH^{l}_s(0) \\
-c&\int_{0}^{1}(H^l_s(u))^2du - \dfrac{c}{\kappa}\sum_{i=0}^{m-1}\alpha^{m-1-i}\prod_{j=0}^{i-1}\langle \pi^n _{s},  \overrightarrow{\iota}^{j\varepsilon}_{\varepsilon}\rangle(H_s^{l}(0))^2 \Bigg)\,ds\Bigg\}\Bigg].
\end{split}
\end{equation}

Now, let 
$
\overleftarrow{\Lambda}^{\ell}_{x} := \{x-\ell+1, \ldots,x\}$ and $ \overrightarrow{\Lambda}^{\ell}_{x} := \{x,\ldots,x+\ell-1\},
$
be the boxes of size $\ell$ to the left and to the right of site $x$, respectively.  We denote by
\begin{equation}\label{empirical densities}
\overleftarrow{\eta}^{\ell}(x) = \frac{1}{\ell} \sum_{y \in \overleftarrow{\Lambda}^{\ell}_{x}} \eta(y) \;\;\; \text{and} \;\;\; \overrightarrow{\eta}^{\ell}(x) = \frac{1}{\ell} \sum_{y \in \overrightarrow{\Lambda}^{\ell}_{x}} \eta(y)
\end{equation}
the empirical densities in the boxes $\overleftarrow{\Lambda}^{\ell}_{x}$ and $\overrightarrow{\Lambda}^{\ell}_{x}$. 

Now we explain why we will need to  introduce the subset $\Sigma_{n,m}^{\varepsilon}$ of the bulk $\Sigma_{n}$. We use this set since, for each $x\in \Sigma_{n,m}^{\varepsilon}$ we  will need to replace the occupation at site $x$ by its average to the left or right of $x$ on a box of size $\varepsilon n$, and we are allowed to do so  for $x\in\Sigma_{n,m}^{\varepsilon}$ but not for $x$ on the whole bulk. 

From \eqref{approx identity} and \eqref{empirical densities}, we have that
\begin{equation*}\label{disc_1}\< \pi^n_s, \overleftarrow\iota^{x/n-i\varepsilon}_\ve\>=\overleftarrow{\eta}^{\ve n}_{sn^2}(x-i\varepsilon n), \quad \< \pi^n_s, \overrightarrow\iota^{x/n+i\varepsilon }_\ve\>=\overrightarrow{\eta}^{\ve n}_{sn^2}(x+1+i\varepsilon n )+O(\tfrac {1}{\varepsilon n}),	
\end{equation*}
and 
\begin{equation*}\label{disc_new} \< \pi^n_s, \overrightarrow\iota^{j\varepsilon }_\ve\>=\overrightarrow{\eta}^{\ve n}_{sn^2}(2+j\varepsilon n )+O(\tfrac{1}{\ve n}),	
\end{equation*} 
for $i=0,\ldots, \frac m2-1$ and $j=0,\ldots, i-1$. Then		
we can rewrite the expectation, now with respect to $\mathbb P_{\mu_n}$,  in \eqref{a_provar_new} as

\begin{equation*}
\begin{split}
\mathbb{E}_{\mu_n}  \Bigg[ \max _{l \leq q} \Bigg\{ \int_0^T \Bigg(\dfrac{1}{n}&\sum_{x\in\Sigma^{\varepsilon}_{n,m}}\prod_{i=0}^{\tfrac{m}{2}-1}\overleftarrow{\eta}^{\varepsilon n}_{sn^2}(x-i\varepsilon n )\prod_{i=0}^{\tfrac{m}{2}-1}\overrightarrow{\eta}_{sn^2}^{\ve n}(x+1+i\varepsilon n  )\;\p_u H^l_s(\pfrac{x}{n}) +\alpha^mH^{l}_s(0) \\
-c&\int_{0}^{1}(H^l_s(u))^2du - \dfrac{c}{\kappa} \sum_{i=0}^{m-1}\alpha^{m-1-i}\prod_{j=0}^{i-1}\overrightarrow{\eta}_{sn^2}^{\ve n}(2+j\varepsilon n )(H_s^{l}(0))^2\Bigg)\,ds\Bigg\}\Bigg]
\end{split}
\end{equation*}
plus terms that vanish as $n\to \infty$. Now we state two useful results whose proofs can be found in Section 5 of \cite{BPGN} for the case $m=2$. Before recall the definition of $\tau_{x}h^{m}$ given in \eqref{shift}.
\begin{theorem}\label{replacement bulk}
	Let $H:[0,T]\times[0,1]\to \mathbb{R}$ be such that $\| H\|_{\infty}<\infty$. For any $t\in [0,T]$, we have that
	\begin{equation*}
	\varlimsup_{\varepsilon\to 0} \varliminf_{n\to \infty} \mathbb{E}_{\mu_n}\left[\left|\int_{0}^{t} \dfrac{1}{n}\sum_{x\in\Sigma_{n,m}^{\varepsilon}}H_s\left(\tfrac{x}{n}\right)\Bigg\{\tau_x h^{m}(\eta_{sn^2})- \prod_{i=0}^{\tfrac{m}{2}-1}\overleftarrow{\eta}^{\varepsilon n}_{sn^2}(x-i\varepsilon n )\overrightarrow{\eta}_{sn^2}^{\ve n}(x+1+i\varepsilon n )\Bigg\}\,ds \right| \right] = 0.
	\end{equation*}
\end{theorem}
The previous result can be proved by following the arguments as described in Theorem 5.10 and Lemmas 5.11,5.13, 5.14 and 5.15 of \cite{BPGN}. We observe that the proof in \cite {BPGN} is given for $m=2$, but for  $m>2$ it is completely analogous.  The only difference is the fact that, for example, for $m=3$, the function $\tau_x h^m(\eta)$ contains terms of the form $\eta(x)\eta(x+1)\eta(x+2)$. In order to replace them by products of averages, one first has to replace  in the product above $\eta(x)$ by $\eta(x-\varepsilon n)$, which can be done by adapting the arguments of  Lemma 5.11 of \cite{BPGN},   and then follow the proof for the case $m=2$, see Figure \ref{figure-box}. 
We do not present the proof of these arguments since they are very similar to those of \cite{BPGN}.

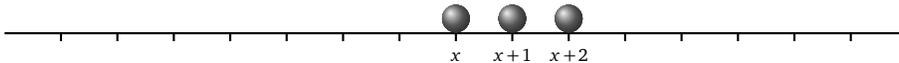
\begin{figure*}[h!]
	\begin{center}
		\begin{tikzpicture}[thick, scale=0.75]
		\draw [line width=0.8] (-8,10.5) -- (8,10.5) ; 
		\foreach \x in  {-7,-6,-5,-4,-3,-2,-1,0,1,2,3,4,5,6,7} 
		\draw[shift={(\x,10.5)},color=black, opacity=1] (0pt,0pt) -- (0pt,-4pt) node[below] {};
		\draw[] (-2.8,10.5) node[] {};
		
		\draw[] (0,10.34) node[below] {\tiny{$x$}};
		\draw[] (1,10.4) node[below] {\tiny{$x+1$}};
		\draw[] (2,10.4) node[below] {\tiny{$x+2$}};
		
		
		\shade[shading=ball, ball color=black!50!] (0,10.76) circle (.245);
		\shade[shading=ball, ball color=black!50!] (1,10.76) circle (.245);
		\shade[shading=ball, ball color=black!50!] (2,10.76) circle (.245);
		
		\end{tikzpicture} 
	\end{center}
\end{figure*}
\vspace{-0.8cm}
\begin{figure*}[h!]
	\begin{center}
		\begin{tikzpicture}[thick, scale=0.75]
		\draw [line width=0.8] (-8,10.5) -- (8,10.5) ; 
		\foreach \x in  {-7,-6,-5,-4,-3,-2,-1,0,1,2,3,4,5,6,7}
		\draw[shift={(\x,10.5)},color=black, opacity=1] (0pt,0pt) -- (0pt,-4pt) node[below] {};
		\draw[] (-2.8,10.5) node[] {};
		
		\draw[] (-3,10.4) node[below] {\tiny{$x-\varepsilon n$}};
		\draw[] (0,10.34) node[below] {\tiny{$x$}};
		\draw[] (1,10.4) node[below] {\tiny{$x+1$}};
		\draw[] (2,10.4) node[below] {\tiny{$x+2$}};
		
		
		\shade[shading=ball, ball color=black!50!] (-3,10.76) circle (.245);
		\shade[shading=ball, ball color=black!50!] (1,10.76) circle (.245);
		\shade[shading=ball, ball color=black!50!] (2,10.76) circle (.245);
		\end{tikzpicture} 
	\end{center}
\end{figure*}
\vspace{-1cm}
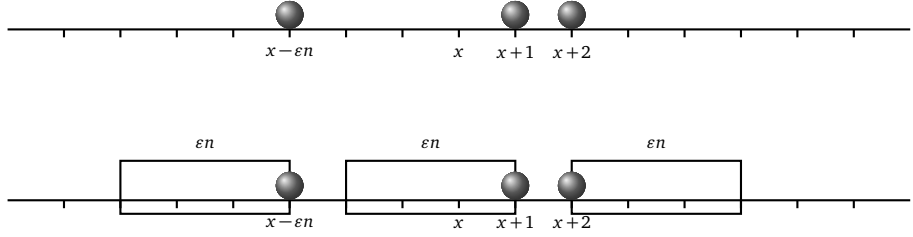
\begin{figure}[h!]
	\begin{center}
		\begin{tikzpicture}[thick, scale=0.75]
		\draw [line width=0.8] (-8,10.5) -- (8,10.5) ; 
		\foreach \x in  {-7,-6,-5,-4,-3,-2,-1,0,1,2,3,4,5,6,7} 
		\draw[shift={(\x,10.5)},color=black, opacity=1] (0pt,0pt) -- (0pt,-4pt) node[below] {};
		\draw[] (-2.8,10.5) node[] {};
		
		\draw[] (-3,10.4) node[below] {\tiny{$x-\varepsilon n$}};
		\draw[] (0,10.34) node[below] {\tiny{$x$}};
		\draw[] (1,10.4) node[below] {\tiny{$x+1$}};
		\draw[] (2,10.4) node[below] {\tiny{$x+2$}};
		\draw[] (-4.5,11.25) node[above] {\tiny{$\varepsilon n$}};
		\draw[] (-0.5,11.25) node[above] {\tiny{$\varepsilon n$}};
		\draw[] (3.5,11.25) node[above] {\tiny{$\varepsilon n$}};
		
		\draw[thick] (-2, 10.26) rectangle (1, 11.2);
		\draw[thick] (-6, 10.26) rectangle (-3, 11.2);
		\draw[thick] (2, 10.26) rectangle (5, 11.2);
		
		\shade[shading=ball, ball color=black!50!] (-3,10.76) circle (.245);
		\shade[shading=ball, ball color=black!50!] (2,10.76) circle (.245);
		\shade[shading=ball, ball color=black!50!] (1,10.76) circle (.245);
		\end{tikzpicture} 
	\end{center}
	\caption{Replacing the occupation sites $x$, $x+1$, and $x+2$ by occupation averages on boxes of size $\varepsilon n$.}
	\label{figure-box}
\end{figure}

\begin{theorem}\label{replacement boundary4}
	For any $t\in[0,T]$ and any $i\in\{1,\ldots, m-1\}$, we have 
	\begin{equation*}
	\varlimsup_{\varepsilon\to 0} \varliminf_{n\to \infty} \mathbb{E}_{\mu_n}\left[\left| \int_{0}^{t}\Bigg\{\prod_{j=0}^{i-1}\overrightarrow{\eta}_{sn^2}^{\ve n}(2+j\varepsilon n )-\prod_{j=0}^{i-1}{\eta}_{sn^2}(2+j )\Bigg\}\, ds \right| \right] = 0.
	\end{equation*}
\end{theorem} 
The last result can be derived by applying  the same arguments as in  Theorem 5.10  of \cite{BPGN}, by taking into consideration that the product has $i-1$ factors.  We leave the details to the reader. 
By putting together the two previous results, we are left to show that 
\begin{equation}\label{a_provar_5}
\begin{split}
\varlimsup_{\ve\to 0}\varliminf_{n\rightarrow +\infty}\mathbb{E}_{\mu_n} \Bigg[ \max _{l \leq q} \Bigg\{ \int_0^T &\Bigg(\dfrac{1}{n}\sum_{x=1}^{n-2}\p_u H^l_s(\pfrac{x}{n})\tau_xh^m(\eta_{sn^2}) +\alpha^mH^{l}_s(0) \\
-c&\int_{0}^{1}(H^l_s(u))^2du - \dfrac{c}{\kappa}\mathcal{R}^{\alpha}_{m}(\eta_{sn^2})(H_s^{l}(0))^2 \Bigg)\,ds\Bigg\}\Bigg] \leq M_0,
\end{split}
\end{equation}
where $\tau_xh^{m}(\eta)$ is defined in \eqref{shift} and \begin{equation}\label{eq:R_new}
\mathcal{R}^{\alpha}_{m}(\eta)=\sum_{i=0}^{m-1}\alpha^{m-1-i}\prod_{j=0}^{i-1}{\eta}(2+j ).\end{equation}

Observe that above we are back to the whole bulk since the replacement from $\Sigma^\varepsilon_{n,m} $ to $\Sigma_n$ vanishes  as $\varepsilon\to 0.$ Now we want to change the initial measure $\mu_{n}$ to a suitable measure,  here being  the Bernoulli product measure $\nu^{n}_{\rho(\cdot)}$,  with marginals  given by
$\nu^{n}_{\rho(\cdot)}\{\eta: \eta(x)=1\} = \rho\left(\tfrac{x}{n}\right),$
where $\rho(\cdot)$  is  a Lipschitz profile such that for all $u\in(0,1)$,
\begin{equation*}\label{eq:choice_profile}
\alpha = \rho(0) \leq \rho(u) \leq \rho(1)=\beta
\end{equation*}
and locally constant at the boundary.
We observe that if $H\left(\mu_n | \nu^{n}_{\rho(\cdot)}\right)$ denotes the relative entropy of $\mu_n$ with respect to $\nu^{n}_{\rho(\cdot)}$, then, a simple computation shows that, there exists a constant $C(\alpha,\beta)$, such that
\begin{equation}\label{ent}
H\left(\mu_n | \nu^{n}_{\rho(\cdot)}\right) \,\leq\, n\,C(\alpha,\beta).	
\end{equation}
We observe that to derive the inequality above we need to restrict $\alpha,\beta\in(0,1)$.
Therefore, by entropy's and Jensen's inequality and the fact that $\exp\left\{\max_{l\leq q} a_{l}\right\}\leq \sum_{l=1}^{q}\exp\{a_{l}\}$, the expectation in \eqref{a_provar_5} is bounded from above by 
\begin{equation}\label{eqqq}
\begin{split}
C(\alpha,\beta)+\dfrac{1}{n}\log \mathbb{E}_{\nu_{\rho(\cdot)}^n}\Bigg[ \sum_{l=1}^{q} \exp\Bigg\{ \int_{0}^{T} \Bigg( \sum_{x=1}^{n-2}\p_u &H^l_s(\pfrac{x}{n})\tau_xh^m(\eta_{sn^2}) +n\alpha^mH^{l}_s(0) \\
-n\,c&\int_{0}^{1}(H^l_s(u))^2du - n\dfrac{c}{\kappa}\mathcal{R}^{\alpha}_{m}(\eta_{sn^2})(H_s^{l}(0))^2  \Bigg) ds\Bigg\} \Bigg].
\end{split}
\end{equation}
From the identity $$\varlimsup_{n \to \infty} n^{-1}\log(a_n + b_n) = \max\left\{ \varlimsup_{n \to \infty} n^{-1}\log(a_n), \, \varlimsup_{n \to \infty} n^{-1}\log(b_n) \right\},$$ in order to estimate \eqref{eqqq}, it is enough to bound 
\begin{equation*}\label{bound1}
\begin{split}
\dfrac{1}{n}\log \mathbb{E}_{\nu_{\rho(\cdot)}^n}\Bigg[ \exp\Bigg\{ \int_{0}^{T} \Bigg(\sum_{x=1}^{n-2}\p_u H_s(\pfrac{x}{n})\tau_xh^m(\eta_{sn^2}) &- n\dfrac{c}{\kappa}\mathcal{R}^{\alpha}_{m}(\eta_{sn^2})\left(H_s(0)\right)^{2} \\
&-n\,c\int_{0}^{1}\left(H_s(u)\right)^{2}\,du+n\alpha^mH_s(0) \Bigg) ds\Bigg\} \Bigg],
\end{split}
\end{equation*}
for a fixed function $H\in C^{0,1}([0,T]\times[0,1])$.
Now, by  Feynman-Kac's formula (see, for example, Lemma A.1 of  \cite{bmns}), we can bound the previous display  from above by 
\begin{equation}\label{bound2} 
\begin{split}
\int _{0}^{T} & \sup _{f}\;\,\Bigg\{\dfrac{1}{n}\int_{\Omega_n} \sum_{x=1}^{n-2}\partial_uH_s(\pfrac{x}{n})\tau_xh^m(\eta)f(\eta)\,d\nu^{n}_{\rho(\cdot)} - \dfrac{c}{\kappa}\left(H_s(0)\right)^2\int_{\Omega_n}\mathcal{R}^{\alpha}_{m}(\eta)f(\eta)\,d\nu^{n}_{\rho(\cdot)}  \\
&-c\int_{0}^{1}\left(H_s(u)\right)^2\,du +\alpha^mH_s(0)+ n \,\langle L^m_{n}\sqrt{f},\sqrt{f}\rangle_{{\nu_{\rho(\cdot)}^n}} \Bigg\} \,\; ds,
\end{split}
\end{equation}
where the supremum is carried over all the densities $f$ with respect to $\nu_{\rho(\cdot)}^n$. Below, we introduce the Dirichlet form, which is given on $\sqrt f$ by
\begin{equation*}
\< \sqrt{f},-L^{m}_n \sqrt{f} \>_{\nu^{n}_{\rho(\cdot)}} = \< \sqrt{f},-L^{m}_P \sqrt{f} \>_{\nu^{n}_{\rho(\cdot)}} + n^{a-2}\< \sqrt{f},-L_S \sqrt{f} \>_{\nu^{n}_{\rho(\cdot)}} + \< \sqrt{f},-L_B \sqrt{f} \>_{\nu^{n}_{\rho(\cdot)}},
\end{equation*}
where the inner product $\< \cdot,\cdot \>_{\nu^{n}_{\rho(\cdot)}}$ is the one of  $L^{2}\left(\Omega_{n}, \nu^{n}_{\rho(\cdot)}\right)$. By the choice of  $\rho(\cdot)$ and from Lemma 5.1 of \cite{dgn}, we have that
\begin{equation}\label{eq:dir_bound}
\< L^{m}_{n}\sqrt{f}, \sqrt{f} \>_{\nu^{n}_{\rho(\cdot)}} \leq -\frac{1}{4} D^m_{n}(\sqrt{f},\nu^{n}_{\rho(\cdot)}) + O(\pfrac{1}{n}),
\end{equation}
where 
$D^m_{n}(\sqrt{f}, \nu^{n}_{\rho(\cdot)}) := (D^m_{P} + n^{a-2}D_{S} + D_{B})(\sqrt{f},\nu^{n}_{\rho(\cdot)}),$
with 
\begin{equation*}\label{D_P}
D^m_{P}(\sqrt{f},\nu^{n}_{\rho(\cdot)})\;:=\;\sum_{x=1}^{n-2}\int_{\Omega_{n}} p^{m}_{x,x+1}(\eta)\big(\sqrt{f(\eta^{x,x+1})}-\sqrt{f(\eta)}\big)^{2} \, d\nu^{n}_{\rho(\cdot)},
\end{equation*}
and 
\begin{equation*}\label{D_S}
\begin{split}
D_{S}(\sqrt{f},\nu^{n}_{\rho(\cdot)})\, &:=\,\sum_{x=1}^{n-2} 
\int_{\Omega_{n}} \big\{a_{x,x+1}(\eta)+a_{x+1,x}(\eta)\big\}\big(\sqrt{f(\eta^{x,x+1})}-\sqrt{f(\eta)}\big)^{2} \, d\nu^{n}_{\rho(\cdot)} \\
&= \sum_{x=1}^{n-2} 
\int_{\Omega_{n}} \big(\sqrt{f(\eta^{x,x+1})}-\sqrt{f(\eta)}\big)^{2} \, d\nu^{n}_{\rho(\cdot)}.
\end{split}
\end{equation*}
Above $p^{m}_{x,x+1}(\eta):=c^{m}_{x,x+1}(\eta)\big\{a_{x,x+1}(\eta)+a_{x+1,x}(\eta)\big\}$, where the rates $c^{m}_{x,x+1}(\eta)$ and $a_{x,x+1}(\eta)$ are given in \eqref{PMM rate} and \eqref{SSEP rate} respectively, and
\begin{equation*}
D_{B}(\sqrt{f},\nu^{n}_{\rho(\cdot)})\, := \, \tfrac{\kappa}{n^\theta}\Big(F_{1}^{\alpha}(\sqrt{f},\nu^{n}_{\rho(\cdot)})+F_{n-1}^{\beta}(\sqrt{f},\nu^{n}_{\rho(\cdot)})\Big),
\end{equation*}
where for $x\in\{1,n-1\}$ and $\gamma\in\{\alpha,\beta\}$, $F_{x}^{\gamma}$ is given by
\begin{equation}\label{F_terms}
F_{x}^{\gamma}(\sqrt{f},\nu^{n}_{\rho(\cdot)})= \int_{\Omega_{n}} I_{z}^{\gamma}(\eta)\big( \sqrt{f(\eta^{x})} - \sqrt{f(\eta)}\big)^{2}\,d\nu^{n}_{\rho(\cdot)},
\end{equation}
with $I_{x}^{\gamma}$ given in \eqref{Boundary rates}.

Now we are back to  \eqref{bound2}. 	From a Taylor expansion on $H$, we can replace its space derivative by the discrete gradient $\nabla^{-}_{n} H_{s}\left(\tfrac{x}{n}\right) = n\left(H_s\left(\tfrac{x}{n}\right)-H_s\left(\tfrac{x-1}{n}\right)\right)$, by paying a price of order $O\left(n^{-1}\right)$. Then, from a summation by parts, the first integral inside the supremum in \eqref{bound2} is equal to
\begin{equation}\label{bound3}
\begin{split}
& \int_{\Omega_{n}} \sum_{x=1}^{n-2} H_{s}\left(\tfrac{x}{n}\right)\big\{\tau_xh^{m}(\eta)-\tau_{x+1}h^{m}(\eta)\big\}f(\eta) \, d{\nu_{\rho(\cdot)}^n} \\
-& \int_{\Omega_{n}} \big\{H_{s}(0)\tau_{1}h^{m}(\eta)- H_{s}\left(\tfrac{n-2}{n}\right)\tau_{n-1}h^{m}(\eta)\big\}f(\eta) \, d{\nu_{\rho(\cdot)}^n}.
\end{split}
\end{equation}
Now, we need to bound both terms in expression \eqref{bound3} separately. We will call the first one, the bulk term, and the second one, the boundary term. We note that the main difference and difficulty when comparing this proof to the one presented in Sect. $6$ of \cite{BPGN} comes from the examination of the boundary term. For that reason we refer the reader to that article for  details on how to bound the bulk term from above by	
\begin{equation}\label{bound9}
\dfrac{n}{4}D_{P}^m(\sqrt{f},\nu_{\rho(\cdot)}^n)
+\dfrac{1}{n}\sum_{x=1}^{n-2} \left(H_{s}\left(\tfrac{x}{n}\right)\right)^2\Big(m+\tilde{C}(m,\alpha,\beta)\Big) + C(\rho),
\end{equation}
with $\tilde C(m,\alpha,\beta)=m^2+m\hat C(\alpha,\beta)$, where $\hat C(\alpha,\beta)$ is a positive constant. We observe that in \cite{BPGN}, since the function $H$ was taken with compact support, the boundary term of \eqref{bound3} is of order $O\left(n^{-1}\right)$ and vanishes, when $n\to \infty$. Since we consider a more general set of functions inside the supremum, our boundary term does not vanish and we need to use a more precise argument to treat it. \\ 
Now we need to examine the second line of  \eqref{bound3}.  Let us begin by examining the leftmost term given by
\begin{equation}\label{tau1-term}
-\int_{\Omega_{n}}H_s\left(0\right)\tau_1h^{m}(\eta)f(\eta)\, d\nu^{n}_{\rho(\cdot)}.
\end{equation}
Recall that we neglected above all the terms from the right boundary so that we will not treat the contribution from the rightmost term on the second line of \eqref{bound3}, but it completely analogous to what we do for the left boundary. 
Recall from \eqref{shift-bd} that $$\tau_{1}h^{m}(\eta) = \sum_{k=0}^{m-1}\alpha^{k}\; \prod_{j=1}^{m-k}\eta(j) - \sum_{k=1}^{m-1}\alpha^{k}\; \prod_{j=2}^{m+1-k} \eta(j),$$ and $\tau_{1}h^{m}(\eta)- \alpha^{m} =(\eta(1)-\alpha)\mathcal{R}^{\alpha}_{m}(\eta).$ Summing and subtracting $\alpha^m$ in \eqref{tau1-term}, and since $f$ is a density with respect to $\nu^{n}_{\rho(\cdot)}$,  we can rewrite \eqref{tau1-term} as
\begin{equation}\label{bound10}
H_s\left(0\right)\Bigg(\int_{\Omega_{n}}(\alpha-\eta(1))\mathcal{R}^{\alpha}_{m}(\eta)f(\eta)\, d\nu^{n}_{\rho(\cdot)} - \alpha^m\Bigg).
\end{equation}
The argument to estimate the leftmost term in \eqref{bound10} is similar, in essence, to the one used to treat the first term of \eqref{bound3}. We write the leftmost term in \eqref{bound10}  as one half of it plus one half of it, and by summing and subtracting $\tfrac{1}{2}f(\eta^{1})$, we obtain
\begin{equation}\label{bound11}
\begin{split}
\frac{1}{2}&H_s\left(0\right)\int_{\Omega_{n}}(\alpha-\eta(1))\mathcal{R}^{\alpha}_{m}(\eta)(f(\eta)+f(\eta^1))\, d\nu^{n}_{\rho(\cdot)} \\
+\frac{1}{2}&H_s\left(0\right)\int_{\Omega_{n}}(\alpha-\eta(1))\mathcal{R}^{\alpha}_{m}(\eta)(f(\eta)-f(\eta^1))\, d\nu^{n}_{\rho(\cdot)}.
\end{split}
\end{equation}
Denoting by $\tilde{\eta}$ the configuration $\eta$ removing its value at site $1$, and noticing that $\mathcal{R}^{\alpha}_{m}(\eta)$ does not depend on $\eta(1)$, we can write the first term in \eqref{bound11} as
\begin{equation*}
\begin{split}
\frac{1}{2}H_s\left(0\right)\sum_{\tilde{\eta}\in \Omega_{n-1}}&\Bigg( \alpha\mathcal{R}^{\alpha}_{m}(\eta)(f(0,\tilde{\eta})+f(1,\tilde{\eta}))\left(1-\rho\left(\tfrac{1}{n}\right) \right)\\
&+(\alpha-1)\mathcal{R}^{\alpha}_{m}(\eta)(f(1,\tilde{\eta})+f(0,\tilde{\eta}))\rho\left(\tfrac{1}{n}\right)\Bigg)\nu^{n-1}_{\rho(\cdot)}(\tilde{\eta}),
\end{split}
\end{equation*}
where the notation $f(1,\tilde{\eta})$ (resp. $f(0,\tilde{\eta})$) means that we are computing $f(\eta)$ with $\eta(1)=1$ (resp. $\eta(1)=0$). Hence, the previous expression is equal to 
\begin{equation*}
\begin{split}
\frac{1}{2}H_s\left(0\right)\left(\alpha-\rho\left(\tfrac{1}{n}\right) \right)\sum_{\tilde{\eta}\in \Omega_{n-1}}\mathcal{R}^{\alpha}_{m}(\eta){(f(0,\tilde{\eta})+f(1,\tilde{\eta}))}\nu^{n-1}_{\rho(\cdot)}(\tilde{\eta}).
\end{split}
\end{equation*}
Since $\mathcal{R}^{\alpha}_{m}(\eta)\leq m$, $\rho$ satisfies the conditions we imposed, and  since $f$ is a density with respect to $\nu^{n}_{\rho(\cdot)}$, the previous expression vanishes when $n\to \infty$. Let us now estimate the second term of \eqref{bound11}. Combining the identity $a-b = (\sqrt{a}-\sqrt{b})(\sqrt{a}+\sqrt{b})$ and Young's inequality, we bound from above the second term of \eqref{bound11} by
\begin{equation*}\label{bound12}
\begin{split}
\frac{A}{2}&\int_{\Omega_{n}}I_{1}^{\alpha}(\eta)\left(\sqrt{f(\eta)}-\sqrt{f(\eta^1)}\right)^2\, d\nu^{n}_{\rho(\cdot)}\\ 
+\frac{\left(H_s\left(0\right)\right)^2}{8A}&\int_{\Omega_{n}} \frac{(\alpha - \eta(1))^2\left(\mathcal{R}^{\alpha}_{m}(\eta)\right)^2}{I_{1}^{\alpha}(\eta)}\left(\sqrt{f(\eta)}+\sqrt{f(\eta^1)}\right)^2\, d\nu^{n}_{\rho(\cdot)},
\end{split}
\end{equation*}
where $A>0$ and $I_{1}^{\alpha}(\eta)$ is defined in \eqref{Boundary rates}. Recall the definition of $F^{\alpha}_{1}(\sqrt{f},\nu^{n}_{\rho(\cdot)})$ in \eqref{F_terms}. Using the inequality $(a+b)^{2} \leq 2a^2 + 2b^2$ and the identity $I^{\alpha}_{1}(\eta)=\tfrac{(\alpha- \eta(1))^2}{I^{\alpha}_{1}(\eta)} $, last expression can be bounded from above by $\frac{A}{2}F^{\alpha}_{1}(\sqrt{f},\nu^{n}_{\rho(\cdot)}) $ plus
\begin{equation*}
\frac{\left(H_s\left(0\right)\right)^2}{4A}\left(\int_{\Omega_{n}} \left(\mathcal{R}^{\alpha}_{m}(\eta)\right)^2 I^{\alpha}_{1}(\eta)f(\eta)\, d\nu^{n}_{\rho(\cdot)}+ \int_{\Omega_{n}} \left(\mathcal{R}^{\alpha}_{m}(\eta)\right)^2 I^{\alpha}_{1}(\eta)f(\eta^1)\, d\nu^{n}_{\rho(\cdot)} \right).
\end{equation*} 
After a change of variables in the second integral of the previous expression, we get
\begin{equation*}
\begin{split}
\frac{A}{2}F^{\alpha}_{1}(\sqrt{f},\nu^{n}_{\rho(\cdot)}) + \frac{\left(H_s\left(0\right)\right)^2}{4A}&\Bigg(\int_{\Omega_{n}} \left(\mathcal{R}^{\alpha}_{m}(\eta)\right)^2 I^{\alpha}_{1}(\eta)f(\eta)\, d\nu^{n}_{\rho(\cdot)}\\
&+ \int_{\Omega_{n}} \left(\mathcal{R}^{\alpha}_{m}(\eta)\right)^2 I^{\alpha}_{1}(\eta^1)\tfrac{d\nu^{n}_{\rho(\cdot)}(\eta^1)}{d\nu^{n}_{\rho(\cdot)}(\eta)}f(\eta)\, d\nu^{n}_{\rho(\cdot)} \Bigg).
\end{split}
\end{equation*}
Since $ I^{\alpha}_{1}(\eta) = I^{\alpha}_{1}(\eta^1)\tfrac{d\nu^{n}_{\rho(\cdot)}(\eta^1)}{d\nu^{n}_{\rho(\cdot)}(\eta)} $, last expression is equal to	
\begin{equation*}
\frac{A}{2}F^{\alpha}_{1}(\sqrt{f},\nu^{n}_{\rho(\cdot)})+\frac{\left(H_s\left(0\right)\right)^2}{2A}\int_{\Omega_{n}} \left(\mathcal{R}^{\alpha}_{m}(\eta)\right)^2 I^{\alpha}_{1}(\eta)f(\eta)\, d\nu^{n}_{\rho(\cdot)}.
\end{equation*}
Thus, taking $A=\tfrac{\kappa n}{2n^\theta}$ and since $\left(\mathcal{R}^{\alpha}_{m}(\eta)\right)^2 I^{\alpha}_{1}(\eta) \leq \left(\mathcal{R}^{\alpha}_{m}(\eta)\right)^2 \leq m\mathcal{R}^{\alpha}_{m}(\eta)$, the first boundary term of \eqref{bound3} is bounded from above by
\begin{equation}\label{bound14}
\frac{n}{4}	\frac{\kappa}{n^\theta}F^{\alpha}_{1}(\sqrt{f},\nu^{n}_{\rho(\cdot)})+\frac{m}{\kappa}	\frac{n^\theta}{n}\left(H_s\left(0\right)\right)^2\int_{\Omega_{n}} \mathcal{R}^{\alpha}_{m}(\eta) f(\eta)\, d\nu^{n}_{\rho(\cdot)} - \alpha^m H_s\left(0\right).
\end{equation}
Thus, combining \eqref{bound9} and  \eqref{bound14}, we bound \eqref{bound3}  from above by
\begin{equation*}
\begin{split}
\dfrac{n}{4}D^m_{P}(\sqrt{f},\nu_{\rho(\cdot)}^n)
&+\Big(m+\tilde{C}(m,\alpha,\beta)\Big)\dfrac{1}{n}\sum_{x=1}^{n-2} \left(H_{s}\left(\tfrac{x}{n}\right)\right)^2 + C(\rho) \\&+ 	\frac{n}{4}	\frac{\kappa}{n^\theta}F^{\alpha}_{1}(\sqrt{f},\nu^{n}_{\rho(\cdot)})+\frac{m}{\kappa}	\frac{n^\theta}{n}\left(H_s\left(0\right)\right)^2\int_{\Omega_{n}} \mathcal{R}^{\alpha}_{m}(\eta) f(\eta)\, d\nu^{n}_{\rho(\cdot)} - \alpha^m H_s\left(0\right),
\end{split}
\end{equation*} plus the terms that come from the right boundary and which are very similar to the ones we obtained for the left boundary. 
Therefore, taking $c =m+\tilde{C}(m,\alpha,\beta)$ in \eqref{bound2}, from last expression and \eqref{eq:dir_bound},  we can bound \eqref{eqqq}  from above by 
\begin{equation*}
\begin{split}
C(\alpha,\beta)+	\int_{0}^{T} \Bigg\{\Big(m+&\tilde{C}(m,\alpha,\beta)\Big)\dfrac{1}{n}\sum_{x=1}^{n-2} \left(H_{s}\left(\tfrac{x}{n}\right)\right)^2 -\Big(m+\tilde{C}(m,\alpha,\beta)\Big)\int_{0}^{1}\left(H_{s}(u)\right)^2\,du \\+  C(\rho) 
+&\frac{(H_s(0))^2}{\kappa}\Bigg(m\left(\frac{n^\theta}{n}-1\right)-\tilde C(m,\alpha,\beta)\Bigg)\int_{\Omega_n}\mathcal{R}^{\alpha}_{m}(\eta)f(\eta)\, d\nu^{n}_{\rho(\cdot)}\Bigg\} ds.
\end{split}
\end{equation*}
Above  $C(\alpha,\beta)$ is given in \eqref{ent}.
Noting that $\tilde{C}(m,\alpha,\beta)$ and  $\mathcal R^\alpha_m(\eta)$ as defined in \eqref{eq:R_new} are non-negative we can bound from above the last display by 
\begin{equation} \label{last}
\begin{split}
C(\alpha,\beta)+ \int_{0}^{T} \Bigg\{\Big(m+\tilde{C}(m,\alpha,\beta)\Big)\Bigg(\dfrac{1}{n}\sum_{x=1}^{n-2}& \left(H_{s}\left(\tfrac{x}{n}\right)\right)^2 -\int_{0}^{1}\left(H_{s}(u)\right)^2\,du\Bigg) \\+  C(\rho)
+&\frac{m}{\kappa}(H_s(0))^2\left(\tfrac{n^\theta}{n}-1\right)\int_{\Omega_n}\mathcal{R}^{\alpha}_{m}(\eta)f(\eta)\, d\nu^{n}_{\rho(\cdot)}\Bigg\} ds.
\end{split}
\end{equation}

Now, recall that $\theta=1$. Therefore, the previous expression converges to $TC(\rho)$, as $n \to \infty$.  From all this, we conclude that the  expectation in \eqref{estimate1} is bounded from above by $M_{0}: = C(\alpha,\beta) + TC(\rho)$. This ends the proof. 

\begin{remark}
	We note that due to the rightmost term in \eqref{last},  the energy estimate with test functions without compact support can only be obtained for  $\theta=1$. Even in the case $\theta<1$, where the factor 
	$\frac{m}{\kappa}\frac{n^\theta}{n}\left(H_s\left(0\right)\right)^2\int_{\Omega_n}\mathcal{R}^{\alpha}_{m}(\eta)f(\eta)\, d\nu^{n}_{\rho(\cdot)}$  would simply vanish as  $n\to \infty$, there would remain the term  $\frac{m}{\kappa}\left(H_s(0)\right)^2\int_{\Omega_n}\mathcal{R}^{\alpha}_{m}(\eta)f(\eta)\, d\nu^{n}_{\rho(\cdot)}ds$
	which  would blow up since we are taking the supremum over functions $f$.
\end{remark}

\section{The Slow bond case}\label{slow bond}

In this section we shortly explain how the previous results could be extended to the  porous medium model evolving on the one dimensional torus $\mathbb T_n=\bb  Z/n\bb Z = \{0,1,2,...,n -1\}$ and  we present the set up for the convergence result for the  respective hydrodynamic equations. We start in  Subsection \ref{macro:slow_bond} by presenting  the macroscopic equations and the convergence result that we conjecture to be true and that we leave for a future work. Then, in Subsection \ref{micro:slow_bond}
we introduce the porous medium model with a slow bond and we present an heuristic argument which shows that its hydrodynamic behavior  is given by the equations  presented in Subsection  \ref{macro:slow_bond}.

\subsection{The porous Medium equation on $\mathbb T$}\label{macro:slow_bond}

We denote by $\mathbb T$ the continuous one-dimensional torus $\bb R/\bb Z$ and for a function $f:\bb T\backslash\{0\}\to \bb R$,
we denote 
$$f(0):=f(0^+)=\lim_{\at{u\to 0}{u>0}}f(u)\quad \mbox{ and }\quad f(1):=f(0^-)=\lim_{\at{u\to 0}{u<0}}f(u)=f(1^-)=\lim_{\at{u\to 1}{u<1}}f(u).$$

\begin{definition}
	\label{Def. Robin2}
	Let $T>0$, $\kappa >0$ and $g:\bb T\rightarrow [0,1]$ a measurable function. We say that  $\rho^\kappa:[0,T]\times\bb T\backslash\{0\} \to [0,1]$ is a weak solution of the porous medium equation with Robin boundary conditions, see \eqref{eq:Robin2}, 
	if the following conditions hold: 
	\begin{enumerate}
		\item $(\rho^\kappa)^m \in L^{2}(0,T; \mathcal{H}^{1})$; 
		\item $\rho^\kappa$ satisfies the integral equation:
		\begin{equation}\label{eq:Robin integral3}
		\begin{split}
		\langle \rho_{t}^\kappa,  G_{t}\rangle  -\langle g,  G_{0}\rangle &- \int_0^t\langle \rho_{s}^\kappa,( \partial_s G_{s}+(\rho_s^\kappa)^{m-1} \Delta G_s ) \rangle   \, ds
		\\&+\int^{t}_{0}  \big\{  ({\rho_s}^\kappa)^m(1) \partial_u G_{s}(1)- ({\rho_s}^\kappa)^m(0) \partial_u G_{s}(0) \big\} \, ds\\
		&+ \kappa\int^{t}_{0} ((\rho_{t}^\kappa(0))^m -(\rho_{t}^\kappa(1) )^m)\big(G_{s}(0)- G_{s}(1)\big) \, ds=0,
		\end{split}   
		\end{equation}	
		for all $t\in [0,T]$ and all functions $G \in C^{1,2} ([0,T]\times \bb T\backslash\{0\} )$. 
	\end{enumerate}
\end{definition}

\begin{remark}
	\label{Def.Neumann2}
If in the porous medium equation with Robin boundary conditions we take $\kappa=0$, we get 	the porous medium equation with Neumann boundary conditions. 
\end{remark}

\begin{definition}
	\label{Def. Dirichlet2}
	Let $T>0$ and $g:\bb T\rightarrow [0,1]$ a measurable function. We say that $\rho:[0,T]\times\bb T\backslash\{0\}\to [0,1]$ is a weak solution of the porous medium equation with  {periodic} boundary conditions
	\begin{equation}\label{eq:Dirichlet2}
	\begin{cases}
	&\partial_{t}\rho_{t}(u)=\Delta\,  ({\rho_t})^m(u), \quad (t,u) \in (0,T]\times\bb T\backslash\{0\},\\
	&{ \rho} _{t}(0)={ \rho}_{t}(1),\quad t \in (0,T], \\
	&{ \rho}_{0}(u)= g(u), \quad u\in\bb T,
	\end{cases}
	\end{equation}
	if the following conditions hold:
	\begin{enumerate}
		\item $\rho^{m} \in L^{2}(0,T; \mathcal{H}^{1})$;
		\item $\rho$ satisfies the integral equation:
		\begin{equation}\label{eq:Dirichlet integral2}
		\begin{split}
		\langle \rho_{t} , G_{t}\rangle  \,-\langle g ,  G_{0}\rangle -& \int_0^t\langle \rho_{s}, (\partial_s G_{s} + (\rho_s)^{m-1} \Delta G_{s} ) \rangle \, ds=0
		\end{split}   
		\end{equation}
		for all $t\in [0,T]$ and all functions $G \in C^{1,2} ([0,T]\times\bb T)$;
	\end{enumerate}
\end{definition}

\begin{remark} 
	Observe that for $m=1$ the equations above become the ones considered in \cite{phase}.
	
\end{remark} 

Now we state the analogous result to our Theorem \ref{main theorem} for the equations given above.

\begin{conjecture}
	Let $g:\bb T\to [0,1]$ be a measurable function. For each $\kappa > 0$, let $\rho^\kappa:[0,T]\times\bb T\backslash\{0\}\to [0,1]$ be the unique weak solution of \eqref{eq:Robin2} with initial condition $g$.
	Then,
	\begin{equation*}
	\displaystyle \lim_{\kappa\to 0} \rho^\kappa \; = \;  \rho^0 \quad \textrm{ and } \quad \displaystyle \lim_{\kappa\to \infty} \rho^\kappa \; = \;  \rho^{\infty}
	\end{equation*}
	in $L^2([0,T]\times \bb T\backslash\{0\})$, where
	$\rho^{0}$ is the unique weak solution of \eqref{Def.Neumann2}, and $\rho^{\infty}$ is the unique weak solution of \eqref{eq:Dirichlet2}, both with initial condition $g$.
\end{conjecture}

\subsection{The porous medium model with a slow bond}\label{micro:slow_bond}

The porous medium model evolving on $\mathbb T_n$ can be defined as above, that is, it is the Markov process $\{\tilde\eta_{tn^2}:t\geq 0\}$ with state space $\tilde\Omega_n:=\{0,1\}^{\mathbb T_n} $ and  whose infinitesimal generator is given,  on functions $f: \tilde\Omega_{n} \to \mathbb{R}$, by
\begin{equation}\label{eq:gen_full}
n^2(\tilde {L}^{m}_{n}f)(\tilde \eta) = n^2(\tilde L^{m}_{P}f)(\tilde\eta) + n^{a}(\tilde L_{S}f)(\tilde\eta),\end{equation}
where 
\begin{equation}\label{generator PMM}
(\tilde{L}^{m}_{P}f)(\tilde\eta) = \sum_{x\in\mathbb T_n}\xi_x^nc^{m}_{x,x+1}(\tilde\eta)\{a_{x,x+1}(\eta)+a_{x+1,x}(\tilde\eta)\}[f(\tilde\eta^{x,x+1})-f(\tilde\eta)],
\end{equation}
and 
\begin{equation}\label{genrator SSEP}
(\tilde L_{S}f)(\tilde\eta) = \sum_{x\in\mathbb T_n}\{a_{x,x+1}(\tilde\eta)+a_{x+1,x}(\tilde\eta)\}[f(\tilde\eta^{x,x+1})-f(\tilde\eta)],
\end{equation}
with  $c^{m}_{x,x+1}(\tilde\eta) $ as in \eqref{PMM rate}, 
	$a_{x,y}(\tilde \eta)$ as in \eqref{SSEP rate} and $\xi_x^n=\textbf{1}_{x\neq n-1}+\frac{\kappa}{n^\theta}\textbf{1}_{x=n-1}$.
	By defining the empirical measure as in \eqref{eq:emp_mea} but replacing $\Sigma_n$ by $\mathbb T_n$, we have  the following result, whose proof we do not present here since it is very similar to the proof of Theorem 2.11 of \cite{BPGN}.

	\begin{conjecture}\label{hydrodynamic limit slow bond}
	Let $g:\mathbb T\rightarrow[0,1]$ be a measurable function and $\lbrace\mu _n\rbrace_{n \in \mathbb{N}}$  a sequence of probability measures on $\Omega_n$ associated with $g$ (as in \eqref{eq:associated}).  Then, for any $t \in [0,T]$ and any $\delta>0$,
	\begin{equation*}\label{limHidreform}
	\lim_{n \to \infty}\mathbb{P}_{\mu_n}\Bigg( \tilde\eta_{\cdot} \in \mathcal{D}([0,T], \tilde\Omega_n): \Bigg| \< \pi^{n}_{t},G\> - \int_{\mathbb T} G(u)\rho_t(u)\,du  \,  \Bigg| > \delta \Bigg)=0,
	\end{equation*}
	where
	\begin{itemize}
		\item[$\bullet$] $\rho_{t}(\cdot)$ is a weak solution of \eqref{eq:Dirichlet2}, for \textcolor{black}{$0 \leq \theta<1$};
		\item[$\bullet$] $\rho_{t}(\cdot)$ is a weak solution of (\ref{eq:Robin2}), for $\theta =1$;
		\item[$\bullet$] $\rho_{t}(\cdot)$ is a \textcolor{black}{weak solution of (\ref{eq:Robin2}) with the choice $\kappa=0$,  for $\theta >1$}.
	\end{itemize}
\end{conjecture}

 Now we present an heuristic argument that justifies the relationship between the porous medium model described above and its hydrodynamic equations.
 From Dynkin's formula, see, for example, Lemma A1.5.1 of \cite{KL},  we know that  for a function $G\in C^{1,2}([0,T]\times\mathbb T\setminus{\{0\}})$
\begin{equation}\label{dynkin}
M^{n}_{t}(G) = \langle \pi^{n}_{t},G_t \rangle - \langle \pi^{n}_{0},G_0 \rangle - \int_{0}^{t} (\partial_{s} + n^{2}L_{P} + n^{a}L_{S} ) \langle \pi^{n}_{s},G_s \rangle \,ds
\end{equation}
is a martingale with respect to the natural filtration $\{\mathcal{F}_{t}\}_{t \geq 0}$, where $\mathcal{F}_{t}= \{\sigma(\tilde\eta_{s}): s \leq t \}$. A simple computation shows that  denoting by $j_{x,x+1}(\tilde\eta)$ the instantaneous current associated to the bond $\{ x,x+1\}$, we have that 
\begin{equation}\label{eq:current_bond}
j_{x,x+1}(\tilde\eta) =\xi_x\Big( \tau_{x}h^m(\tilde\eta)-\tau_{x+1}h^m(\tilde\eta)\Big), \;\; \text{for} \;\; x\in \mathbb T_n , 
\end{equation}
where 
$\tau_xh^m(\tilde \eta) $ is defined in \eqref{shift}.
A simple computation shows that  $n^2L_n\< \pi_s^{n},G \>$ is given by
\begin{equation}\label{action_generator}
\begin{split}
 \frac{1}{n}\sum_{x\neq 0,n-1}\Delta_n G\left(\tfrac xn\right)\tau_{x}h(\tilde \eta_{sn^{2}})
+&  \nabla^{+}_nG(0)\tau_0h(\tilde\eta_{sn^2}) - \frac{\kappa}{n^\theta}\nabla^{+}_nG(\tfrac{n-1}{n}) \tau_{0}h(\tilde\eta_{sn^2})\\
 +&\frac{\kappa}{n^\theta}\nabla_n^+ G\left(\tfrac {n-1}{n}\right) \tau_{-1}h^m(\tilde\eta_{sn^2}) -\nabla_n^+G\left(\tfrac{n-2}{n}\right)\tau_{-1}h^m(\tilde\eta_{sn^2}),
\end{split}
\end{equation}
Now, we explain just the case $\theta=1$ for which we note that $G \in C^{1,2} ([0,T]\times\bb T\setminus{\{0\}})$. The remaining cases are similar and easier. From last computations we get that 
\begin{equation}\label{Complete Dynkin}
\begin{split}
 M^{n}_{t}(G) =& \langle \pi^{n}_{t},G \rangle - \langle \pi^{n}_{0},G \rangle  - \int_{0}^{t} \frac{1}{n}\sum_{x\neq 0,n-1}\Delta_n G\left(\tfrac xn\right)\tau_{x}h(\tilde \eta_{sn^{2}}) \, ds  \\
&+\int_{0}^{t} \Big\{ \nabla^{+}_nG(0)- \kappa (G(\tfrac{0}{n})-G(\tfrac{n-1}{n}))
\}  \tau_{0}h(\tilde\eta_{sn^2})\,ds \\
&+\int_{0}^{t}\{\kappa (G(\tfrac{0}{n})-G\left(\tfrac {n-1}{n}\right))  -\nabla_n^+G\left(\tfrac{n-2}{n}\right)\} \tau_{-1}h^m(\tilde\eta_{sn^2}) \,ds.
\end{split}
\end{equation}

Following the same reasoning as below \eqref{heuristics}, and discarding the terms coming from $\tilde L_S$, we see that the bulk term, namely the rightmost term on the first line of the previous display, will give rise to $\int_{0}^t\int_{\mathbb T}\Delta G(u)(\rho_s(u))^m du\, ds.$
For the boundary term we do the following. Let us just analyse the second line of last display, the remaining one is completely analogous. 
For simplicity take $m=2$ and note that $$\tau_0h(\tilde \eta)=\tilde\eta(n-1)\tilde\eta(0)-\tilde \eta(0)\tilde\eta(1)-\tilde\eta(n-1)\tilde\eta(1).$$ Since these terms involve the slow bond, namely  the bond $\{n-1,0\}$   we have to derive a  replacement lemma, similar to the one of Lemma  5.3 of \cite{BPGN}, which allows replacing occupation variables $\tilde\eta(x)$ by $\tilde\eta^\epsilon n(x)$ but the averages should avoid the slow bond. This means we can replace the time integral of $\tau_0h(\tilde \eta)$  by the time integral of 
$$\overleftarrow{\tilde\eta}^{\epsilon n}(n-1)\overrightarrow{\tilde\eta}^{\epsilon n}(0)-\overrightarrow{\tilde\eta}^{\epsilon n}(0)\overrightarrow{\tilde\eta}^{\epsilon n}(1)-\overleftarrow{\tilde\eta}^{\epsilon n}(n-1)\overrightarrow{\tilde\eta}^{\epsilon n}(1)$$ and from the computations developed below \eqref{approx identity} and \eqref{lebesgue diff} we are able to replace, for $n$ sufficiently big, the  second line of \eqref{Complete Dynkin}, by
\begin{equation}
\int_{0}^{t} \Big\{ \partial_uG(0)- \kappa (G(0)-G(1))
\Big\} (\rho_s(0))^2\,ds.
\end{equation}

Putting all this together we see that we should obtain the notion of  weak solution of  Definition \ref{Def. Robin2}.

These arguments give an intuition why the hydrodynamic equations associated to this model should be the equations presented in the previous subsection. The derivation of the energy estimates as stated in the first part of the article is subject to future work. The main difficulty we will face when deriving the energy estimates for this model is the fact that the form of the energy functional will be trickier. We leave this for a future work. 

	\appendix
	\section{Analysis Tools}
This appendix is dedicated to the proof of some results from analysis that we  used along the article.

\begin{lemma}\label{dense}
	The set $C^{0,1}([0,T]\times[0,1])$ is a dense subset of $L^{2}_{\kappa,\xi}([0,T]\times[0,1])$.
\end{lemma}
\begin{proof}
	The proof presented here is adapted from Proposition $A.1$ of \cite{phase}. Let $H \in L^{2}_{\kappa,\xi}([0,T]\times[0,1])$. Recall that $ L^{2}_{\kappa,\xi}([0,T]\times[0,1]) \subset L^{2}([0,T]\times[0,1])$, thus $H\in L^{2}([0,T]\times[0,1])$, $H(\cdot,0)\in L^{2}([0,T])$ and $H(\cdot,1)\in L^{2}([0,T])$. Since $C^{0,1}([0,T]\times[0,1])$ is a dense subset of $L^{2}([0,T]\times[0,1])$, consider a sequence of functions $\{H_n\}_{n\in\mathbb{N}}$ such that for each $n\in\mathbb{N}$, $H_n \in C^{0,1}([0,T]\times[0,1])$ with compact support in $[0,T]\times(0,1)$ converging in $L^{2}([0,T]\times[0,1])$ to $H$ as $n$ goes to infinity. Consider also the sequences $\{h_{n}\}_{n\in\mathbb{N}}$ and $\{\tilde{h}_{n}\}_{n\in\mathbb{N}}$, where for each $n\in\mathbb{N}$, $h_{n}, \tilde{h}_n \in C^{0}([0,T])$, converging in $L^{2}([0,T])$ to $H(\cdot,0$) and $H(\cdot,1)$, respectively. The idea now is to define a bump function to obtain a smooth function $G_n$ which approximates  $H$ in $L^2_{\kappa, \xi}$. We begin by introducing a  function $f\in C^{\infty}([0,1])$ defined by
	$
	f(u) =e^{-\frac{1}{u}}\textbf{1}_{(0,1]} $
	and let $g\in C^{\infty}([0,1])$ defined by 
	\begin{equation*}
	g(u) = \frac{f(u)}{f(u)+f(1-u)}, \quad \text{for} \quad u\in[0,1].
	\end{equation*}
	For each $n\in \mathbb{N}$, let $a_{n}=\frac{1}{4n}$ and $b_{n}=\frac{1}{2n}$. Then, $v_{n}:[0,1]\to[0,1]$ and $\tilde{v}_{n}:[0,1]\to [0,1]$ are $C^{\infty}$ functions such that
	\begin{equation*}
	v_{n}(u) = \begin{cases}
	1, \quad \quad \quad \quad \quad \quad \;\; 0\leq u< a_n,\\
	1-g\left(\frac{u-a_{n}}{b_{n}-a_{n}}\right), \quad a_n\leq u\leq b_n, \\
	0, \quad \quad \quad \quad \quad \quad b_n<u\leq1,
	\end{cases}
\quad 
	\tilde{v}_{n}(u) = \begin{cases}
	0, \quad \quad \quad \quad \quad \quad\quad \;\; 0\leq u< 1-b_n\\
	1-g\left(\frac{(1-a_n)-u}{b_{n}-a_{n}}\right), \quad 1-b_n\leq u\leq 1-a_n, \\
	1, \quad \quad \quad \quad \quad \quad \quad 1-a_n< u \leq 1.
	\end{cases}
	\end{equation*}
	Thus, for each $n\in \mathbb{N}$, let
	$G_{n}(t,u) := H_{n}(t,u) + h_{n}(t)v_{n}(u) + \tilde{h}_{n}(t)\tilde{v}_{n}(u). $
	Note that for each $n\in \mathbb{N}$, we have that $\|G_n - H\|^{2}_{\kappa,\xi}$ equals to
	\begin{equation*}
	\begin{split}
	&\|G{n}-H\|_{2}^{2}  + \int_{0}^{T} \frac{P_{m}^{\alpha}(\rho_{t}^{\kappa}(0))}{\kappa}\left(G_{n}(t,0)-H(t,0)\right)^{2}\,dt+\int_{0}^{T}\frac{P_{m}^{\beta}(\rho_{t}^{\kappa}(1))}{\kappa}\left(G_{n}(t,1)-H(t,1)\right)^{2}\,dt \\
	=&\; \|H_{n}-H\|_{2}^{2}\,+\,O_H(\pfrac{1}{n})\\&+\frac{1}{\kappa}\int_{0}^{T}P_{m}^{\alpha}(\rho_{t}^{\kappa}(0))\left(H(t,0)-h_{n}(t)\right)^{2}\,dt+\frac{1}{\kappa}\int_{0}^{T}P_{m}^{\beta}(\rho_{t}^{\kappa}(1))\left(H(t,1)-\tilde{h}_{n}(t)\right)^{2}\,dt
. 
	\end{split}
	\end{equation*}
	Since $\rho^\kappa_t(\cdot)\in[0,1]$ for all $t$, we have $P_{m}^{\alpha}(\rho^\kappa_t(0)), P_{m}^{\beta}(\rho^\kappa_t(1)) \leq m$  and   from the previous identity we get
	\begin{equation*}
	\|G_n - H\|^{2}_{\kappa,\xi} \leq \|H_{n}-H\|_{2}^{2} + \frac{m}{\kappa}\left(\int_{0}^{T}\left(H(t,0)-h_{n}(t)\right)^{2}\,dt + \int_{0}^{T}\left(H(t,1)-\tilde{h}_{n}(t)\right)^{2}\,dt\right)+O_H(\pfrac{1}{n}),
	\end{equation*}
	which vanishes, as $n\to \infty$, concluding the proof.
\end{proof}

\begin{lemma}\label{L2.2}
	Let $\xi^m\in L^2([0,T]\times [0,1])$ such that 
	$\mathcal{E}^{\alpha,\beta}_{m,\kappa,c}(\xi) \leq M_0<\infty$, for some $c>0$ and $M_0>0$.
	Then
		$\mathcal{T}^{\alpha,\beta}_{\xi,m}(\cdot)$ is a bounded, with respect to the norm of $L^2_{k,\xi}([0,T]\times [0,1])$, linear functional in $C^{0,1}([0,T]\times [0,1])$. 
\end{lemma}

\begin{proof}
Recalling the  Definition  \ref{enerrrrgy}, we have that
$
\mathcal{T}_{\xi,m}^{\alpha,\beta}(H) - c\<\!\<H, H\>\!\>_{\kappa,\xi}^{\alpha,\beta} \leq M_0,
$ for all $H\in C^{0,1}([0,T]\times [0,1])$. Replacing $H$ by $-H$ we get $
	-\mathcal{T}_{\xi,m}^{\alpha,\beta}(H) - c\<\!\<H, H\>\!\>_{\kappa,\xi}^{\alpha,\beta} \leq M_0$. Then $
	|\mathcal{T}_{\xi,m}^{\alpha,\beta}(H)| \leq M_0+c\<\!\<H, H\>\!\>_{\kappa,\xi}^{\alpha,\beta} $.   Thus, the norm of the linear functional $H\mapsto \mathcal{T}_{\xi,m}^{\alpha,\beta}(H) $ is bounded from above by

$$\Vert \mathcal{T}_{\xi,m}^{\alpha,\beta}\Vert:=\sup_{\at{H\in C^{0,1}([0,T]\times [0,1])}{\<\!\<H, H\>\!\>_{\kappa,\xi}^{\alpha,\beta} =1}}|\mathcal{T}_{\xi,m}^{\alpha,\beta}(H) |\leq M_0+c.$$

\end{proof}

\begin{lemma}\label{charact}
	A function $\zeta\in L^2(0,T; \mathcal{H}^1)$ if, and only if, there exists $\partial_{u} \zeta\in L^2([0,T]\times [0,1]) $ such that 	$$\dl \zeta, \partial_{u} G \dr=-\dl \partial_{u} \zeta, G \dr,$$ 
	for all $G\in C^{0,\infty}_c([0,T]\times(0,1))$.
\end{lemma}

\begin{proof} 	For $\zeta\in L^2(0,T; \mathcal{H}^1)$,
	from \eqref{double bracket}, 
	we can rewrite \eqref{sobolev norm 2} as
	\begin{equation*}
	\| \zeta\|^2_{L^{2}(0,T;\mathcal{H}^1)} = \dl \zeta,\zeta \dr + \dl \partial_{u} \zeta, \partial_{u} \zeta\dr.
	\end{equation*}
	Thus, $\partial_{u} \zeta\in L^2([0,T]\times [0,1]) $.
	By Definition \ref {Def. Sobolev space}, $\partial_{u} \zeta_t$ satisfies	$$\langle \zeta_t,\partial_u g\rangle =-\langle  \partial_{u} \zeta_t,g \rangle,$$ for all $g\in C^{\infty}_c((0,1))$ and for almost every  $t\in[0,T]$.
	For $G \in C^{0,\infty}_c([0,T]\times(0,1))$ and $t$ fixed $g=G_t\in C^{\infty}_c((0,1))$. Then above we replace $g$  by $G_t$ and we get 
	$$\langle \zeta_t,\partial_u G_t\rangle =-\langle  \partial_{u} \zeta_t,G_t \rangle,$$ 
	for almost every  $t\in[0,T]$.
	Then, integrating in time we obtain
	$$\dl \zeta, \partial_{u} G \dr=-\dl \partial_{u} \zeta, G \dr,$$ 
	for all $G\in C^{0,\infty}_c([0,T]\times(0,1))$. 
	
	On the other hand, if  
	there exists $\partial_{u} \zeta\in L^2([0,T]\times [0,1]) $ such that 	
	\begin{equation}\label{pppp}\dl \zeta, \partial_{u} G \dr=-\dl \partial_{u} \zeta, G \dr,\end{equation}
	for all $G\in C^{0,\infty}_c([0,T]\times(0,1))$, then, $\partial_{u} \zeta_t\in L^2( [0,1])$, for almost every $t\in[0,T]$.
	For all $g\in C^\infty_c((0,1))$ and $h\in C^{0}_c([0,T])$, we consider $G_t(u):=g(u)\,h(t)$, for all $u\in[0,1]$ and $t\in[0,T]$. Note that $G\in C^{0,\infty}_c([0,T]\times(0,1))$ and, by \eqref{pppp}, we have
	\begin{equation*}\int_0^T\big\{\<\zeta_t, \partial_{u} g \>+\<\partial_{u} \zeta_t, g \>\big\}\,h(t)\,dt=0,\end{equation*}
	for all $g\in C^\infty_c((0,1))$ and $h\in C^{0}_c([0,T])$. Then
	$\<\zeta_t,\partial_u g\>=-\<\partial_u\zeta_t, g\>$ for almost every $t\in[0,T]$ and all $g\in C^\infty_c((0,1))$, which implies that $\zeta_t\in \mathcal{H}^{1} \subset L^{2}([0,1])$ for almost every $t\in[0,T]$. Then it is possible to change $\zeta$ in a time null set to get $\zeta \in L^2(0,T;\mathcal H^1)$.
\end{proof}

	\begin{lemma}[Integration by parts]\label{IIP}
		Let $\zeta\in L^2(0,T;\mathcal H^1)$. Then
		\begin{equation}
		\int_{0}^{t}\left\< \partial_{u}\zeta_s, G_s \right\>\, ds= -\int_{0}^{t} \left\< \zeta_s, \partial_{u}G_s \right\>\, ds + \int_0^t\big\{\zeta_{s}(1)G_s(1) -\zeta_{s}(0)G_s(0)   \big\}\,ds, \\
		\end{equation}
		for all $G\in C^{\,0,1}([0,T]\times [0,1])$.
	\end{lemma}\begin{proof} 
	This proof can be found, for example, in Lemma 7.1 of \cite{fgn1}.
	\end{proof}

	\begin{lemma}\label{L_bound} Let  $\rho^\kappa:[0,T]\times[0,1] \to [0,1]$ be such that  $(\rho^\kappa)^m \in L^{2}(0,T; \mathcal{H}^{1})$. For all $\varepsilon>0$, $m\in\mathbb N$,	$j\in \{0,1,\dots, m-1\}$ and for almost every $s\in[0,T]$,	 we have
				\begin{equation}\label{L_bound_1}
			\big| \rho^\kappa _{s}(0)-\langle \rho^\kappa _{s},  \overrightarrow{\iota}^{j\varepsilon}_{\varepsilon}\rangle\big|\leq \varepsilon^{\sfrac{1}{2(m+1)} }\;\;+\;\;\varepsilon^{\sfrac{1}{(m+1)}\; }\pfrac{2^m}{3}m^{\sfrac{3}{2}}\;\Vert \partial_u( \rho^\kappa _{s})^m\Vert_2
			\end{equation}and
				\begin{equation*}
			\big| \rho^\kappa _{s}(1)-\langle \rho^\kappa _{s},  \overleftarrow{\iota}^{1-j\varepsilon}_{\varepsilon}\rangle\big|\leq \varepsilon^{\sfrac{1}{2(m+1)} }\;\;+\;\;\varepsilon^{\sfrac{1}{(m+1)}\; }\pfrac{2^m}{3}m^{\sfrac{3}{2}}\;\Vert \partial_u( \rho^\kappa _{s})^m\Vert_2\,.
			\end{equation*}

	\end{lemma}

\begin{remark}
	We observe that for $m=1$  the inequalities above can be refined. Indeed, in that case  since $\rho ^\kappa\in L^2(0,T;\mathcal H^1)$, then
	\begin{equation*}
	\rho^\kappa _{s}(0)-\langle \rho^\kappa _{s},  \overrightarrow{\iota}^{0}_{\varepsilon}\rangle=
	\frac{1}{\varepsilon}\int_0^\varepsilon \int _0^v\partial_z\rho^\kappa_s(z)\,dz\,dv,
	\end{equation*} for all $\varepsilon>0$ and for almost every $s\in[0,T]$.
	Using Cauchy-Schwarz inequality, we get 
	\begin{equation*}
	\big| \rho^\kappa _{s}(0)-\langle \rho^\kappa _{s},  \overrightarrow{\iota}^{0}_{\varepsilon}\rangle\big|\leq \sqrt\varepsilon\,\; \pfrac{2}{3}\;\Vert \partial_u( \rho^\kappa _{s})\Vert_2\,,
	\end{equation*}
for all $\varepsilon>0$ and for almost every $s\in[0,T]$. The same bound holds for $\big| \rho^\kappa _{s}(1)-\langle \rho^\kappa _{s},  \overleftarrow{\iota}^{1}_{\varepsilon}\rangle\big|$.
\end{remark}
\begin{proof}
	We will only prove the first inequality  on the statement of this lemma, because the proof of the other one is similar. 
	 For each $A>0$ and  $s\in[0,T]$, we define
	 $\Gamma_{s,A}=\{v\in[0,1];\, \rho^\kappa_s(0)\leq A \mbox{ and } \rho^\kappa_s(v)\leq A\, \}$. In the end, we will choose a suitable $A$. From this, we can bound $|\rho^\kappa _{s}(0)-\langle \rho^\kappa _{s},  \overrightarrow{\iota}^{j\varepsilon}_{\varepsilon}\rangle|$ from above  by
	\begin{equation*}
\Big|\frac{1}{\varepsilon}\int_{j\varepsilon}^{j\varepsilon+\varepsilon }\big(\rho^\kappa_s(0)-\rho^\kappa_s(v)\big)\,\, \Big(1_{\Gamma_{s,A}}(v)+1_{\Gamma_{s,A}^\complement}(v)\Big) \,dv\Big|\;\leq\; 2A\,+\frac{1}{\varepsilon}\int_{j\varepsilon}^{j\varepsilon+\varepsilon }\big|\rho^\kappa_s(0)-\rho^\kappa_s(v)\big|\,\, 1_{\Gamma_{s,A}^\complement}(v) \,\,dv,
	\end{equation*}  for all $\varepsilon>0$ and  $j\in \{0,1,\dots, m-1\}$.
 Now,
note that, for all $v\in \Gamma_{s,A}^\complement$, we have
		\begin{equation*}
\big|\rho^\kappa_s(0)-\rho^\kappa_s(v)\big|=\Bigg|\frac{(\rho^\kappa_s(0))^m-(\rho^\kappa_s(v))^m}{\sum_{i=0}^{m-1} (\rho^\kappa_s(0))^{m-1-i} (\rho^\kappa_s(v))^{i}}
\Bigg|\leq \frac{|(\rho^\kappa_s(0))^m-(\rho^\kappa_s(v))^m|}{A^{m-1}}
\,.
	\end{equation*} 
Moreover, since $(\rho ^\kappa)^m\in L^2(0,T;\mathcal H^1)$ and from the Cauchy-Schwarz inequality, we obtain
\begin{equation*}
\big|(\rho^\kappa_s(0))^m-(\rho^\kappa_s(v))^m\big|\,=\,\Bigg| \int _0^v\partial_z(\rho^\kappa_s(z))^m\,dz\Bigg|\;\leq \;\sqrt v\,\,\;\Vert \partial_u( \rho^\kappa _{s})^m\Vert_2\,,
\end{equation*} for almost every $s\in[0,T]$ and all $v\in[0,1]$.
Thus,
\begin{equation*}
\frac{1}{\varepsilon}\int_{j\varepsilon}^{j\varepsilon+\varepsilon }\big|\rho^\kappa_s(0)-\rho^\kappa_s(v)\big|\,\, 1_{\Gamma_{s,A}^\complement}(v) \,\,dv\leq 
\frac{1}{\varepsilon\, A^{m-1}}\;\Vert \partial_u( \rho^\kappa _{s})^m\Vert_2\int_{j\varepsilon}^{j\varepsilon+\varepsilon }\sqrt v\,\, dv\;\leq\,\;\, \frac{2}{3}m^{\sfrac{3}{2}}\;\frac{\sqrt \varepsilon}{ A^{m-1}}\;\Vert \partial_u( \rho^\kappa _{s})^m\Vert_2\,,
\end{equation*}
for almost every $s\in[0,T]$. Taking $A=\frac{1}{2}\varepsilon^{\sfrac{1}{2(m+1)}} $, we obtain the inequality stated in \eqref{L_bound_1}.
\end{proof}

	\section*{Acknowledgements}
	A.N. was supported through a grant ``L'OR\' EAL - ABC - UNESCO Para Mulheres na Ci\^encia''. R.P. thanks  FCT/Portugal for support through the project Lisbon Mathematics PhD (LisMath). This project has received funding from the European Research Council (ERC) under  the European Union's Horizon 2020 research and innovative programme (grant agreement   No 715734).

\end{document}